\RequirePackage{fix-cm}
\documentclass[smallextended]{svjour3}       % onecolumn (second format)
\smartqed  % flush right qed marks, e.g. at end of proof
\usepackage{graphicx}
\usepackage{amsmath}
\usepackage{amsbsy}
\usepackage{amssymb}
\usepackage{textcomp}
\usepackage{graphicx}
\usepackage{amsfonts}
\usepackage{longtable}
\usepackage{multirow}
\usepackage{lscape}
\usepackage{url}
\usepackage{citeref}
\usepackage{subfig}
\usepackage{longtable}
\usepackage{lscape}
\usepackage{hhline}
\usepackage[utf8]{inputenc}
\usepackage{colortbl,array} % für farbige cells
\usepackage{multirow,bigdelim}
\usepackage{booktabs}
\usepackage[ruled]{algorithm2e}
\usepackage{color, colortbl}

\newtheorem{thm}{Theorem}%[section]

\newtheorem{lem}[thm]{Lemma}

\newtheorem{exam}[thm]{Example}
\date{}

\begin{document}

\title{On efficiency of nonmonotone Armijo-type line searches}
\author{Masoud Ahookhosh \and Susan Ghaderi}

\institute{M. Ahookhosh
           \at Faculty of Mathematics, University of Vienna, Oskar-Morgenstern-Platz 1, 1090 Vienna,
            Austria\\
           \email{masoud.ahookhosh@univie.ac.at}
           \and S. Ghaderi
           \at Department of Mathematics, Faculty of Science, Razi University,
              Kermanshah, Iran \\ 
           %Department of Mathematics, Azarbaijan Shahid Madani University, Tabriz, Iran\\
           \email{susan.ghaderi23@gmail.com}                
            }
% \date{Received: date / Accepted: date}
\maketitle

%###########################################################################################################
%###########################################################################################################
\begin{abstract}
Monotonicity and nonmonotonicity play a key role in studying the global convergence and the efficiency of iterative schemes employed in the field of nonlinear optimization, where globally convergent and computationally efficient schemes are explored. This paper addresses some features of descent schemes and the motivation behind nonmonotone strategies and investigates the efficiency of an Armijo-type line search equipped with some popular nonmonotone terms. More specifically, we propose two novel nonmonotone terms, combine them into Armijo's rule and establish the global convergence of sequences generated by these schemes. Furthermore, we report extensive numerical results and comparisons indicating the performance of the nonmonotone Armijo-type line searches using the most popular search directions for solving unconstrained optimization problems. Finally, we exploit the considered nonmonotone schemes to solve an important inverse problem arising in signal and image processing.

\end{abstract}

%###########################################################################################################
%###########################################################################################################
\section{Introduction} \label{intro}
In this paper, we shall be concerned with some iterative schemes for solving the unconstrained minimization problem   
\begin{equation}\label{e.func} 
\begin{array}{ll} 
\textrm{minimize}   &  ~  f(x)\\ 
\textrm{subject to} &  ~  x\in \mathbb{R}^n .
\end{array} 
\end{equation} 
where $f:\mathbb{R}^n\rightarrow \mathbb{R}$  is a real-valued nonlinear function, which is bounded and continuously-differentiable. We suppose that first- or second-order black-box information of $f$ is available.\\\\
{\bf Motivation \& history.} Over the last five decades many iterative schemes for locally solving (\ref{e.func}) have been established according to the availability of information of the objective function $f$. Indeed, the conventional approaches are {\bf \emph{descent}} methods, also called {\bf \emph{monotone}} methods, generating a sequence of iterations such that the corresponding sequence of function values is monotonically decreasing, see \cite{Fle,NocW}. There exists a variety of descent methods that are classified in accordance with required information of the objective function in terms of computing function values and derivatives. More precisely, the availability of first- and second-order black-box information leads to two prominent classes so-called {\bf \emph{first-}} and {\bf \emph{second-order}} methods, where first-order methods only need function and gradient evaluations, and second-order methods require function and gradient and Hessian evaluations, see \cite{NesBo}.

In general, descent methods determine a descent direction $d_k$, specify a step-size $\alpha_k \in (0,1]$ by an inexact line search such as Armijo, Wolfe or Goldstein backtracking schemes, generate a new iteration by setting $x_{k  +  1} = x_k  +  \alpha_k d_k$, and repeat this scheme until a stopping criterion holds. The key features of these methods is characterized by choosing an appropriate inexact line search guaranteeing that 
{\renewcommand{\labelitemi}{$\bullet$} 
\begin{itemize} 
\item the sequence of function values is monotonically decreasing, i.e.,  $f_{k  +  1} \leq f_k$ where $f_k = f(x_k)$; 
\item the sequence $\{x_k\}$ is convergent globally meaning that the method is convergent for an arbitrary initial point $x_0$, especially when $x_0$ is far away from the minimizer.
\end{itemize}  
The first property seems natural due to the aim of minimizing the objective function, and the second feature makes the method independent on the initial point $x_0$. In particular, Armijo's line search satisfies 
\begin{equation}\label{e.armijo} 
 f(x_k  +  \alpha_k d_k) \leq f_k  +  \sigma \alpha_k g_k^T d_k,
\end{equation} 
where $g_k = \nabla f(x_k)$,  $\sigma \in (0, \frac{1}{2})$, and $\alpha_k$ is the largest $\alpha \in \{s, \rho s, \cdots\}$ with $s > 0$ and $\rho \in (0,1)$ such that (\ref{e.armijo}) holds, see \cite{Arm}. Since the direction $d_k$ is descent, i.e. $g_k^T d_k < 0$, function values satisfy the condition $f_{k  +  1} \leq f_k$ imposing the monotonicity to the sequence $\{f_k\}$ generated by this scheme. Moreover, it is globally convergent, see for example \cite{NocW}. A version of descent algorithms using Armijo's rule is outlined in the following:\\\\
\begin{algorithm}[H] \label{a.data} 
\DontPrintSemicolon
\KwIn{$x_0 \in \mathbb{R}^n$ ,  $\rho \in (0 ,  1)$ ,  $\sigma\in (0 ,  \frac{1}{2})$ ,
$s \in (0,1]$ ,  $\epsilon>0$;} 
\KwOut{$x_b$; $f_b$;} 
\Begin{$k \leftarrow 0$; ~ compute $f_0$;\; 
 \While {$\|g_k\| \geq \epsilon$}{ 
         generate a descent the direction $d_k$;\; 
         $\alpha \leftarrow s$; $\hat{x}_k \leftarrow x_k+\alpha d_k$;\; 
         \While {$f(\hat{x}_k) > f_k  +  \sigma \alpha g_k^T d_k $}{  
             $\alpha \leftarrow \rho \alpha$; 
             $\hat{x}_k \leftarrow x_k+\alpha d_k$;\; 
        }
         $x_{k+1} \leftarrow \hat{x}_k$; ~  $k \leftarrow k+1$;\; 
    }
     $x_b \leftarrow x_k, ~  f_b \leftarrow f_k$; 
}
\caption{ {\bf DATA} (descent Armijo-type algorithm)} 
\end{algorithm} 

\vspace{5mm}
Despite the advantages considered for imposing monotonicity to the sequence of function values, it causes some difficulties. We here mention two important cases:
\begin{itemize} 
\item The algorithm losses its efficiency if an iteration is trapped close to the bottom of a curved narrow valley of the objective function, where the monotonicity enforces iterations to follow the valley's floor causing very short steps or even undesired zigzagging, see for example \cite{GriLL1,Toi}. In the sequel, we will verify this fact in Examples 1 and 2;  
\item The Armijo-type backtracking line search can break down for small step-sizes because of the condition $f(x_k  +  \alpha d_k) \simeq f_k$ and rounding errors. In such a situation, the step $x_k$ may still be far from the minimizer of $f$, however, the Armijo condition cannot be verified because the function values required to be compared are indistinguishable in the floating-point arithmetic, i.e.,
\begin{equation*}
0 \simeq f(x_k  +  \alpha d_k) - f_k > \sigma \alpha g_k^T d_k 
\end{equation*}
 since $g_k^T d_k < 0 $ and $\sigma, \alpha>0$, see \cite{BorH}. 
\end{itemize}  
These disadvantages of descent methods have inspired many researchers to work on some improvements to avoid such drawbacks. In the remainder of this section, some of these developments will be reviewed.

According to the availability of first- or second-order information of $f$, the direction $d_k$ can be determined in various ways imposing different convergence theories and computational results, \cite{Fle,NocW, NesBo}. On the one hand, first-order methods only need function values and gradients leading to low memory requirement making them appropriate to solve large-scale problems. On the other hand, if second-order information is available, the classical method for solving (\ref{e.func}) is {\bf \emph{Newton}}'s method producing an excellent local convergence. More specifically, Newton's method minimizes the quadratic approximation of the objective function, where the corresponding direction is derived by solving the linear system  
\begin{equation}\label{e.newton} 
 H_k d_k =  -  g_k,
\end{equation} 
in which $H_k$ is Hessian of $f$ evaluated at the current iterate $x_k$. Indeed, if Hessian is positive definite and the dimension of the problem is not very large, Newton's method is possibly the most successful descent method for minimizing a twice continuously-differentiable function. The derivation of Newton's method implies that it converges to the stationary point of a quadratic function in one iteration. However, for general functions, it usually exhibits a quadratic convergence rate near the solution, however, there is no reason to expect that Newton's method behaves well if $x_0$ is chosen far away from the optimizer $x^*$, see \cite{OrtR}. This implies that Newton's method can be enhanced to obtain the global convergence, which is the convergence to a stationary point from an arbitrary starting point $x_0$ that may be far away from it. A globally convergent modification of Newton's method is called {\bf \emph{damped Newton}}'s method exploiting Newton's direction (\ref{e.newton}) and a line search similar to that discussed in Algorithm 1.  

The sequence produced by Algorithm 1 converges to an $\epsilon$-solution $x^*$ satisfying $\|\nabla f(x^*)\| < \epsilon$, which is by no means sufficient to guarantee that $x^*$ is a local minimizer. Indeed, it can converge to a local maximizer or a saddle point. Furthermore, if the iterate $x_k$ is trapped in the bottom of a deep narrow valley, it generates very short steps to keep the monotonicity resulting to a very slow convergence. This fact clearly means that employing a monotone line search to ensure the global convergence can ruin the excellent local convergence of Newton's method. We verify this fact in the next example.

%###########
\begin{exam} 
Consider two-dimensional Rosenbrock's function
\begin{equation*}
f(x) = 100(x_2-x_1^2)^2+(1-x_1)^2.
\end{equation*}
We solve the problem (\ref{e.func}) by Newton's method and damped Newton's method with the initial point $x_0 = (-\frac{1}{10}, \frac{1}{10})$. It is clear that $(1, 1)$ is the optimizer. The implementation indicates that damped Newton's method needs 15 iterations and 17 function evaluations while Newton's method needs only 7 iterations and 8 function evaluations. To study the result more precisely, we depict the contour plot of the objective function and iterations attained by these two algorithms in Figure 1. Subfigure (a) of Figure 1 shows that iterations of damped Newton's method follow the bottom of the valley in contrast to those for Newton's method that can go up and down to reach the $\epsilon$-solution with the accuracy parameter $\epsilon = 10^{-5}$. We see that Newton's method attains larger step-sizes compared with those of damped Newton's method. Subfigure (b) of Figure 1 illustrates function values versus iterations for both algorithms showing that the related function values of damped Newton's method decreases monotonically while it is fluctuated nonmonotonically for Newton's method.
\begin{figure}[h]\label{comr} 
\centering 
\subfloat[][iteration]{\includegraphics[width=7.7cm]{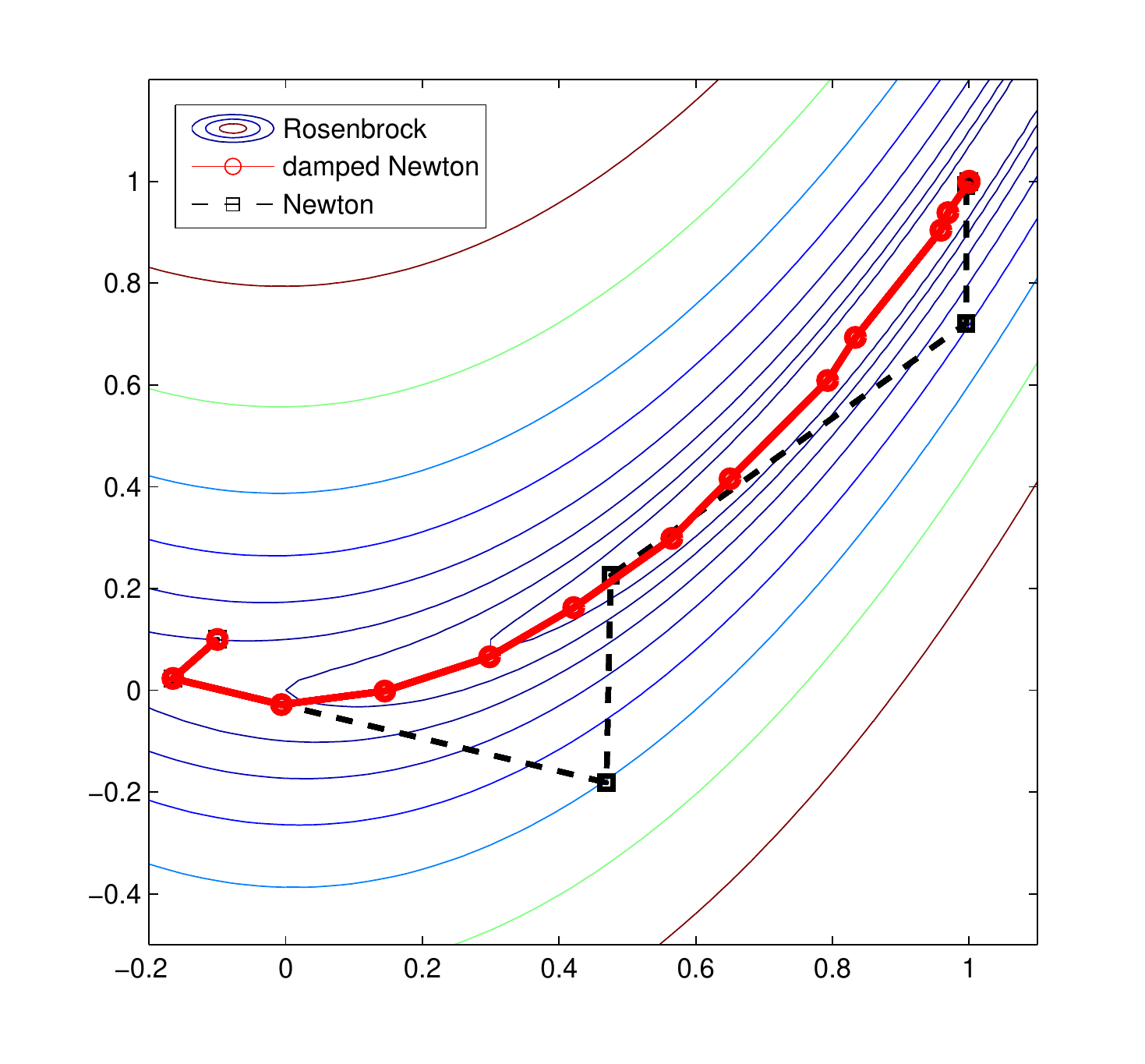}}% 
\qquad 
\subfloat[][function evaluation]{\includegraphics[width=7.45cm]{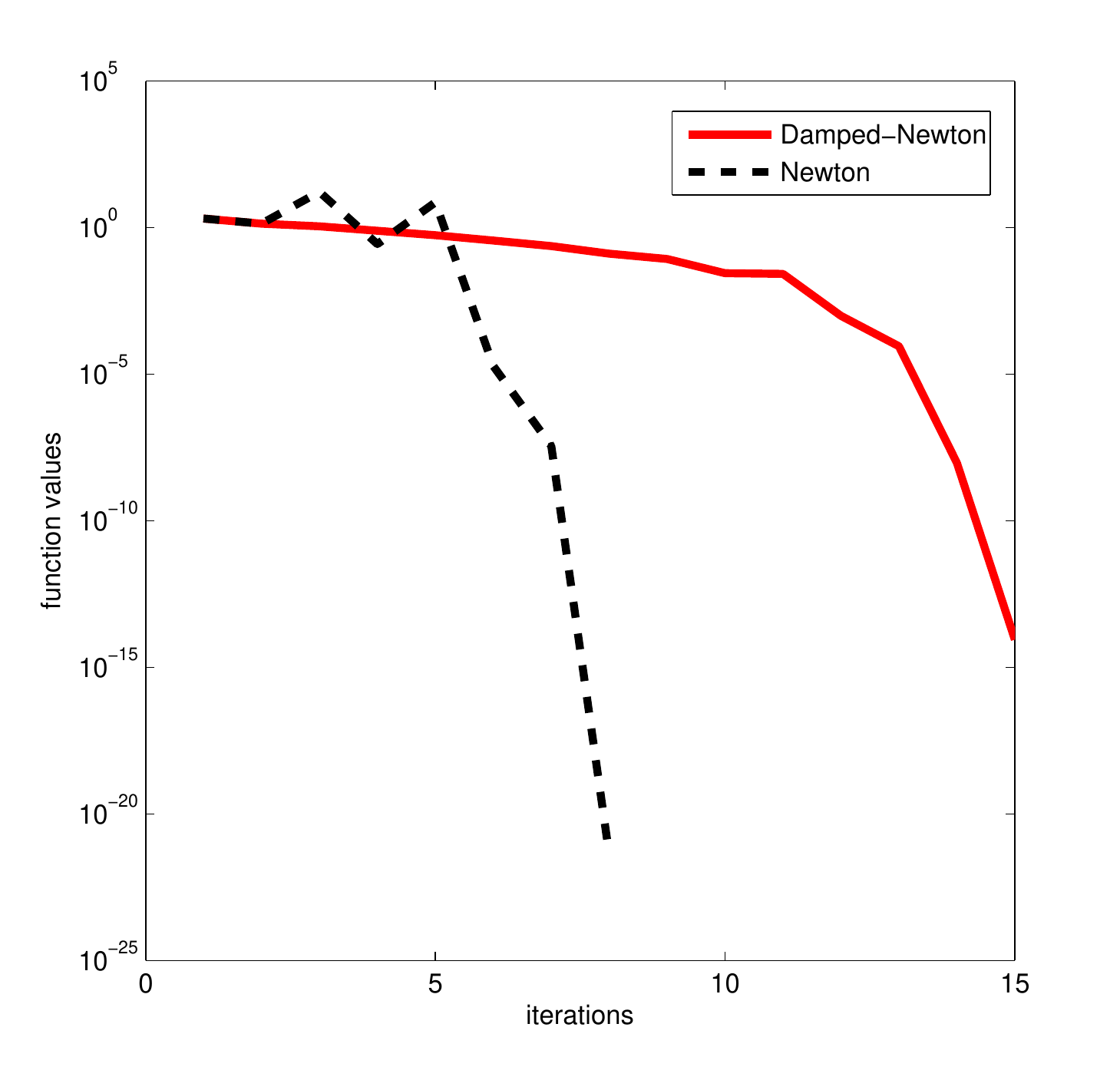}} 
\caption{A comparison between Newton and damped Newton methods: (a) the contour plot of Rosenbrock's function and iterations of Newton and damped Newton methods; (b) the function values vs. iterations.} 
\end{figure} 
\end{exam}  
 
Thanks to the increasing interest in using nonlinear optimization during the few past decades and to avoid the above-mentioned drawbacks of monotone schemes, many researchers have conducted lots of investigations on developing methods guaranteeing the global convergence and preserving the local convergence rate of descent methods at the same time. The pioneering work dating back to 1986 proposed by {\sc Grippo} et al. in \cite{GriLL1} introducing a variant of Armijo's rule using the term $f_{l(k)}$ in place of $f_k$ in (\ref{e.armijo}) defined by 
\begin{equation}\label{e.max} 
f_{l(k)} = \max_{0\leq j\leq m(k)} \{f_{k-j}\},~~  k=0,1,2,\cdots,
\end{equation} 
where $m(0)=0$ and $0\leq m(k)\leq \min\{m(k-1)+1,N\}$ with a positive constant $N \in \mathbb{N}$. The fact that $f_{l(k)} \geq f_k$ along with the convergence theory presented in \cite{GriLL1} reveal the following properties of the modified Armijo-type line search :
\begin{itemize} 
\item The sequence $\{x_k\}$ generated by the new scheme is still globally convergent to first-order stationary points of $f$; 
\item The function value at the new point $x_{k  +  1} = x_k  +  \alpha_k d_k$ can be  greater than $f_k$, so the sequence of function values $\{f_k\}$ is not monotonically decreasing, similar to the natural behaviour of the pure Newton method. However, the subsequence $\{f_{l(k)}\}$ of $\{f_k\}$ is still monotonically decreasing, see \cite{GriLL1}.
\item The right hand side of the new Armijo-type line search is greater than original Armijo's rule implying that the new method can take  bigger step-sizes compared to descent methods using original Armijo's rule (\ref{e.armijo}); 
\item In original Armijo's rule, if no step-size can be found to satisfy (\ref{e.armijo}), the algorithm usually stops by rounding errors preventing further progress. Since $f_{l(k)} \geq f_k$, it is much less possible that rounding errors break down the new nonmonotone line search.
\end{itemize} 
Since the new line search is not imposing the monotonicity to the sequence of function values, it is called {\bf \emph{nonmonotone}}. The corresponding numerical results for the nonmonotone Armijo's rule reported in \cite{GriLL1, Toi} are totally interesting. We verify the efficiency of this scheme in the subsequent example for the gradient descent direction and the Barzilai-Borwein direction described in Section 3.3.  

% ########################
\begin{exam} \label{ex.bb}
We now consider Rosenbrock's function described in Example 1 and solve (\ref{e.func}) by the gradient descent method and a version of Barzilai-Borwein method using the nonmonotone line search of {\sc Grippo} et al. in \cite{GriLL1}. In our implementation, gradient descent and Barzilai-Borwein methods require 11987 and 45 iterations and 118449 and 53 function evaluations, respectively. Subfigure (a) of Figure 2 implies that iterations of the gradient descent method zigzag at the bottom of the valley while the iterations of the Barzilai-Borwein method go up and down from both sides of the valley. Subfigure (b) implies that the Barzilai-Borwein method is substantially superior to the gradient descent method and the corresponding sequence of function values behaves nonmonotonically in contrast to that for the gradient descent method. 
 
\begin{figure}[h]\label{comr} 
\centering 
\subfloat[][iteration]{\includegraphics[width=7.7cm]{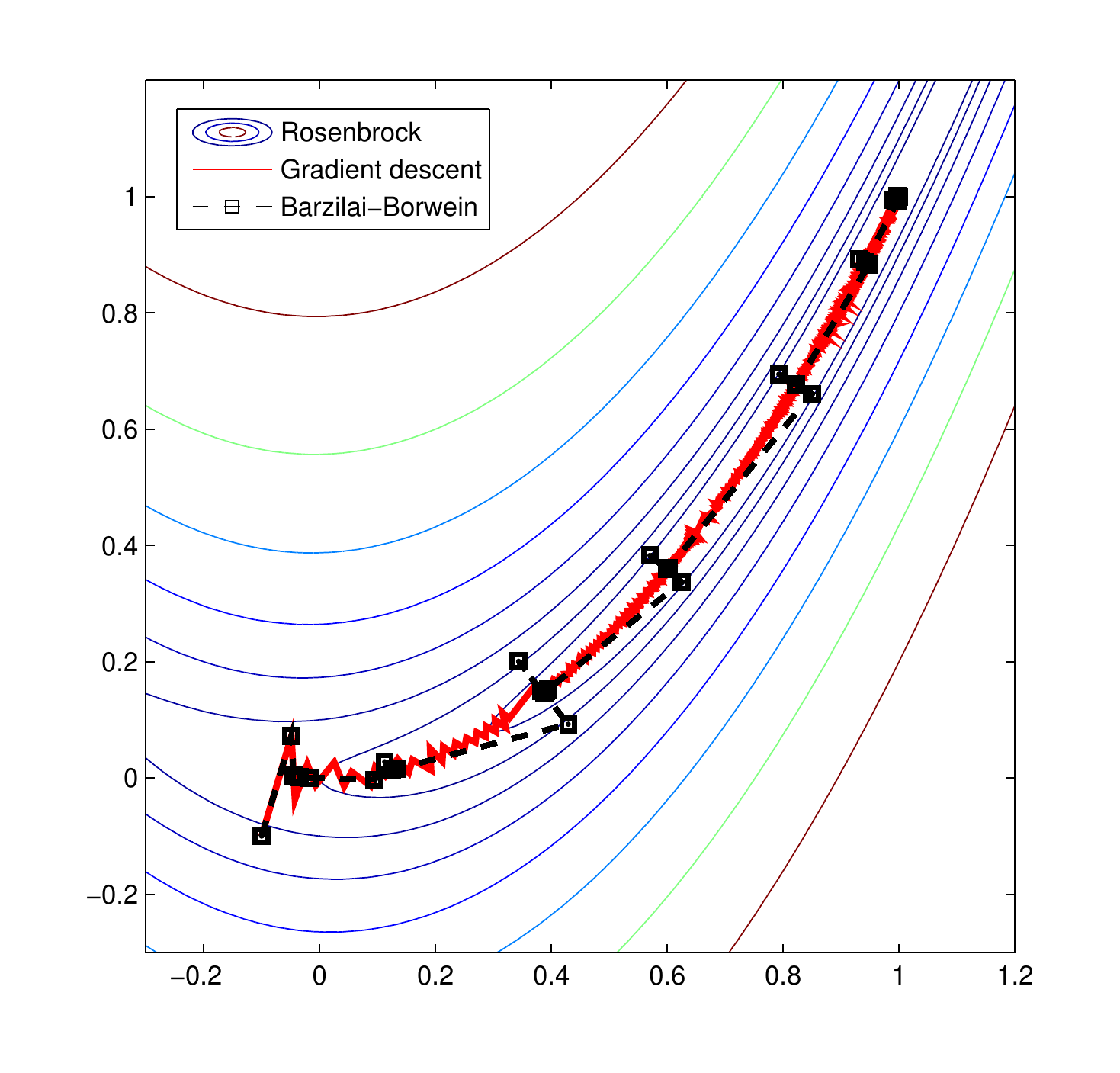}}% 
\qquad 
\subfloat[][function evaluation]{\includegraphics[width=7.7cm]{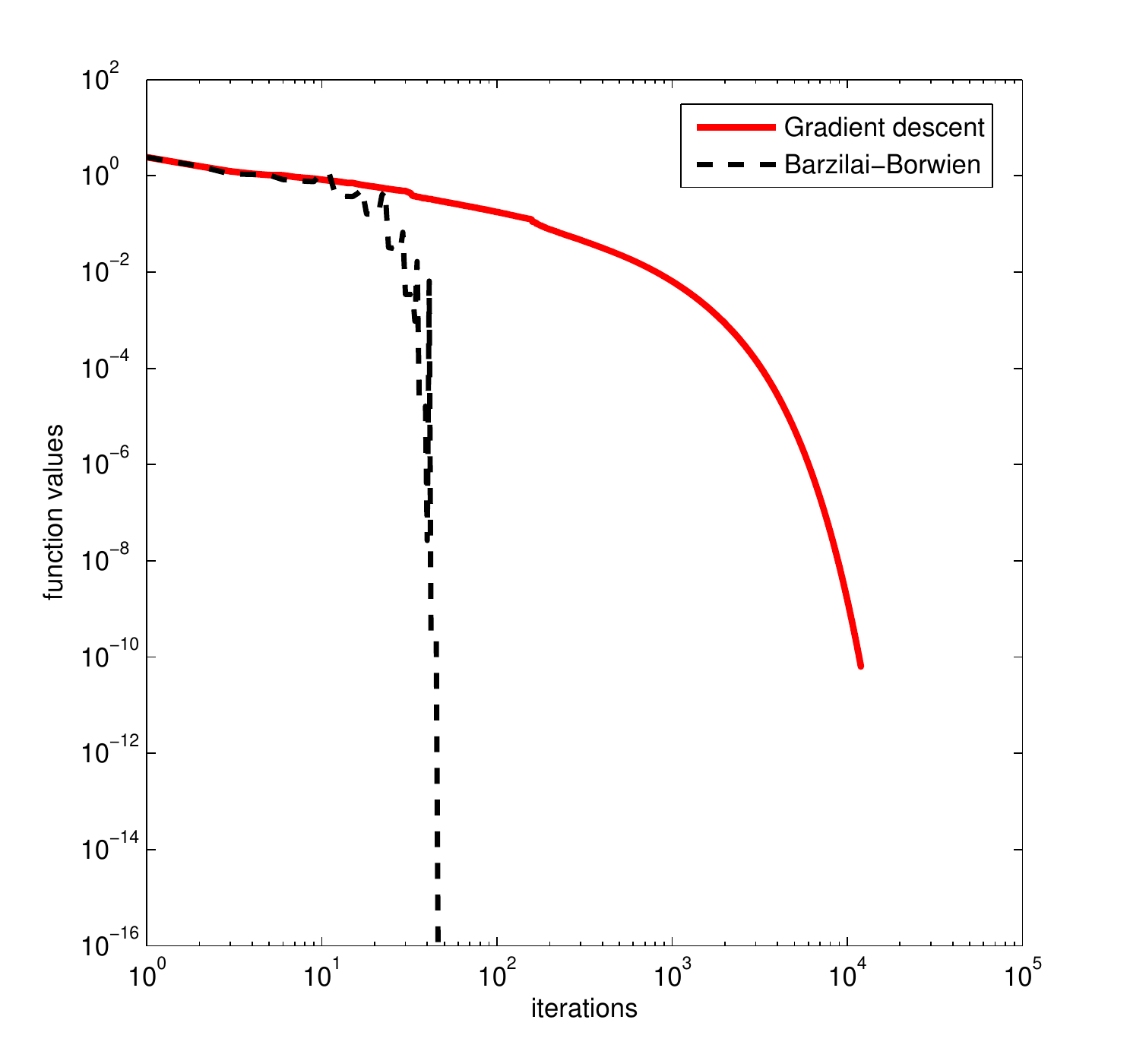}} 
\caption{A comparison between gradient descent and Barzilai-Borwein methods: (a) the contour plot of Rosenbrock's function and iterations of the considered methods; (b) the function values vs. iterations.} 
\end{figure} 
\end{exam}

Later the nonmonotone term (\ref{e.max}) was used in a more sophisticated algorithm by {\sc Grippo} et al. in \cite{GriLL3}, and they also proposed a nonmonotone truncated Newton method in \cite{GriLL2}. {\sc Toint} in \cite{Toi} conducted extensive numerical results and proposed a new nonmonotone term. For more references, see also \cite{BonPTZ,CarST,ChaPLP,Dai,PanT}. In 2004, some disadvantages of the nonmonotone term (\ref{e.max}) were discovered by {\sc Zhang} and {\sc Hager} in \cite{ZhaH}, and to avoid them the following nonmonotone term was proposed
\begin{equation}\label{e.ck} 
\begin{array}{l} 
   C_{k}=\left\{ 
   \begin{array}{ll} 
     f_k                                &~~ \mathrm{if}~ k=0; \\ 
     (\eta_{k-1}Q_{k-1}C_{k-1}+f_k)/Q_k &~~ \mathrm{if}~ k \geq 1, 
   \end{array}  \right. 
   ~~~
   Q_{k}=\left\{% 
   \begin{array}{ll} 
     1                   &~~ \mathrm{if}~ k=0; \\ 
     \eta_{k-1}Q_{k-1}+1 &~~ \mathrm{if}~ k \geq 1, \\ 
   \end{array} \right. 
\end{array} 
\end{equation} 
where $\eta_{k-1}\in[\eta_{min},\eta_{max}]$ with $\eta_{min}\in[0,1]$ and $\eta_{max}\in[\eta_{min},1]$. It is easy to see that this term is a weighted combination of all accepted function values of their scheme. Their algorithm combines the new term (\ref{e.ck}) into a Wolfe-type line search producing favourable results. Recently, another term constructed based of a convex combination of all former successful function values investigated by {\sc Mo } et al. in \cite{MoLY} and {\sc Ahookhosh} et al. in \cite{AhoAB}, where it is defined by
\begin{equation}\label{e.dk} 
D_{k}=\left\{ 
\begin{array}{ll} 
     f_0                         &~~ \mathrm{if}~ k=0; \\ 
     f_k+\eta_{k-1}(D_{k-1}-f_k) &~~ \mathrm{if}~ k \geq 1,  \\ 
\end{array}  \right.  \\ 
\end{equation} 
in which $\eta_{k-1}\in[\eta_{min},\eta_{max}]$ with $\eta_{min}\in[0,1]$ and $\eta_{max}\in[\eta_{min},1]$. Combination of this term by Armijo's rule shows a promising computational behaviour, see \cite{AhoAB}. More recently, {\sc Amini} et al. in \cite{AmiAN} proposed a new nonmonotone term relaxing the max-based term (\ref{e.max}) by an adaptive convex combination of $f_{l(k)}$ and $f_k$, which is defined by
\begin{equation}\label{e.rk} 
 R_k=\eta_k f_{l(k)}+(1-\eta_k)f_k,
\end{equation} 
in which $\eta_{k}\in[\eta_{min},\eta_{max}]$ with $\eta_{min}\in[0,1]$ and $\eta_{max}\in[\eta_{min},1]$, see also \cite{AhoA,AhoAP}. This nonmonotone term exploits an adaptive determination of the convexity parameter $\eta_k$. Indeed, it uses bigger $\eta_k$ far from the optimizer and smaller ones close to it. The reported numerical results, which was tested for Brazilai-Borwein, LBFGS and truncated Newton directions, indicate that by using an appropriate sequence $\{\eta_k\}$ the scheme behaves favourably. 

An algorithm is considered to be efficient if its computational cost to reach an $ \varepsilon $-solution optimizer is at or below some prescribed level of running time and memory usage. In general, the efficiency depend on the way in which the input data is arranged and can be measured by various measures that are generally depend on the size of the input data. In practice, there are some more factors which can affect the efficiency of an algorithm, such as requirements for accuracy and reliability. Since the most computational cost and the related running time of Armijo-type line searches are dependent on the computation of function values and gradients, we here measure the efficiency of an algorithm by counting the number of iterations ($N_i$), the number of gradient evaluations ($N_g$), the number of function evaluations ($N_f$) and some combination of them.\\\\ 
%################################## 
{\bf Contribution \& organization.} This paper addresses some nonmonotone terms and combines them into Armijo's line search (\ref{e.armijo}) for solving the unconstrained problem (\ref{e.func}). It is clear that the nonmonotone terms (\ref{e.ck}) and (\ref{e.dk}) use all previous function values, however if the initial point of an algorithm is far away from the optimizer, it does not make sense to use initial function values to construct a nonmonotone term that cannot tell us too much about the local behaviour of the objective function. In such a case, we prefer to just use the last $N$ function values to construct the new nonmonotone terms. In particular, we propose two novel nonmonotone terms, where the basic idea is to construct the new nonmonotone terms by a convex combination of the last $N$ successful function values if the current iteration counter $k$ is greater than or equal to a positive integer $N$. In case $k \leq N$, we exploit either the nonmonotone terms (\ref{e.max}) or (\ref{e.dk}). The global convergence to first-order stationary points as well as convergence rates are established under some suitable conditions. The efficiency of Armijo's rule using state-of-the-art nonmonotone terms involving new ones are evaluated by doing extensive numerical experiments on a set of unconstrained optimization test problems.

The remainder of this paper organized as follows. In Section 2, we derive two novel nonmonotone terms and establish and algorithmic framework along with its convergence analysis. Numerical results on a set of various directions are reported in Section 3. Finally, some conclusions are given in Section 4.

%###########################################################################################################
%###########################################################################################################
\section{New algorithm and its convergence} 
This section addresses two novel nonmonotone strategies and unifies them with Armijo's rule (\ref{e.armijo}) to achieve efficient schemes for solving the problem (\ref{e.func}). As discussed in Section 1, the nonmonotone term of {\sc Grippo} et al. involves some disadvantages, see \cite{AhoAB,AmiAN,ZhaH}. One claim is that the term $f_{l(k)}$ is sometimes too much big allowing jump over the optimizer, especially close to the optimizer. Furthermore, the nonmonotone terms (\ref{e.ck}) and (\ref{e.dk}) exploit all previous function values that may decrease the effect of more recent function values in these terms. In the remainder of this section, we construct two novel nonmonotone terms using a convex combination of a few past function values. 

Let fix the current iteration $k$ and the number $N \in \mathbb{N}$. The main idea is to set up a nonmonotone term determined by a convex combination of the last $k$ successful function values if $k < N$ and by a convex combination of the last $N$ successful function values if $k \geq N$. In the other words, we produce new terms using function values collected in the set 
\begin{equation} 
 \mathcal{F}_k = \left\{ 
\begin{array}{ll} 
\{f_0,f_1, \cdots, f_k\}             &~~ \mathrm{if}~ k < N;\\ 
\{f_{k-N+1},f_{k-N+2}, \cdots, f_k\} &~~ \mathrm{if}~ k \geq N,
\end{array} \right.
\end{equation} 
which should be updated in each iteration. To this end, motivated by the term (\ref{e.dk}), we construct a new term $\overline{T}_k$ using the subsequent procedure    
\begin{equation*} 
\left\{ 
\begin{array}{lll} 
\overline{T}_0     & = f_0                                                               &  ~~\mathrm{if}~k = 0;\\ 
\overline{T}_1     & = (1-\eta_{0})f_1 + \eta_{0} f_0                                    &  ~~\mathrm{if}~k = 1;\\ 
\overline{T}_2     & = (1-\eta_{1})f_2 + \eta_{1}(1-\eta_{0}) f_1 + \eta_{1}\eta_{0} f_0 &  ~~\mathrm{if}~k = 2;\\ 
\vdots             & \vdots                                                              &  ~~  \vdots\\ 
\overline{T}_{N-1} & = (1-\eta_{N - 2})f_{N - 1} + \eta_{N - 2}(1-\eta_{N - 3}) f_{N-2}  
             + \cdots + \eta_{N - 2} \cdots \eta_{0} f_0                                 &  ~~\mathrm{if}~k = N  - 1;\\ 
\overline{T}_N     & = (1-\eta_{N - 1})f_N + \eta_{N - 1}(1-\eta_{N - 2}) f_{N-1}  
            +  \cdots  +  \eta_{N  -  1} \cdots \eta_{0} f_0                             &  ~~\mathrm{if}~k = N;\\
\vdots             & \vdots                                                              &  ~~  \vdots\\ 
\overline{T}_k     & = (1-\eta_{k-1})~f_k  +  \eta_{k-1}(1-\eta_{k-2}) ~ f_{k-1}  
            +  \cdots + \eta_{k-1} \cdots \eta_{k-N} ~  f_{k-N}                             &  ~~\mathrm{if}~k \geq N;
\end{array} 
\right.
\end{equation*} 
where $\eta_i \in [0,1)$ for $i = 0, 1, \cdots, N$ are some weight parameters. Hence the new nonmonotomne term is generated by
\begin{equation}\label{e.tk} 
\overline{T}_k  : = \left\{ 
\begin{array}{ll} 
 (1-\eta_{k-1}) f_k  +  \eta_{k-1} \overline{T}_{k-1}      &~~ \mathrm{if}~ k < N;\\ 
 (1-\eta_{k-1})~f_k  +  \eta_{k-1}(1-\eta_{k-2}) ~ f_{k-1}  
+  \cdots + \eta_{k-1} \cdots \eta_{k-N} ~  f_{k-N}        &~~ \mathrm{if}~ k \geq N,
\end{array} \right.
\end{equation} 
where $ \overline{T}_0 = f_0$ and $\eta_i \in [0,1)$ for $i = 0, 1, \cdots, k$. To show that $\overline{T}_k$ is a convex combination of the function values collected in the set $\mathcal{F}_k$, it is enough to show the summation of multipliers are equal to one. For $k \geq N$, the definition for $\overline{T}_k$ implies  
\begin{equation}\label{e.conv} 
 (1-\eta_{k  -  1})  +  \eta_{k  -  1}(1-\eta_{k  -  2})  +  \cdots  +  \eta_{k  -  1} \cdots \eta_{k  -  N  -  1} (1-\eta_{k  -  N}) +  \eta_{k  -  1} \cdots \eta_{k  -  N} = 1. 
\end{equation} 
For $k < N$, a similar equality shows that a summation of the last $k$ multipliers is equal to one. Therefore, the generated term $\overline{T}_k$ is a convex combination of the elements of $\mathcal{F}_k$.

The definition of $\overline{T}_k$ clearly implies that the set $\mathcal{F}_k$ should be updated and saved in each iteration. Moreover, $N(N  +  1)/2$ multiplications are required to compute $ \overline{T}_k$. To avoid saving $\mathcal{F}_k$ and decrease the required number of multiplications, we will derive a recursive formula for (\ref{e.tk}). From the definition of $ \overline{T}_k$ for $k \geq N$, it follows that 
\begin{equation*} 
\begin{split} 
 \overline{T}_k  -  \eta_{k  -  1}  \overline{T}_{k  -  1} & = (1-\eta_{k  -  1})~f_k  +  \eta_{k  -  1}(1-\eta_{k  -  2}) ~  f_{k  -  1}  
+  \cdots  +  \eta_{k  -  1} \cdots \eta_{k  -  N} ~  f_{k  -  N}\\ 
 & -  \eta_{k  -  1} (1-\eta_{k  -  2})~f_{k  -  1}  -  \cdots  -  \eta_{k  -  1} \cdots (1-\eta_{k  -  N  -1})~  f_{k  -  N}  -  \eta_{k  -  1} \eta_{k  -  2} \cdots \eta_{k  -  N  -  1} ~  f_{k  -  N  - 1}\\ 
 &= (1-\eta_{k  -  1})~f_k  +  \eta_{k  -  1} \eta_{k  -  2} \cdots \eta_{k  -  N  -  1} ~  (f_{k  -  N}  -  f_{k  -  N  - 1})\\ 
 &= (1-\eta_{k  -  1})~f_k  +  \xi_k ~  (f_{k  -  N}  -  f_{k  -  N  - 1}), 
\end{split} 
\end{equation*} 
where $\xi_k := \eta_{k  -  1} \eta_{k  -  2} \cdots \eta_{k  -  N  -  1}$. Thus, for $k \geq N$, this equation leads to
\begin{equation}\label{e.tk0} 
 \overline{T}_k = (1-\eta_{k  -  1})~f_k  +  \eta_{k  -  1}  \overline{T}_{k  -  1}   +\xi_k~   (f_{k  -  N}  -  f_{k  -  N  -1}),
\end{equation} 
which requires to save only $f_{k-N}$ and $f_{k-N-1}$ and needs three multiplications to be updated.The definition of $\xi_k$ implies  
\begin{equation}\label{e.xik} 
 \xi_k = \eta_{k  -  1} \eta_{k  -  2} \cdots \eta_{k  -  N  -  1} = \frac{\eta_{k  -  1}}{\eta_{k  -  N  -  2}} \eta_{k  -  2} \eta_{k  -  3} \cdots \eta_{k  -  N  -  2} = \frac{\eta_{k  -  1}}{\eta_{k  -  N  -  2}} \xi_{k-1} .
\end{equation} 
If $\xi_k$ is recursively updated by (\ref{e.xik}), then (\ref{e.tk}) and (\ref{e.tk0}) imply that the new nonmonotone term can be specified by 
\begin{equation}\label{e.tk1} 
 T_k  : = \left\{ 
\begin{array}{ll} 
 f_k  +  \eta_{k  -  1} ( \overline{T}_k  -  f_k) &~~ \mathrm{if}~ k < N;\\ 
 \max ~  \left\{\overline{T}_{k}, f_k \right\}    &~~ \mathrm{if}~ k \geq N,
\end{array} \right.
\end{equation} 
where the max term guaranteeing $T_k \geq f_k$.

As discussed in Section 1, a nonmonotone method performs better whenever it uses a stronger nonmonotone term far away from the optimizer and uses weaker term close to it. This fact motivates us to consider a new version of the derived nonmonotone term by employing $f_{l(k)}$ in the case $k < N$. More precisely, the second nonmonotone term is defined by   
\begin{equation}\label{e.tk2} 
 T_k = \left\{ 
  \begin{array}{ll} 
 f_{l(k)}                                      &~~ \mathrm{if}~ k < N;\\ 
 \max ~  \left\{\overline{T}_{k}, f_k \right\} &~~ \mathrm{if}~ k \geq N,
 \end{array} \right.
\end{equation} 
where $\xi_k$ is defined by (\ref{e.xik}). It is clear that the new term uses stronger term $f_{l(k)}$ defined by (\ref{e.max}) for the first $k < N$ iterations and then employs the relaxed convex term proposed above.

We now incorporate the two novel nonmonotone terms into Armijo's line search and outline the subsequent algorithm:\\\\ 

\begin{algorithm}[H] \label{a.nmls} 
 \DontPrintSemicolon 
 \KwIn{$x_0 \in \mathbb{R}^n$, $\rho \in (0, 1)$, $\sigma\in (0, \frac{1}{2})$,
 $s \in (0,1]$, $\eta_0 \in [0, 1)$, $\epsilon>0$, $N \geq 0$;} 
 \KwOut{$x_b$; $f_b$;} 
 \Begin{ 
     $k \leftarrow 0$;~ compute $f_0$;\; 
     $T_0 \leftarrow f_0$;\; 
     \While {$\|g_k\| \geq \epsilon$}{ 
         generate a descent direction $d_k$;\; 
         $\alpha \leftarrow s$;\; 
         $\hat{x}_k \leftarrow x_k+\alpha d_k$;\; 
         \While {$f(\hat{x}_k) > T_k  +  \sigma \alpha g_k^T d_k $}{  
             $\alpha \leftarrow \rho \alpha$;\; 
             $\hat{x}_k \leftarrow x_k+\alpha d_k$;\; 
        }
         $x_{k+1} \leftarrow \hat{x}_k$;\; 
         choose $\eta_{k+1}\in [0, 1)$;\; 
         update $\xi_{k  +  1}$ by (\ref{e.xik});\; 
         update $T_{k  +  1}$ by (\ref{e.tk1}) or (\ref{e.tk2});\; 
         $k \leftarrow k+1$; 
    }
    $x_b \leftarrow x_k, ~  f_b \leftarrow f_k$; 
 }
\caption{ {\bf NMLS} (novel nonmonotone Armijo-type line search algorithm)} 
\end{algorithm} 

\vspace{8mm} 
Notice that Algorithm 2 is a simple backtracking line search producing an $\epsilon$-solution $x_b$ satisfying $\|g_b\| < \epsilon$. However, the novel nonmonotone Armijo-type line search can be employed as a part of more sophisticated line searches like Wolfe, strong Wolfe and Goldstein line searches, see \cite{NocW}. 
 
Throughout the paper, we suppose that the following classical assumptions hold in order to verify the global convergence of Algorithm 2:\\\\ 
\textbf{(H1)} The upper level set $L(x_0)=\{x\in \mathbb{R}^n  ~|~  f(x)\leq f(x_0), ~  x_0\in \mathbb{R}^n\}$ is bounded.\\ 
\textbf{(H2)} The gradient of $f$ is Lipschitz continuous over an open convex set $C$ containing $L(x_0)$, i.e., there exists a 
positive constant $L$ such that 
\begin{equation*} \|g(x)-g(y)\| \leq L \|x-y\|~~~ \forall x,  y \in C.  \end{equation*} 
\textbf{(H3)} There exist constants $0 < c_1 <1 < c_2$ such that the direction $ d_k $ satisfies the next conditions\\ 
\begin{equation}\label{e.dir} g_k^Td_k \leq  - c_1 \|g_k\|^2,~~~ \|d_k\| \leq c_2 \|g_k\|,  \end{equation} 
for all $k \in \mathbb{N} \cup \{0\}$.

Note that the assumptions (H1) and (H2) are popular assumptions frequently used to establish the global convergence of the sequence $\{x_k\}$ generated by descent methods. There are several possible ways to determine the direction $d_k$ satisfying (\ref{e.dir}). For example, the gradient descent direction and some kind of spectral gradient direction and conjugate gradient directions are satisfying these conditions, see \cite{AhoAB1,ZhaH}. Newton and quasi-Newton directions can satisfy (\ref{e.dir}) with some more assumptions, see \cite{Fle,GriLL1}. In practice, if Algorithm 2 uses Newton-type or quasi-Newton directions, and one condition of (\ref{e.dir}) is not satisfied, then one of the gradient-based directions satisfied these conditions can be used in this iteration. In view of rounding error, sometimes the directions generated by Algorithms \ref{a.nmls} may not be descent so that if $g_k^T d_k > -10^{-14}$, one can take a advantage of the gradient descent direction instead.

We now verify the global convergence of the sequence gradient $\{x_{k}\}$ generated by Algorithm \ref{a.nmls}. Thanks to the similarity of the convergence proof of the current study and those reported in \cite{AhoAB,AmiAN}, we refer most of proofs to the related results of these literatures to avoid the repetition. 

% ==========================================================================================================
\begin{lem}\label{lem1}  
Suppose that the sequence $\{x_{k}\}$ is generated by Algorithm \ref{a.nmls}, then we have 
\begin{equation}\label{e.ine}  
 f_{k}\leq T_{k}\leq f_{l(k)},
\end{equation} 
for all $k \in \mathbb{N} \cup \{0\}$.
\end{lem}  
% =========== 
\begin{proof} 
For $k \leq N$, we divide the proof into two cases: (i) $T_k$ defined by (\ref{e.tk1}); (ii) $T_k$ defined by (\ref{e.tk2}). For Case (i), Lemma 2.1 in \cite{AhoAB} implies $f_k \leq f_{l(k)}$ for  
 $i = 0, 1, \cdots k $, and since summation of multipliers in $T_k$ equal to one give the result. Case (ii) is deduced from(\ref{e.tk2}).

For $k \geq N$, if $T_k = f_k$, then the result is evident. Otherwise, (\ref{e.conv}), (\ref{e.tk1}) and the fact that $f_k \leq f_{l(k)}$ for $i = k - N + 1, \cdots, k$  imply 
 \[ 
 \begin{split} 
 f_k \leq T_k &= (1-\eta_{k  -  1})~f_k  +  \eta_{k  -  1}(1-\eta_{k  -  2}) ~  f_{k-1}  
+  \cdots  +  \eta_{k  -  1} \cdots \eta_{k  -  N} ~  f_{k  -  N}\\ 
 & \leq [(1-\eta_{k  -  1})  +  \eta_{k  -  1}(1-\eta_{k  -  2})  
+  \cdots  +  \eta_{k  -  1} \cdots \eta_{k  -  N}] f_{l(k)} = f_{l(k)},
 \end{split} 
 \] 
giving the result.
\end{proof} 

% ==========================================================================================================
\begin{lem}\label{lem2}  
Suppose that (H1) and (H2) hold, and let the sequence $\{x_{k}\}$ be generated by Algorithm \ref{a.nmls}, then we have 
\begin{equation}\label{e:3}  
 \lim_{k\rightarrow \infty} T_k=\lim_{k\rightarrow \infty}f(x_k).
\end{equation} 
\end{lem} 
% =========== 
\begin{proof} 
From (\ref{e.ine}) and Lemma 2 of \cite{AmiAN}, the result is valid.
\end{proof} 

% ==========================================================================================================
\begin{lem}\label{lem3}  
Suppose that the sequence $\{x_{k}\}$ be generated by Algorithm \ref{a.nmls}. Then, the new nonmonotone line search is well-defined. Moreover, if $ \tilde{\alpha} $ and $ \alpha $ are step-sizes generated by monotone Armijo's rule (\ref{e.armijo}) and the nonmonotone line search of Algorithm \ref{a.nmls}, respectively, then $ \tilde{\alpha} \leq \alpha $.
\end{lem}  
% ===========
\begin{proof} 
Using (\ref{e.ine}) and similar to Lemma 2.3 of \cite{AhoAB}, the results hold. 
\end{proof} 

% ==========================================================================================================
\begin{thm}\label{TH1}  
Suppose that (H1), (H2) and (H3) hold, and let the sequence $\{x_k\}$ is 
generated by Algorithm \ref{a.nmls}. Then, we have 
\begin{equation}\label{e.gloc} \lim_{k\rightarrow\infty}\|g_{k}\|= 0. \end{equation} 
Furthermore, there is not any limit point of the sequence $\{x_k\}$ 
that be a local maximum of $f(x)$.
\end{thm} 
% ===========
\begin{proof} 
By similar proofs of Lemma 3 and Theorem 1 of \cite{AmiAN}, the results are valid.
\end{proof} 

Theorem \ref{TH1} suggests that the sequence $\{x_{k}\}$ generated by Algorithm \ref{a.nmls} is globally convergent to a first-order stationary point of (\ref{e.func}). The $R$-linearly convergence of the sequence $\{x_{k}\}$ for strongly convex objective function can be proved the same as Lemma 4.1 and Theorem 4.2 in \cite{AhoAB}. Furthermore, if the algorithm exploits quasi-Newton or Newton directions, the superlinear or quadratic convergence rate also can be established similar to Theorem 4.3 and Theorem 4.4 of \cite{AhoAB} by a slight modification.  

%###########################################################################################################
%###########################################################################################################
\section{Numerical experiments and comparisons}  
This section reports some numerical results of our experiments with Algorithm 2 using the two novel nonmonotone terms to verify and assess their efficiency for solving unconstrained optimization problems. In our experiments, we first consider modified versions of damped Newton's method, Algorithm 2 with Newton's direction, and the BFGS method, Algorithm 2 with the BFGS direction, using various nonmonotone terms for solving some small-scale test problems. Afterwards, versions of Algorithm 2 equipped with the novel nonmonotone terms and LBFGS and Barzilai-Borwein directions are performed on a large set of test problems, and their performance are compared to some state-of-the-art algorithms.  

In the experiment with damped Newton's method and the BFGS method, we only consider 18 unconstrained test problems from {\sc Mor\'{e}} et al. in \cite{MorGH}, while in implementation of Algorithm 2 with LBFGS and Barzilai-Borwein directions a set of 94 standard test functions from {\sc Andrei} in \cite{And} and {\sc Mor\'{e}} et al. in \cite{MorGH} is used. In our comparisons, we employ the following algorithms:
{\renewcommand{\labelitemi}{$\bullet$} 
\begin{itemize} 
 \item NMLS-G: the nonmonotone line search of {\sc Grippo} et al.  \cite{GriLL1}; 
 \item NMLS-H: the nonmonotone line search of {\sc Zhang \& Hager} \cite{ZhaH}; 
 \item NMLS-N: the nonmonotone line search of {\sc Amini} et al.  \cite{AmiAN}; 
 \item NMLS-M: the nonmonotone line search of {\sc Ahookhosh} et al.  \cite{AhoAB1};
 \item NMLS-1: a version of Algorithm 2 using the nonmonotone term (\ref{e.tk2});
 \item NMLS-2: a version of Algorithm 2 using the nonmonotone term (\ref{e.tk1});
\end{itemize} 
All of these codes are written in MATLAB using the same subroutine, and they are tested on 2Hz core i5 processor laptop with 4GB of RAM with double precision format. The initial points are standard ones reported in \cite{And} and \cite{MorGH}. All the algorithms use the parameters $\rho = 0.5$ and $\sigma = 0.01$. For NMLS-N, NMLS-G, NMLS-1 and NMLS-2, we set $N = 10$. As discussed in \cite{ZhaH}, NMLS-H uses $\eta_k = 0.85$.  On the basis of our experiments, we update the parameter $\eta_k$ adaptively by 
\begin{equation}\label{e.eta} 
  \eta_{k}=\left\{ 
  \begin{array}{ll} 
    \eta_0/2   \vspace{4mm}    &  ~~ \mathrm{if}~ k=1; \\ 
    (\eta_{k-1}+\eta_{k-2})/2  &  ~~ \mathrm{if}~ k\geq 2,
  \end{array} \right.
\end{equation} 
for NMLS-N, NMLS-M, NMLS-1 and NMLS-2, where the parameter $\eta_0$ will be tuned to get a better performance. In our experiments, the algorithms are stopped whenever the total number of iterates exceeds $\mathsf{maxiter} = 50000$ or     
\begin{equation} \| g_k \| < \epsilon \end{equation} 
holds with the accuracy parameter $\epsilon = 10^{-5}$. We further declare an algorithm "failed" if the maximum number of iterations is reached.

To compare the results appropriately, we use the performance profiles of {\sc Dolan \& Mor\'{e}} in \cite{R7}, where the measures of performance are the number of iterations ($N_i$), function evaluations ($N_f$) and gradient evaluations ($N_g$). It is clear that in the considered algorithms the number of iterations and gradient evaluations are the same, so we only consider the performance of gradients. It is believed that computing a gradient is as costly as computing three function values, i.e., we further consider the measure $N_f+3 N_g$. In details, the performance of each code is measured by considering the ratio of its computational outcome versus the best numerical outcome of all codes. This profile offers a tool for comparing the performance of iterative processes in a statistical structure. Let $\mathcal{S}$ be a set of all algorithms and $\mathcal{P}$ be a set of test problems. For each problem $p$ and solver $s$,  $t_{p,s}$ is the computational outcome regarding to the performance index, which is used in defining the next performance ratio 
\begin{equation}\label{e.perf}
r_{p,s}=\frac{t_{p,s}}{\min \{ t_{p,s}: s\in \mathcal{S}\}}.
\end{equation} 
If an algorithm $s$ is failed to solve a problem $p$, the procedure sets $r_{p,s}=r_{failed}$, where $r_{failed}$ should be strictly larger than any performance ratio (\ref{e.perf}). For any factor $\tau$, the overall performance of a algorithm $s$ is given by 
\begin{equation*} 
\rho_s(\tau)=\frac{1}{n_p} \textrm{size}\{p \in \mathcal{P}: r_{p,s} \leq \tau\}.  
\end{equation*} 
In fact $\rho_s(\tau)$ is the probability that a performance ratio $r_{p,s}$ of the algorithm $s\in \mathcal{S}$ is within a factor $\tau \in \mathbf{R}^n$ of the best possible ratio. The function 
$\rho_s(\tau)$ is a distribution function for the performance ratio. In particular, $\rho_s(1)$ gives the probability that an algorithm $s$ wins over all other considered algorithms, and $\lim_{\tau\rightarrow 
r_{failed}}\rho_s(\tau)$ gives the probability of that algorithm $s$ 
solve all considered problems. Therefore, this performance profile can be considered as a measure of efficiency among all considered algorithms. In Figures 1-4, the x-axis shows the number $\tau$ while 
the y-axis inhibits $P(r_{p,s} \leq \tau: 1 \leq s \leq n_s)$.

% ########################################################################################################## 
\subsection{Experiments with damped Newton and BFGS} 
In this section, we report numerical results of solving the problem (\ref{e.func}) by NMLS-G, NMLS-H, NMLS-N, NMLS-M, NMLS-1 and NMLS-2 using damped Newton and BFGS directions and compare their performance. Since both damped Newton and BFGS methods require to solve a linear system of equations in each iteration, it is expected to solve large-scale problems with them. Thus we only consider 18 small-scale test problems from {\sc Mor\'{e}} \cite{MorGH} with their standard initial points. For NMLS-1 and NMLS-2, we use $\eta_0 = 0.75$. The results for damped Newton's method and the BFGS method are summarized in Tables 1 and 2, respectively.

Table 1 and 2 show that the results are comparable for the considered algorithms, however NMLS-1 and NMLS-2 perform slightly better. To see the results of implementations in details, we illustrate the results with performance profile in Figure 3 with $N_g$, $N_f$ and $N_f + 3 N_g$ as measures of performance. Subfigures (a), (c) and (e) of Figure 3 show the results of damped Newton's method, while subfigures (b), (d) and (f) of Figure 3 demonstrate the results of the BFGS method.

\begin{figure}\label{com6} 
 \centering 
 \subfloat[][$N_i$ and $N_g$ performance profile]{\includegraphics[width=7.3cm]{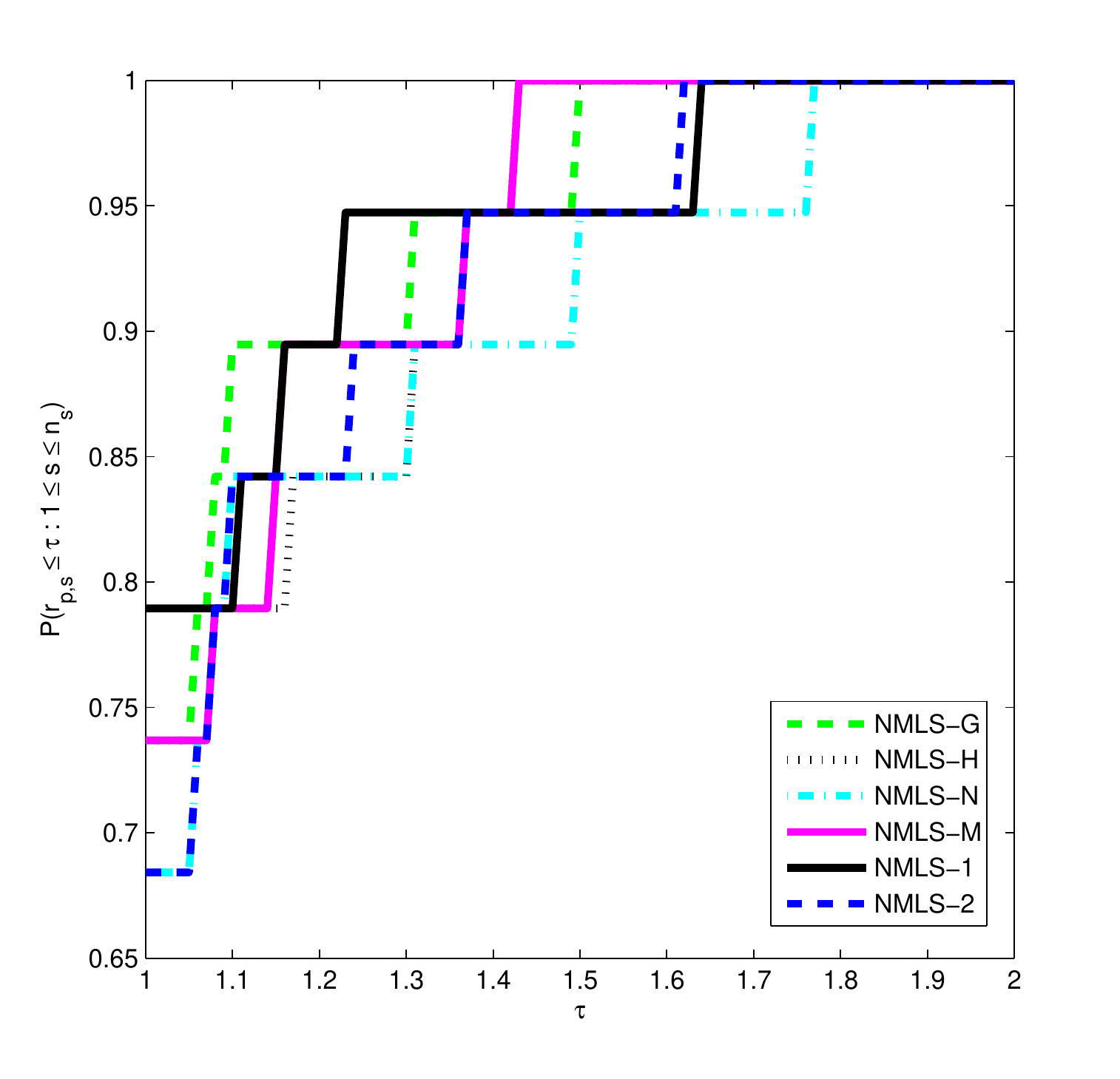}}% 
 \qquad 
 \subfloat[][$N_i$ and $N_g$ performance profile]{\includegraphics[width=7.3cm]{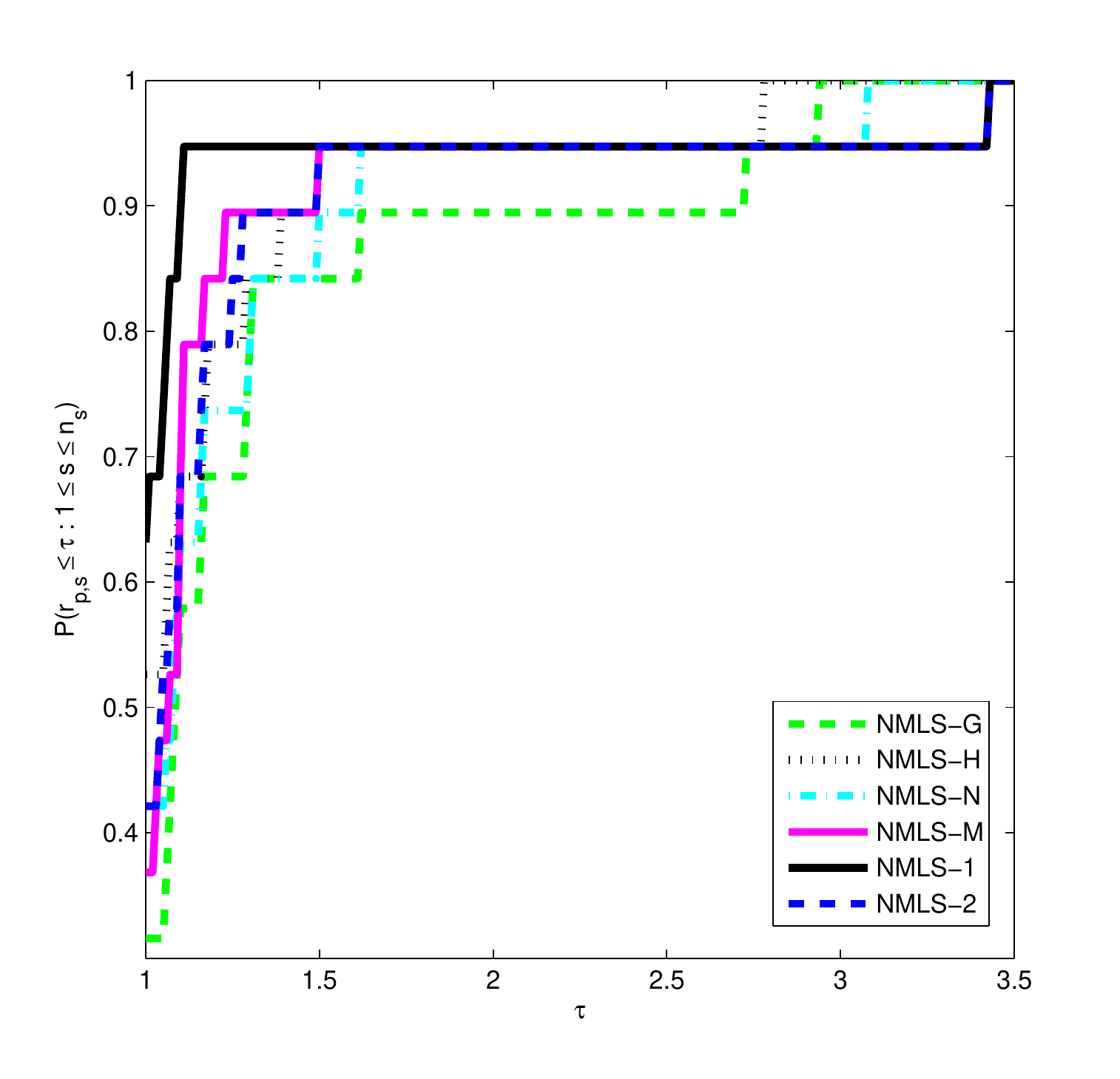}} 
 \qquad 
 \subfloat[][$N_f$ performance profile]{\includegraphics[width=7.3cm]{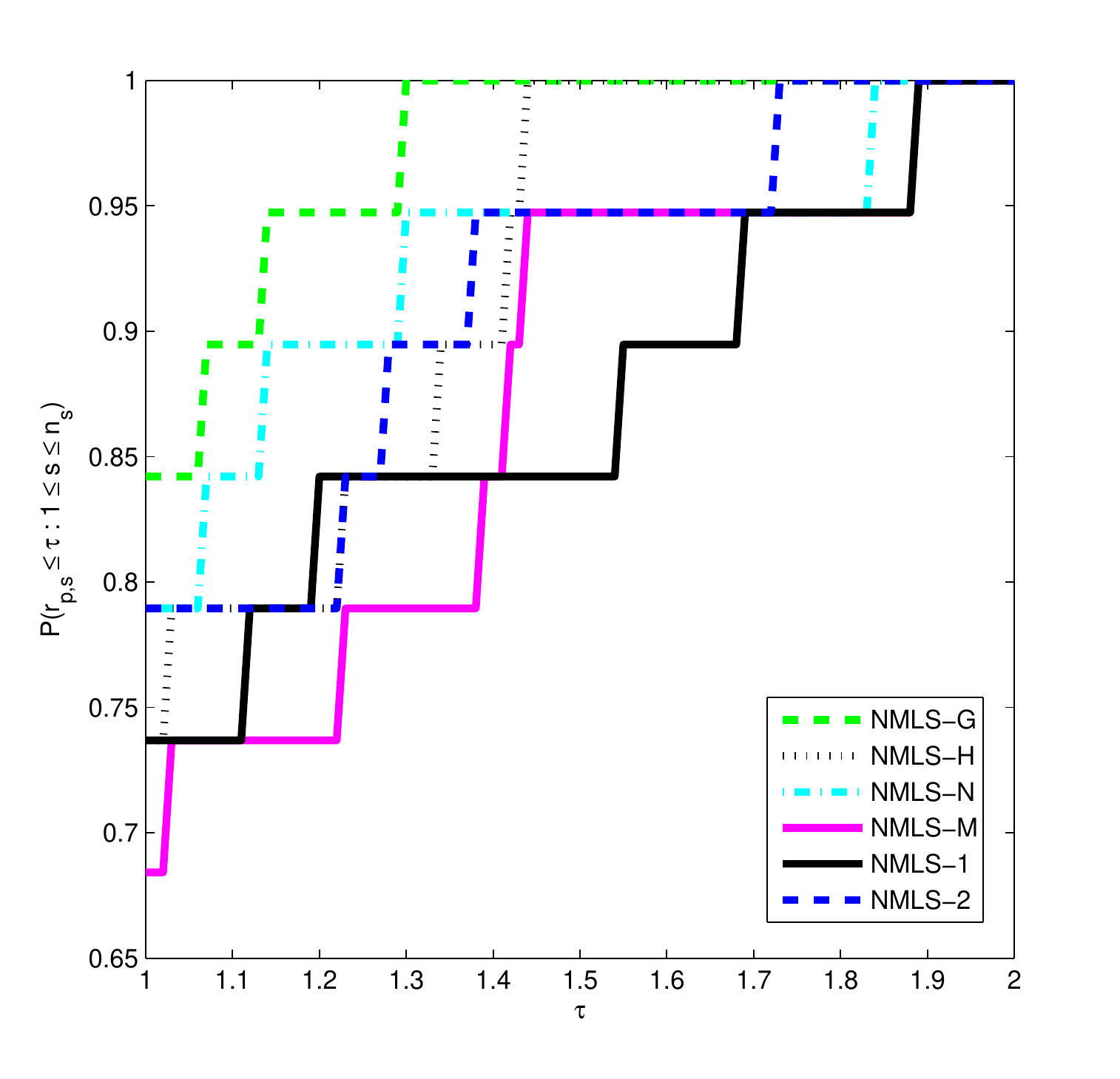}} 
 \qquad 
 \subfloat[][$N_f$ performance profile]{\includegraphics[width=7.3cm]{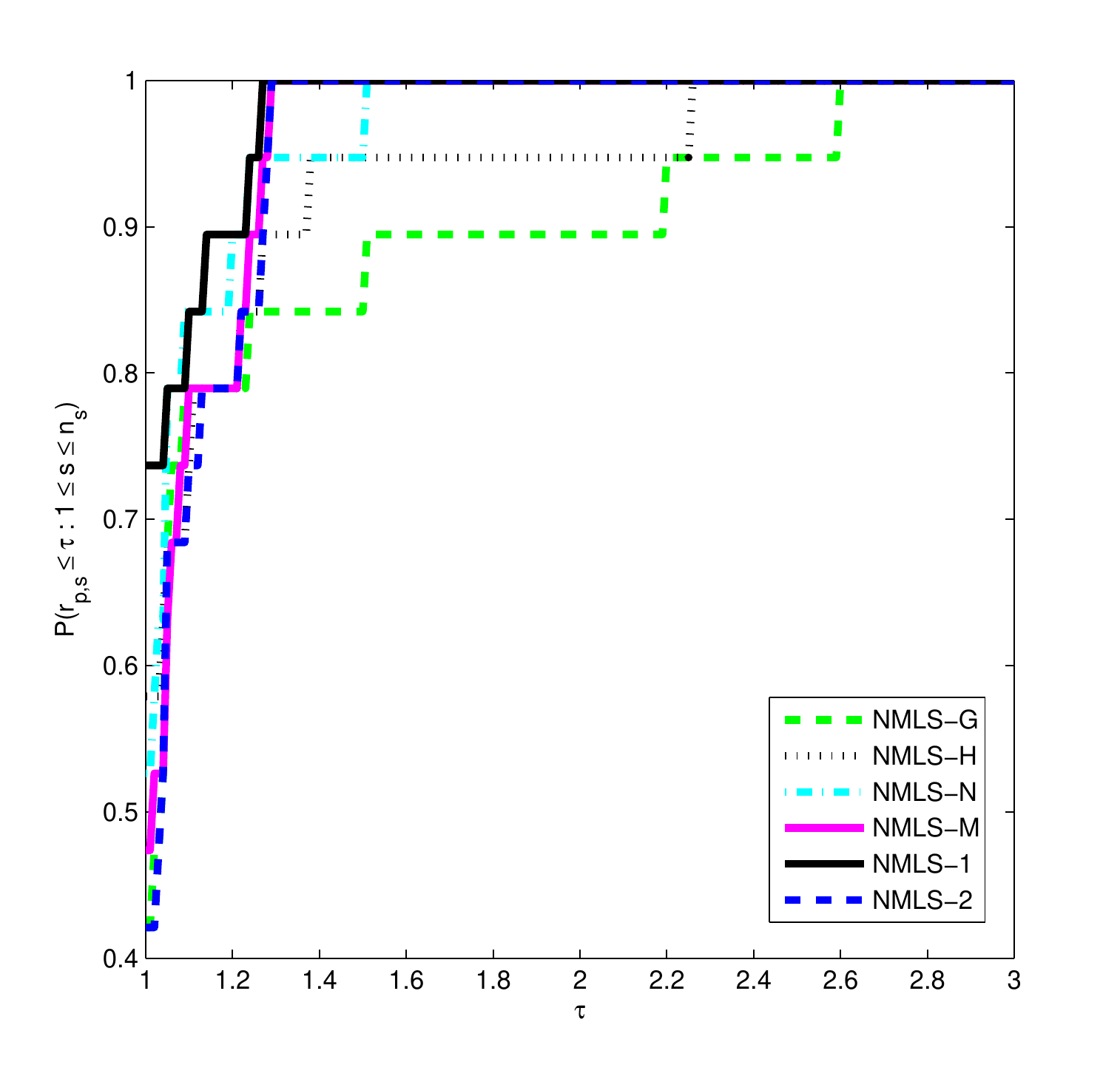}} 
 \qquad 
 \subfloat[][$N_f  +  3 N_g$ performance profile]{\includegraphics[width=7.3cm]{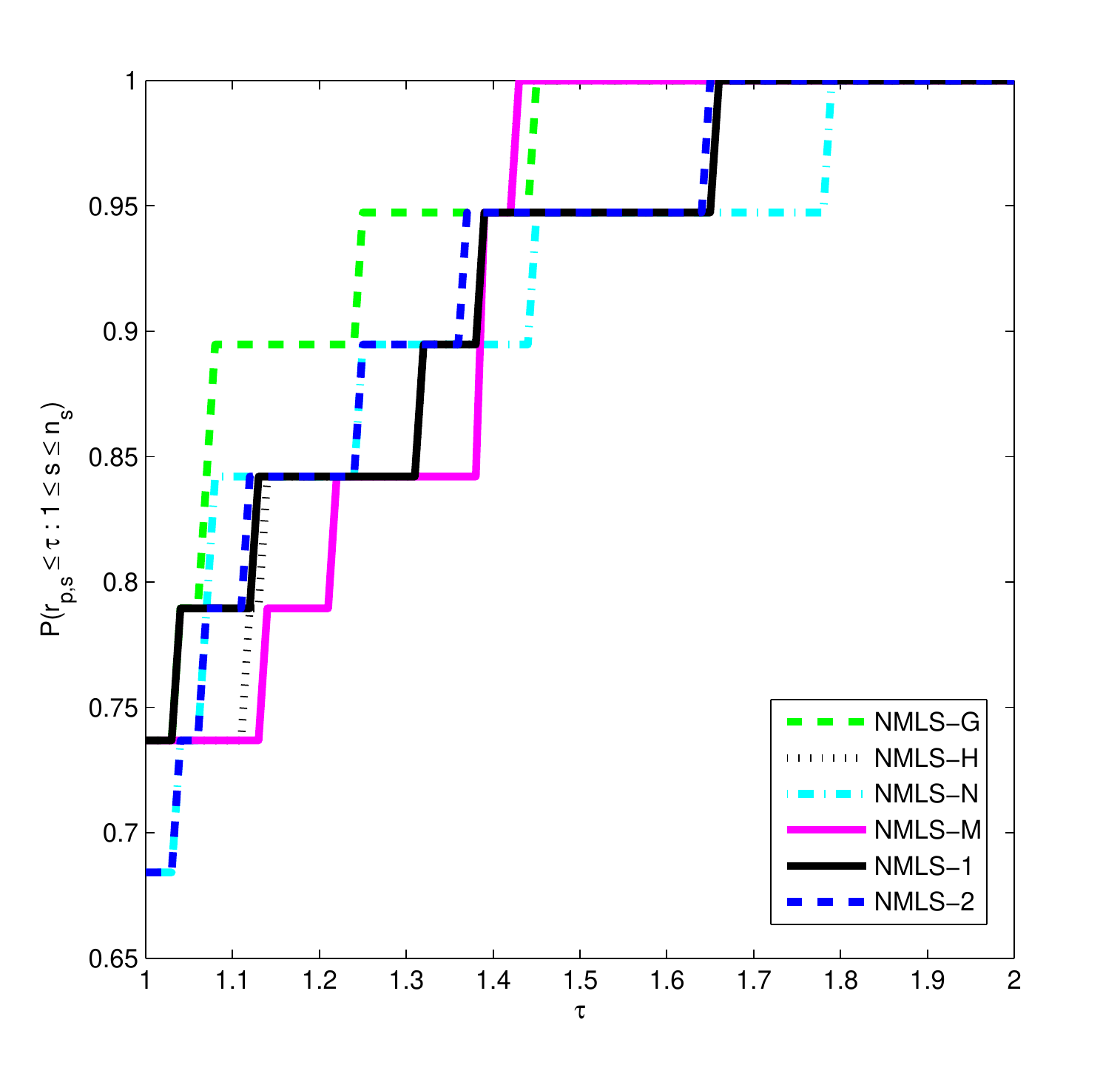}} 
 \qquad 
 \subfloat[][$N_f$ performance profile]{\includegraphics[width=7.3cm]{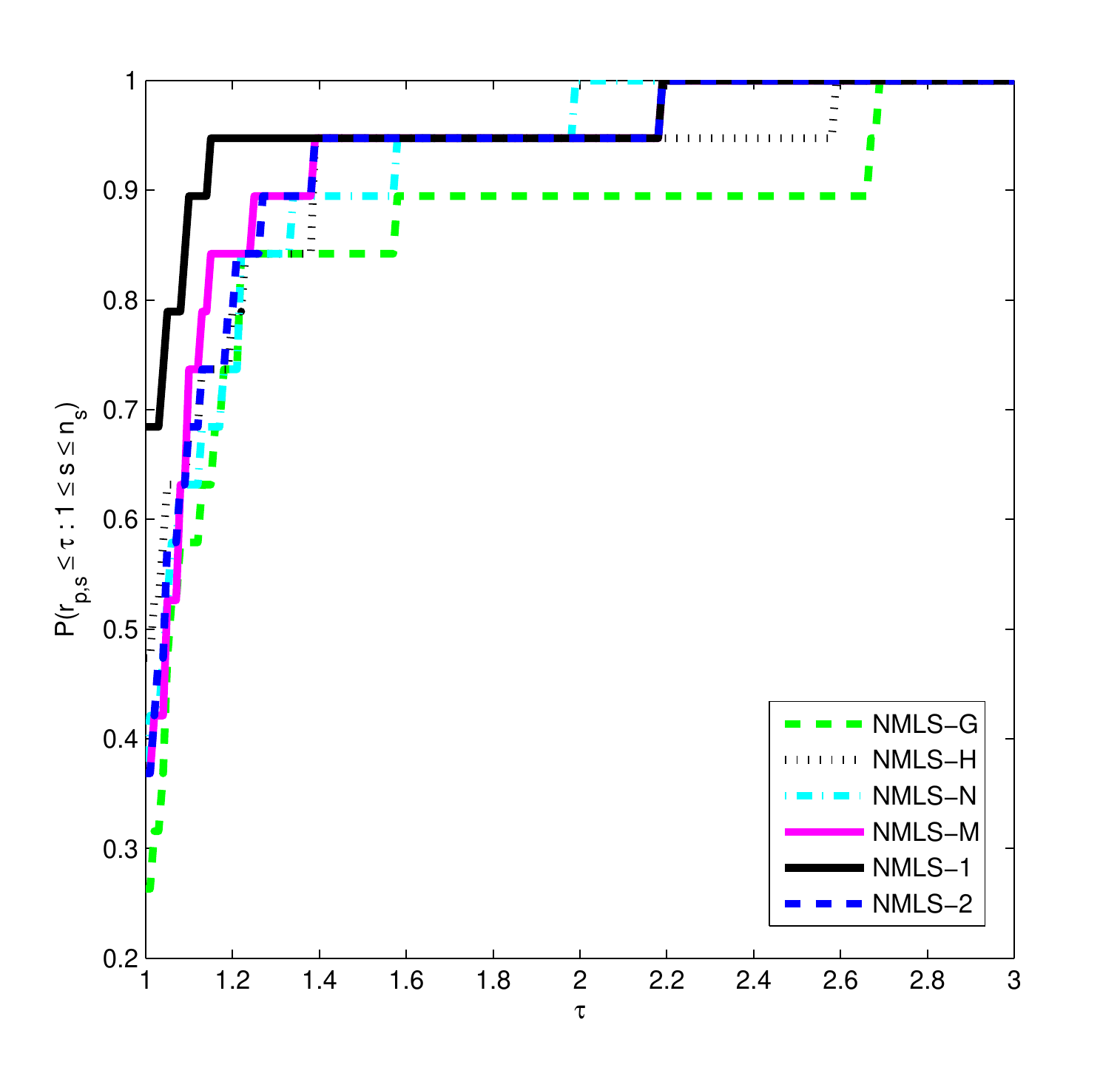}} 
 \caption{Performance profiles of all considered algorithms with the performance measures: (a) and (b) for the number of iterations ($N_i$) or gradient evaluations ($N_g$); (c) and (d) for the number of function evaluations ($N_f$); (e) and (f) for the hybrid measure $N_f  +  3 N_g$.} 
\end{figure} 

% ########################################################################################################## 
 \subsection{Experiments with LBFGS} 
In the recent decades, the interest for solving optimization problems with large number of variables is remarkably increased thanks to the dramatic emerge of big data in science and technology. This section devotes to an experiment with the considered algorithms with LBFGS which is the limited memory version of the BFGS scheme, and that is much more appropriate for solving large problems, see \cite{LiuN,Noc}. 

The LBFGS scheme calculates a search direction by $d_k = -  H_k g_k$,  where $H_k$ is an approximate of inverse Hessian determined by 
\begin{equation*}\begin{split} 
 H_{k+1} & = (V_k^T \cdots V_{k-m}^T) H_0 (V_{k-m} \cdots V_k)\\ 
 & +  \rho_{k-m} (V_k^T \cdots V_{k-m+1}^T) s_{k-m}s_{k-m}^T (V_{k-m+1} \cdots V_k)\\ 
 & +  \rho_{k-m+1} (V_k^T \cdots V_{k-m+2}^T) s_{k-m+1}s_{k-m+1}^T (V_{k-m+2} \cdots V_k)\\ 
 & \vdots \\ 
 & +  \rho_k s_k s_k^T,
\end{split}\end{equation*} 
in which $\rho_k=1/y_k^Ts_k$ and $V_k=I-\rho_k y_k s_k^T$. The scheme starts from a symmetric positive definite initial matrix $H_0$ and sets ${m}={\min}\{k,10\}$. Indeed, it does not need to save the previous approximate matrix, instead it employs only small number of former information to construct the new search direction $d_k$. This causes that the method needs much less memory compared with the original BFGS method making it suitable for solving large-scale problems. The LBFGS code is publicly available from \cite{Url1}, however, we rewrite it in MATLAB.

It is believed that nonmonotone algorithms perform better when they employ a stronger nonmonotone term far away from the optimizer and a weaker term close to it. Hence, to get the best performance of the proposed algorithms, we first conduct some test to find a better starting parameter for $\eta_0$ in the adaptive process (\ref{e.eta}). To this end, for both algorithms NMLS-1 and NMLS-2, we consider cases that the algorithms start from $\eta_0 = 0.65$, $\eta_0 = 0.75$, $\eta_0 = 0.85$ and $\eta_0 = 0.95$. The corresponding versions of algorithms NMLS-1 and MNLS-2 are denoted by NMLS-1-0.65, NMLS-1-0.75, NMLS-1-0.85, NMLS-1-0.95, NMLS-2-0.65, NMLS-2-0.75, NMLS-2-0.85 and NMLS-2-0.95, respectively. The results of this test are summarized in Figure 4.

\begin{figure}\label{com_par} 
 \centering 
 \subfloat[][$N_i$ and $N_g$ performance profile (NMLS-1)]{\includegraphics[width=7.5cm]{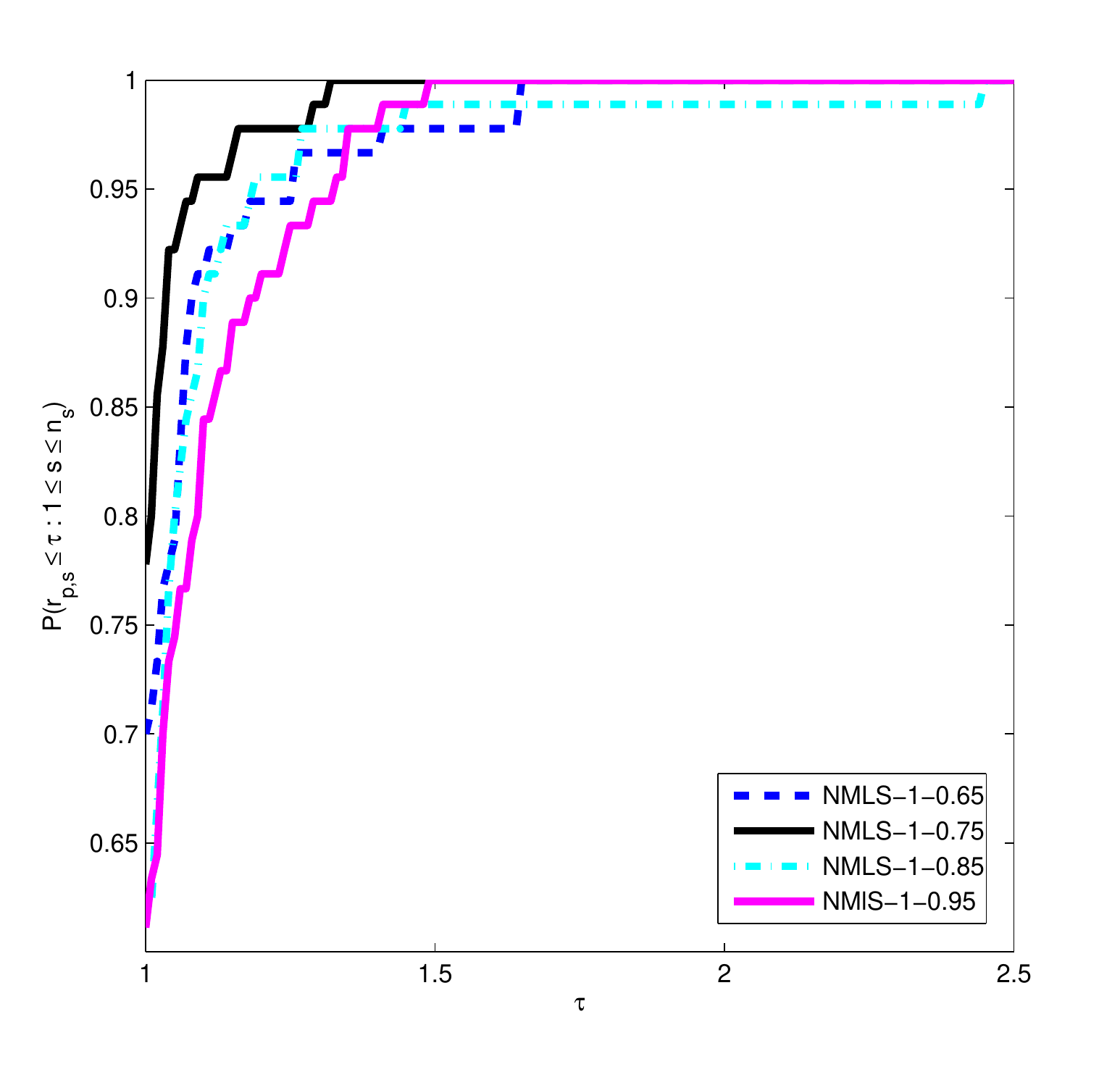}}% 
 \qquad 
 \subfloat[][$N_i$ and $N_g$ performance profile (NMLS-2)]{\includegraphics[width=7.5cm]{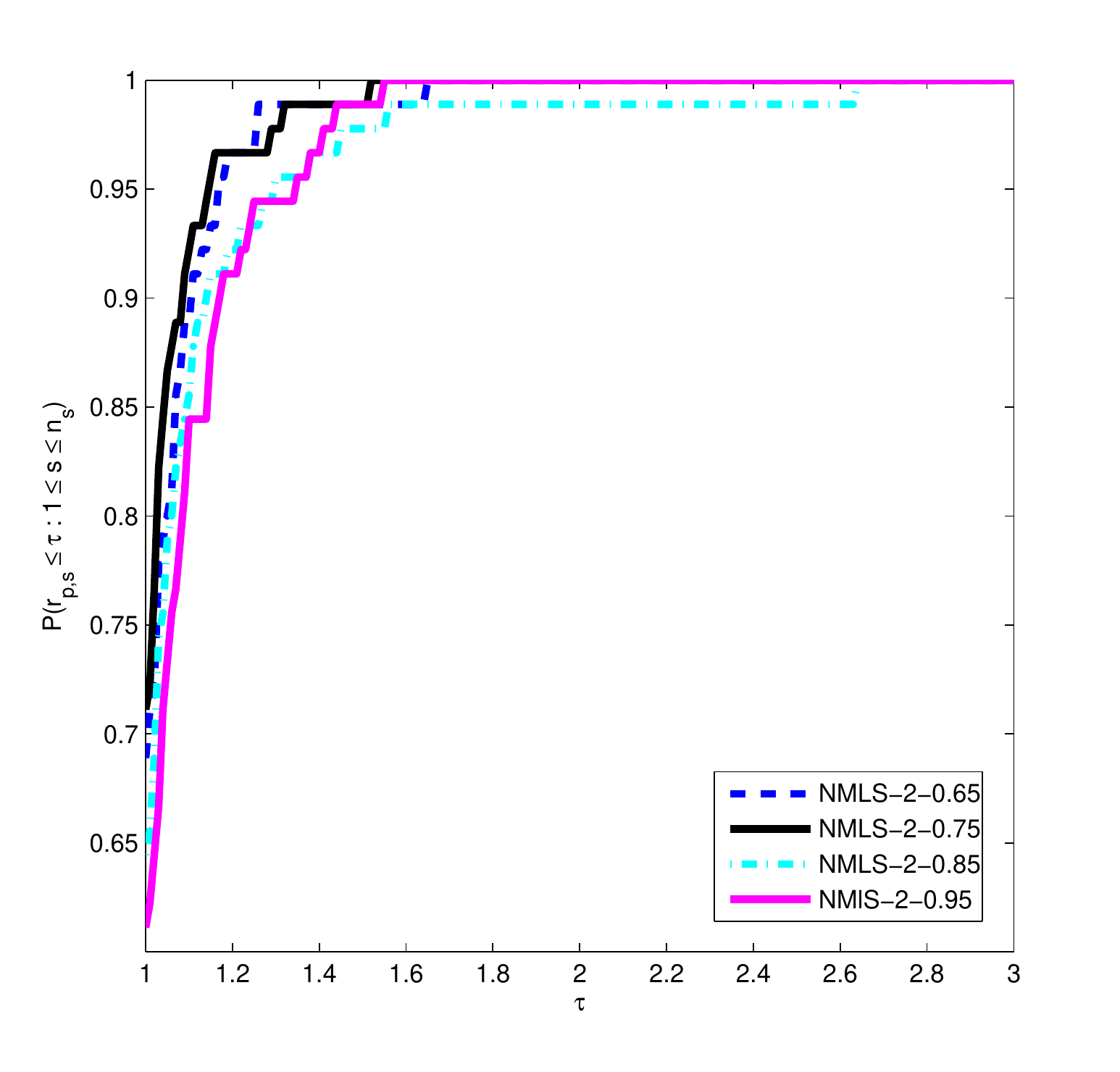}} 
 \qquad 
 \subfloat[][$N_f$ performance profile (NMLS-1)]{\includegraphics[width=7.5cm]{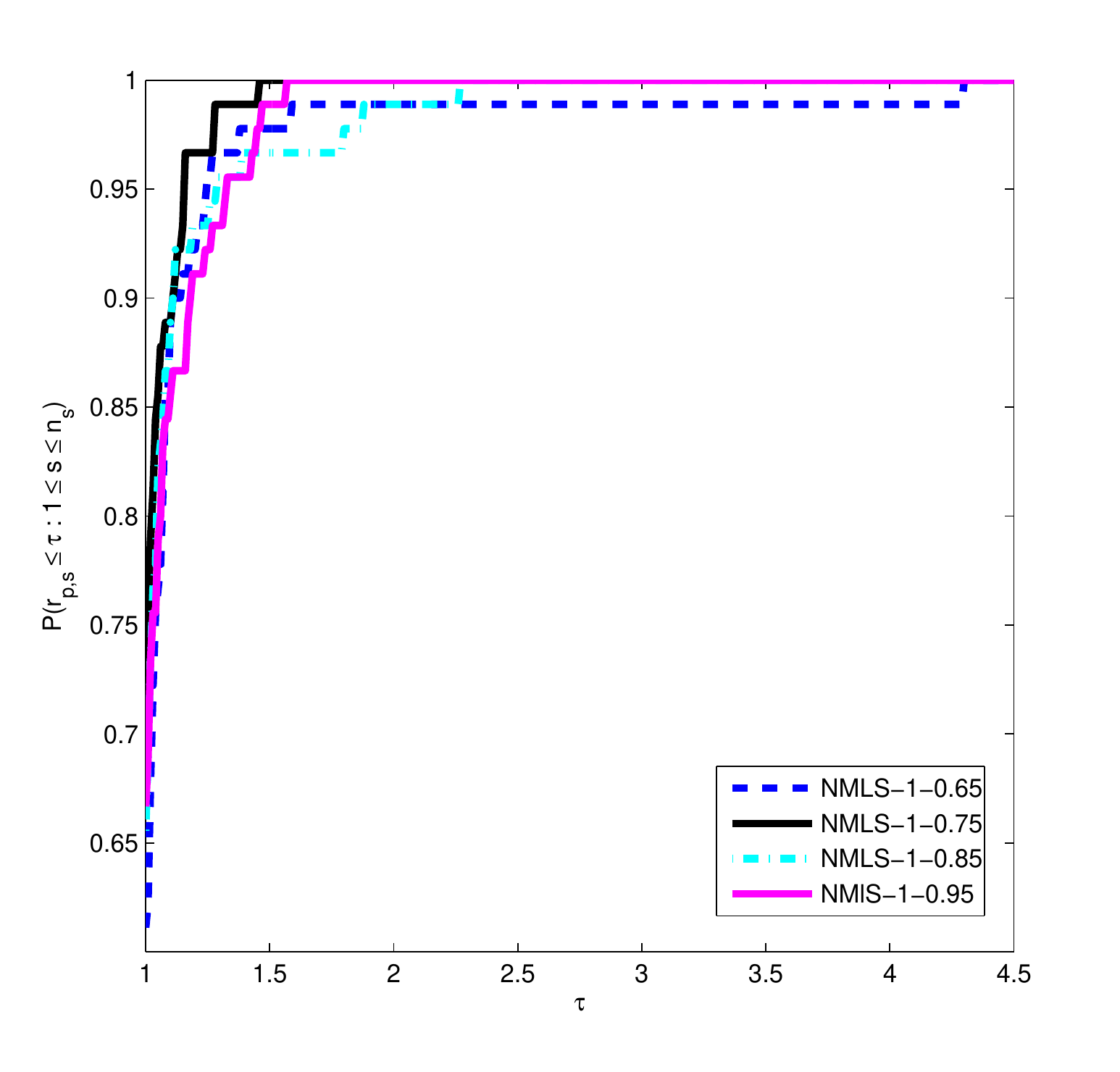}}% 
 \qquad 
 \subfloat[][$N_f$ performance profile (NMLS-2)]{\includegraphics[width=7.5cm]{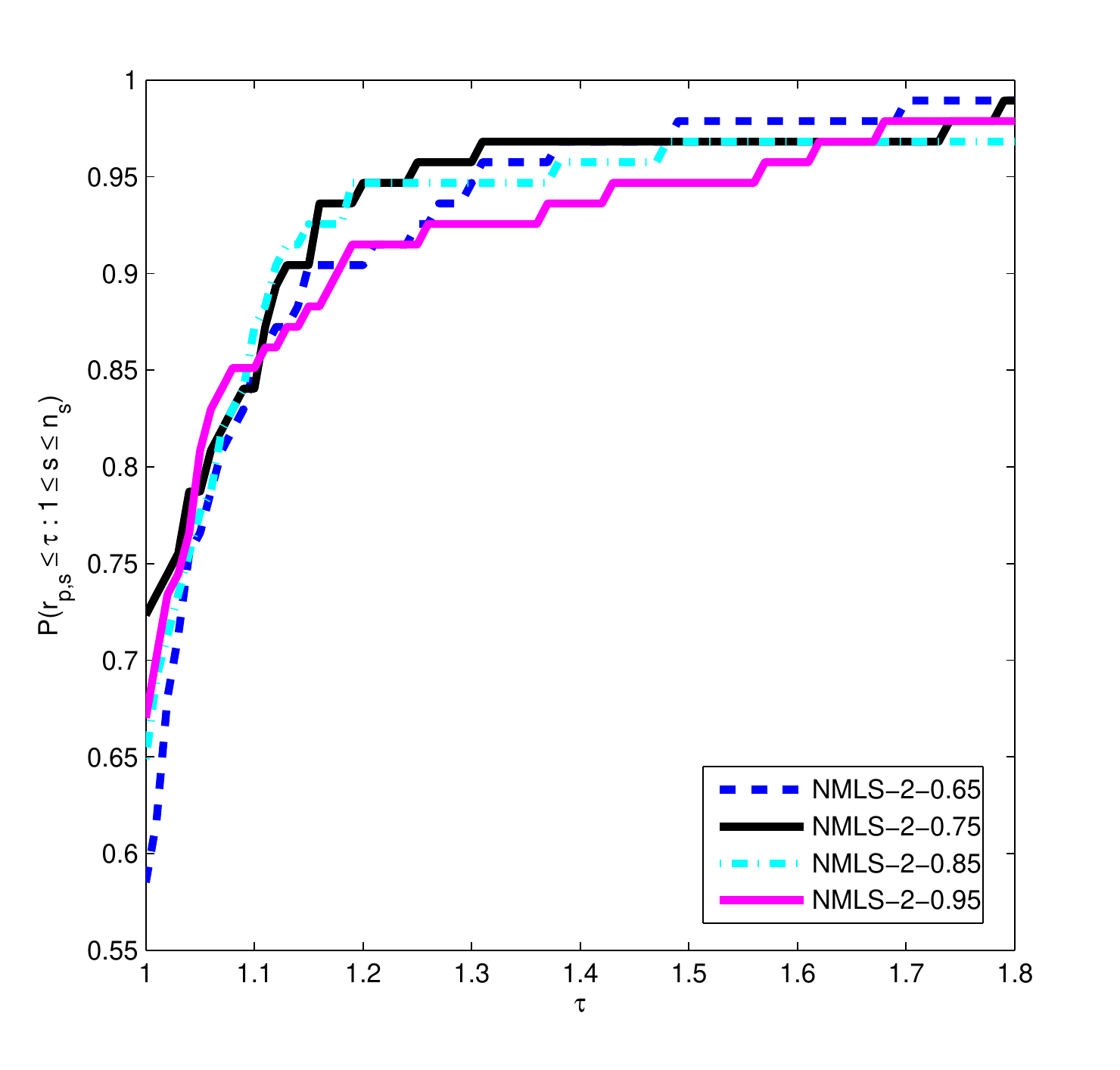}} 
 \qquad 
 \subfloat[][$N_f  +  3 N_g$ performance profile (NMLS-1)]{\includegraphics[width=7.5cm]{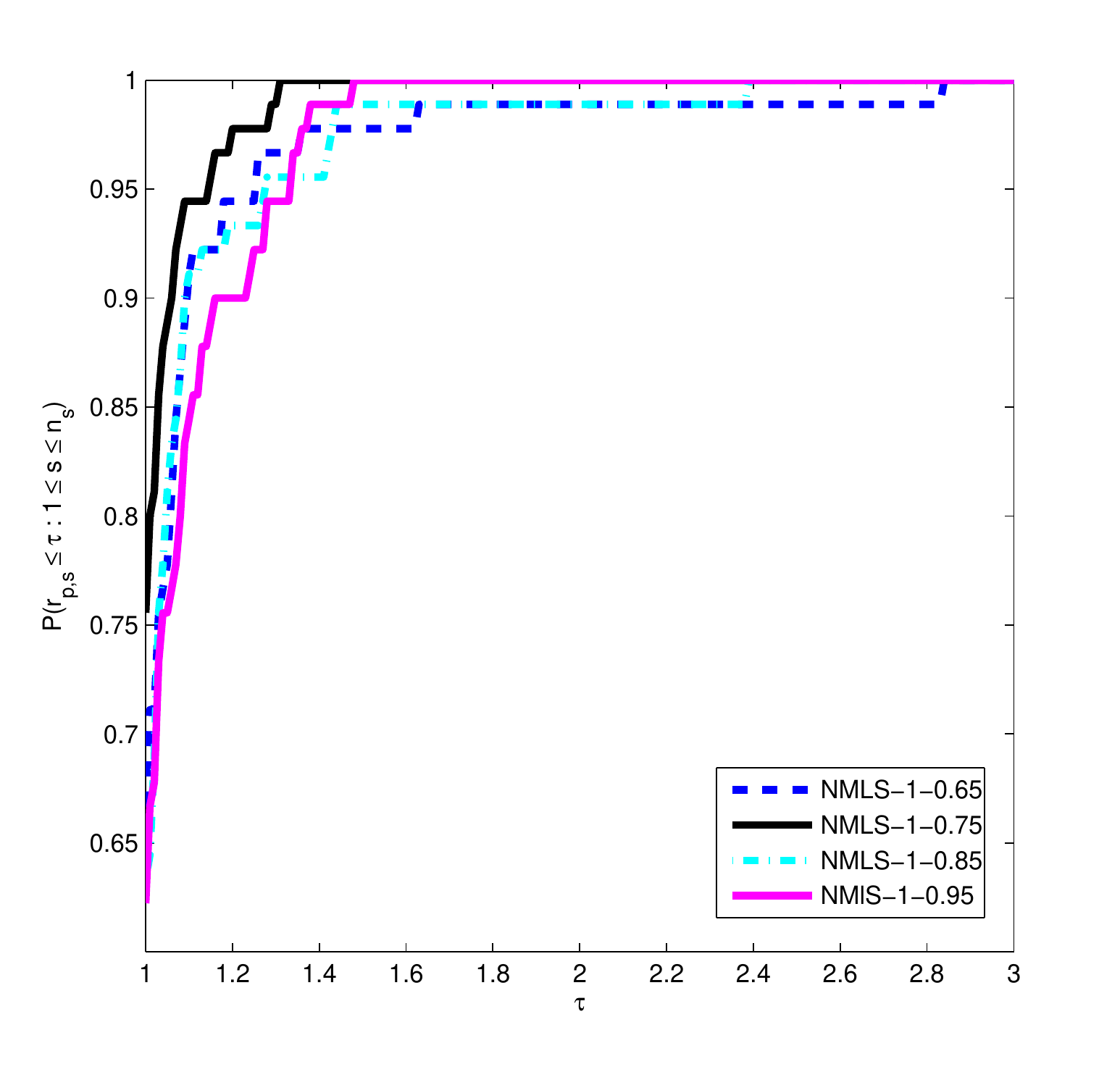}}% 
 \qquad 
 \subfloat[][$N_f  +  3 N_g$ performance profile (NMLS-2)]{\includegraphics[width=7.5cm]{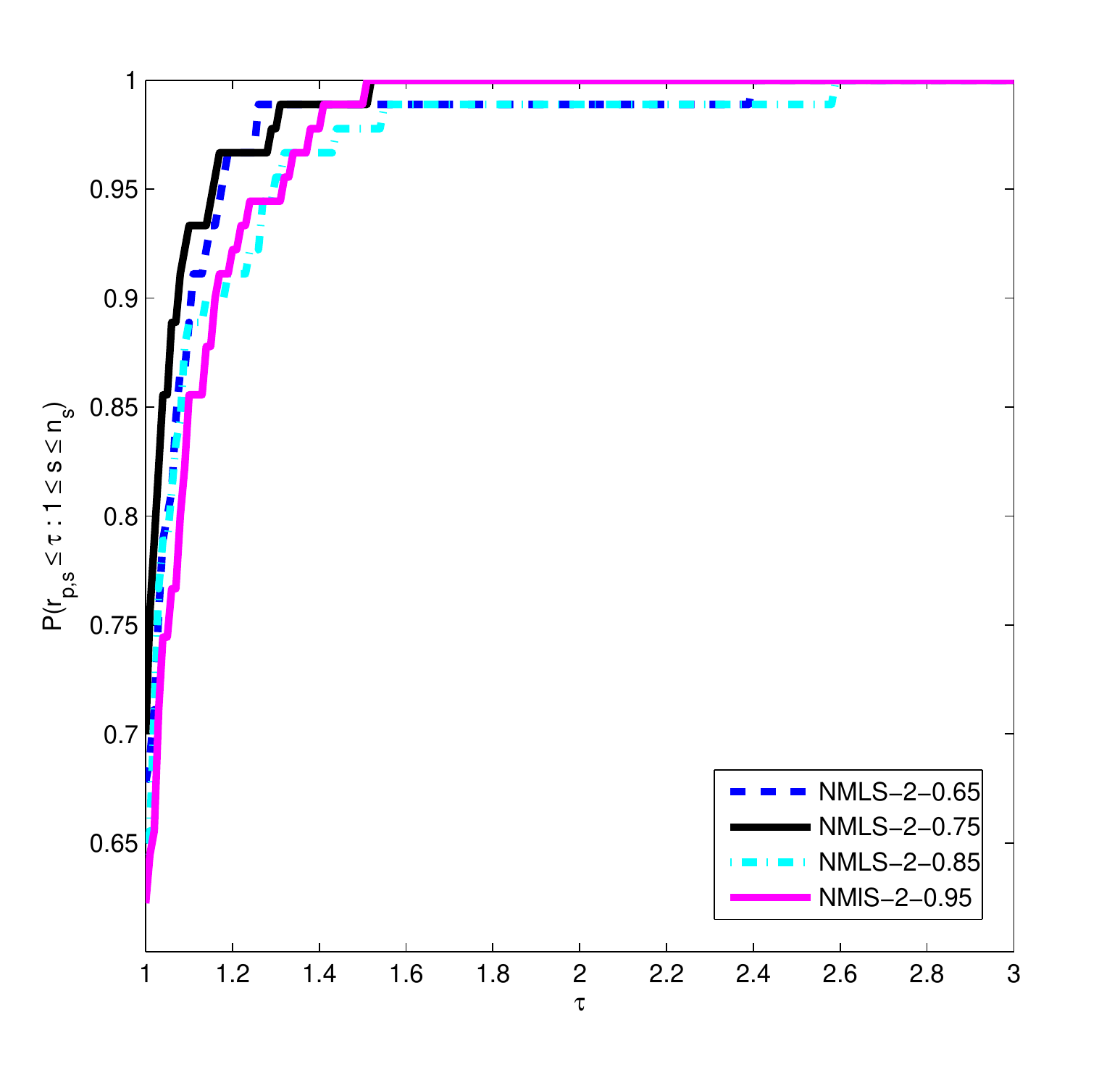}} 
 \caption{Performance profiles of NMLS-1 and NMLS-2 with the performance measures: (a) and (b) for the number of iterations ($N_i$) or gradient evaluations ($N_g$); (c) and (d) for the number of function evaluations ($N_f$); (e) and (f) for the hybrid measure $N_f  +  3 N_g$.}
\end{figure} 

In Figure 4, subfigures (a), (c) and (e) suggest that the results of NMLS-1 with $\eta_0=0.75$ are considerably better than those reported for others parameters regarding all of considered measures. In particular, it wins 77\%, 73\% and 75\% score among others for $N_i$, $N_f$ and $N_f+3 N_g$, respectively. The same results for NMLS-2 in subfigures (b), (d) and (f) of Figure 4 can be observed, where NMLS-2-0.75 respectively wins in 70\%, 71\% and 70\% of the cases for the considered measures. Therefore, we consider $\eta_0 = 0.75$ for our algorithms and for the sake of simplicity denote NMLS-1-0.75 and NMLS-2-0.75 by NMLS-1 and NMLS-2 in the rest of the paper. 

\begin{figure}\label{com6} 
 \centering 
 \subfloat[][$N_i$ and $N_g$ performance profile]{\includegraphics[width=7.7cm]{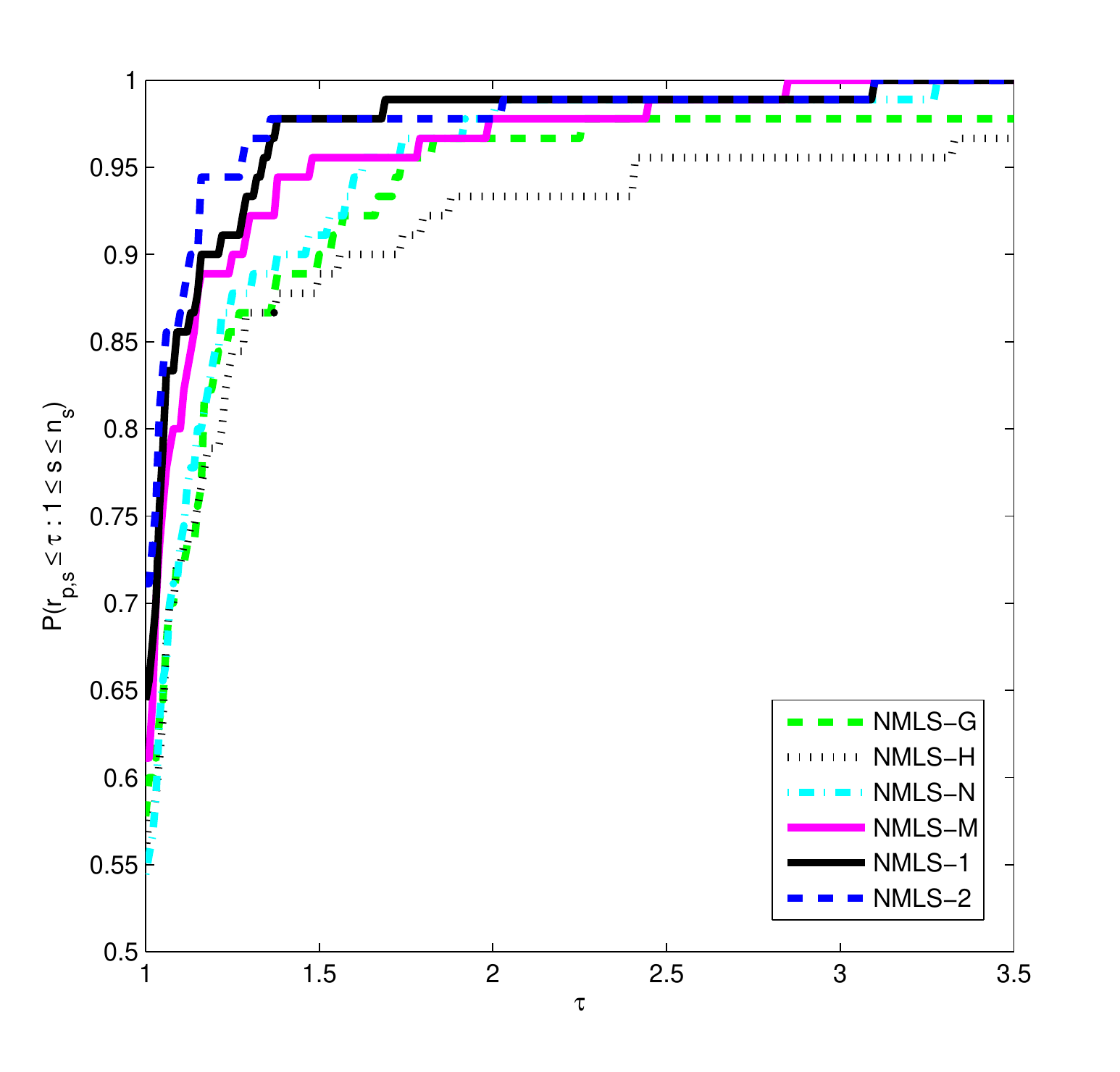}}% 
 \qquad 
 \subfloat[][$N_f$ performance profile]{\includegraphics[width=7.7cm]{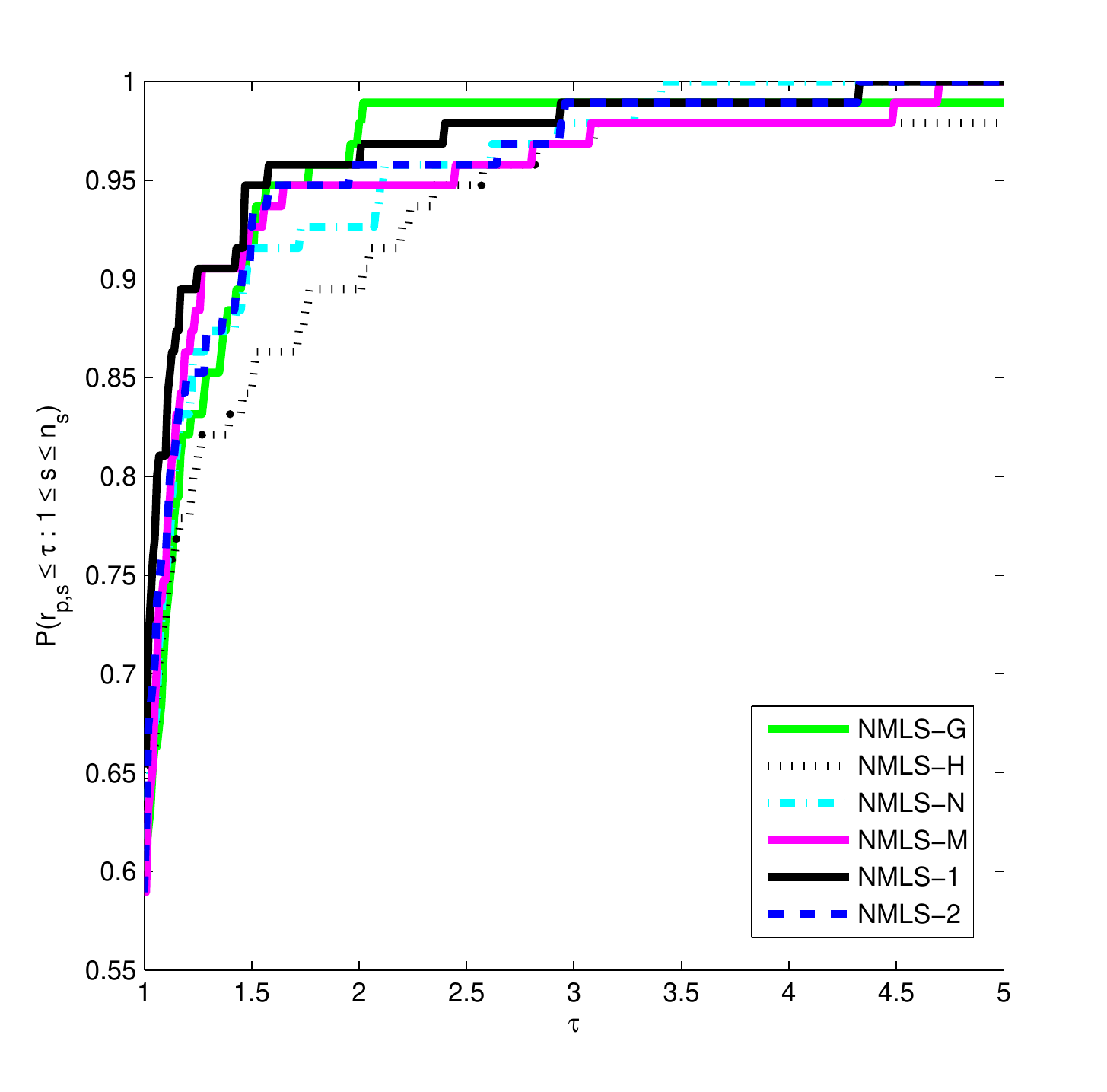}} 
 \qquad 
 \subfloat[][$N_f  +  3 N_g$ performance profile ($\tau = 1.25$)]{\includegraphics[width=7.7cm]{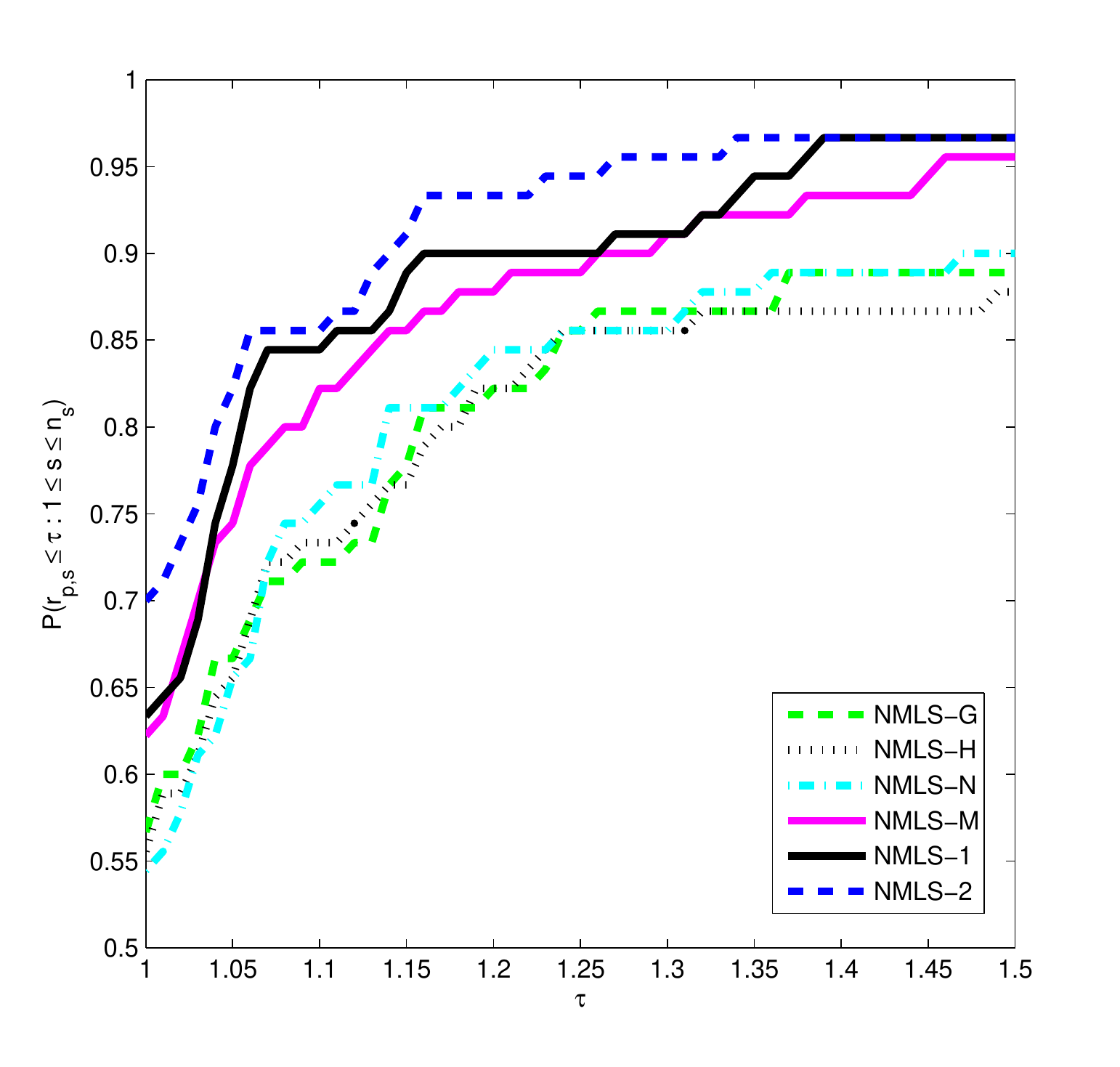}}% 
 \qquad 
 \subfloat[][$N_f  +  3 N_g$ performance profile ($\tau = 5$)]{\includegraphics[width=7.7cm]{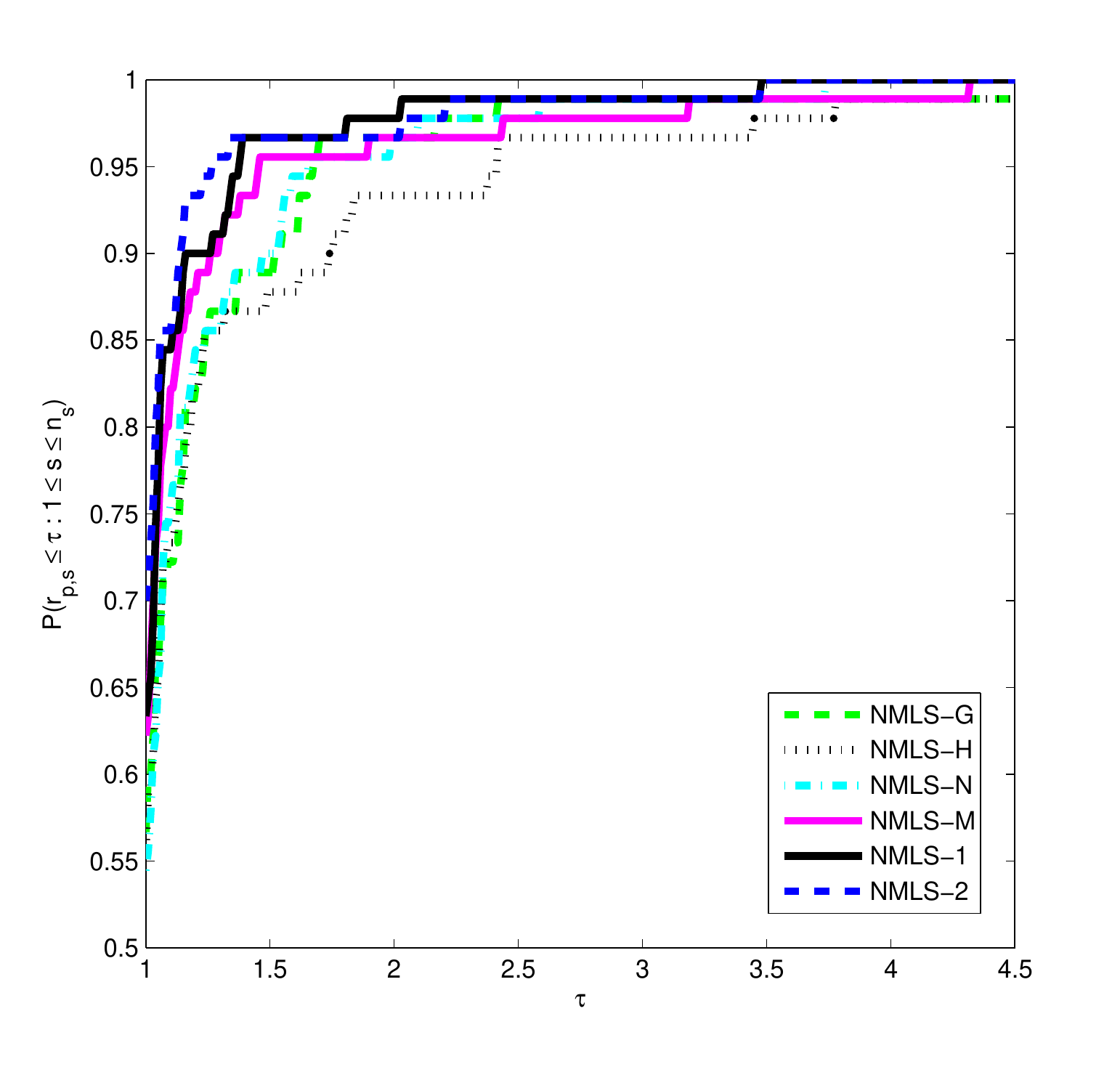}} 
 \caption{Performance profiles of all considered algorithms measured by: (a) the number of iterations ($N_i$) or gradient evaluations ($N_g$); (b) the number of function evaluations ($N_f$); (c) and (d) the hybrid measure $N_f  +  3 N_g$.} 
\end{figure} 

In this point, we test NMLS-G, NMLS-H, NMLS-R, NMLS-M, NMLS-1 and NMLS-2 to solve the problem (\ref{e.func}) by the LBFGS direction. Results of implementation are illustrated in Figure 5. From subfigure (a) of Figure 5, NMLS-2 obtains the most wins by 71\%, then NMLS-1 has the next place by 64\%. Moreover, NMLS-1 and NMLS-2 solve all problems in about $\tau = 3.8$. The subfigure (b) of Figure 5 demonstrates the number of function evaluations suggesting the similar results discussed about the subfigure (a). In Figure 5, subfigures (c) and (d) illustrate the performance profile of the algorithms with the measure $N_f+3 N_g$ by different amount of $\tau$ indicating that NMLS-1 and NMLS-2 win by 63\% and 70\% score among others, and they also solve the problems in less amount of $\tau$.   

% ##########################################################################################################   
\subsection{Experiment with Barzilai-Borwein}
The Barzilai-Borwein (BB) method for solving the problem (\ref{e.func}) is a gradient-type method proposed by {\sc Barzilai \& Borwein} in \cite{BarB}, where a step-size along the gradient descent direction $-g_k$ is generated using a two-point approximation of the secant equation $B_k s_{k-1} = y_{k-1}$ with $s_{k-1}=x_{k-1}-x_{k-2}$ and $y_{k-1}=g_{k-1}-g_{k-2}$. In particular, by imposing $B_k= \sigma_k I$ and solving the least-squares problem
\begin{equation*}
\begin{array}{ll} 
\textrm{minimize}   &  ~  \|\sigma s_{k-1} - y_{k-1}\|_2^2 \\ 
\textrm{subject to} &  ~  \sigma \in \mathbb{R},
\end{array} 
\end{equation*}
one can obtain
\[
\sigma_k^{BB1} = \frac{s_{k-1}^T y_{k-1}}{s_{k-1}^T s_{k-1}}.
\]
Hence the two-point approximated quasi-newton direction is computed by
\begin{equation*} 
d_{k} = - \left( \sigma_k^{BB1} \right)^{-1} g_k = - \frac{s_{k-1}^T s_{k-1}}{s_{k-1}^T y_{k-1}} g_k.
\end{equation*}
Since $\sigma_k$ in this direction can be unacceptably small or large for non-quadratic objective function, we use the following safeguarded step-size
\begin{equation}\label{e.mdbb1}
d_k^{BB1} = \left\{
\begin{array}{ll}
-\left( \sigma_k^{BB1} \right)^{-1} g_k &~~ \mathrm{if}~ 10^{-10} \leq \left( \sigma_k^{BB1} \right)^{-1} \leq 10^{10},\\\\
-g_k                                           &~~ \mathrm{otherwise.}
\end{array}
\right.
\end{equation}
Similarly, setting $B_k^{-1} = \frac{1}{\sigma_k} I$ and solving the minimization problem
\begin{equation*} 
\begin{array}{ll} 
\textrm{minimize}   &  ~  \|s_{k-1} - \frac{1}{\sigma} y_{k-1}\|_2^2 \\ 
\textrm{subject to} &  ~  \sigma \in \mathbb{R},
\end{array} 
\end{equation*} 
we obtain the step-size
\[
\sigma_k^{BB2} = \frac{y_{k-1}^T s_{k-1}}{y_{k-1}^T y_{k-1}}.
\]
Considering the safeguard used in (\ref{e.mdbb1}), we obtain the following search direction 
\begin{equation}\label{e.mdbb2}
d_k^{BB2} = \left\{
\begin{array}{ll}
-\left( \sigma_k^{BB2} \right)^{-1} g_k &~~ \mathrm{if}~ 10^{-10} \leq \left( \sigma_k^{BB2} \right)^{-1} \leq 10^{10},\\\\
-g_k                                           &~~ \mathrm{otherwise.}
\end{array}
\right.
\end{equation}
The numerical experiments with the Barzilai-Borwein directions have shown the significant development in efficiency of gradient methods. Being computationally efficient and needing low memory requirement make this scheme interesting to solve large-scale optimization problems. Therefore, it receives much attention during the last two decades and lots of modifications and developments for both unconstrained and constrained optimization have been proposed, for example see \cite{AndBM,BirMR1,BirMR2, DaiF, DaiHSZ, DaiYY, HagMZ, JudRRS,LueRGH, Ray} and references therein.

In the rest of this subsection, we consider versions of Algorithm 2 equipped with the nonmonotone terms using the Barzilai-Borwein directions (\ref{e.mdbb1}) and (\ref{e.mdbb2}) for solving the problem (\ref{e.func}). We here set $\sigma = 10^{-4}$. To find the best possible parameter $\eta_0$, we consider $\eta_0 = 0.65$, $\eta_0 = 0.75$, $\eta_0 = 0.80$ and $\eta_0 = 0.90$ and run NMLS-1 and NMLS-2 for both directions (\ref{e.mdbb1}) and (\ref{e.mdbb2}). The corresponding results are summarized in Figures 6 and 7, where the first row shows the performance profile for the number of gradients $N_g$, the second row shows the performance profile for the number of function evaluations $N_f$ and the third row shows the performance profile for $N_f + 3N_g$. From all subfigures of Figures 6 and 7, we conclude that $\eta_0 = 0.80$ and $\eta_0 = 0.90$ produce acceptable results for our algorithms with respect to the directions (\ref{e.mdbb1}) and (\ref{e.mdbb2}), i.e., NMLS-1 and NMLS-2 exploit $\eta_0 = 0.80$ and $\eta_0 = 0.90$ for these directions, respectively.

\begin{figure}\label{com7_par} 
 \centering 
 \subfloat[][$N_i$ and $N_g$ performance profile (NMLS-1)]{\includegraphics[width=7.5cm]{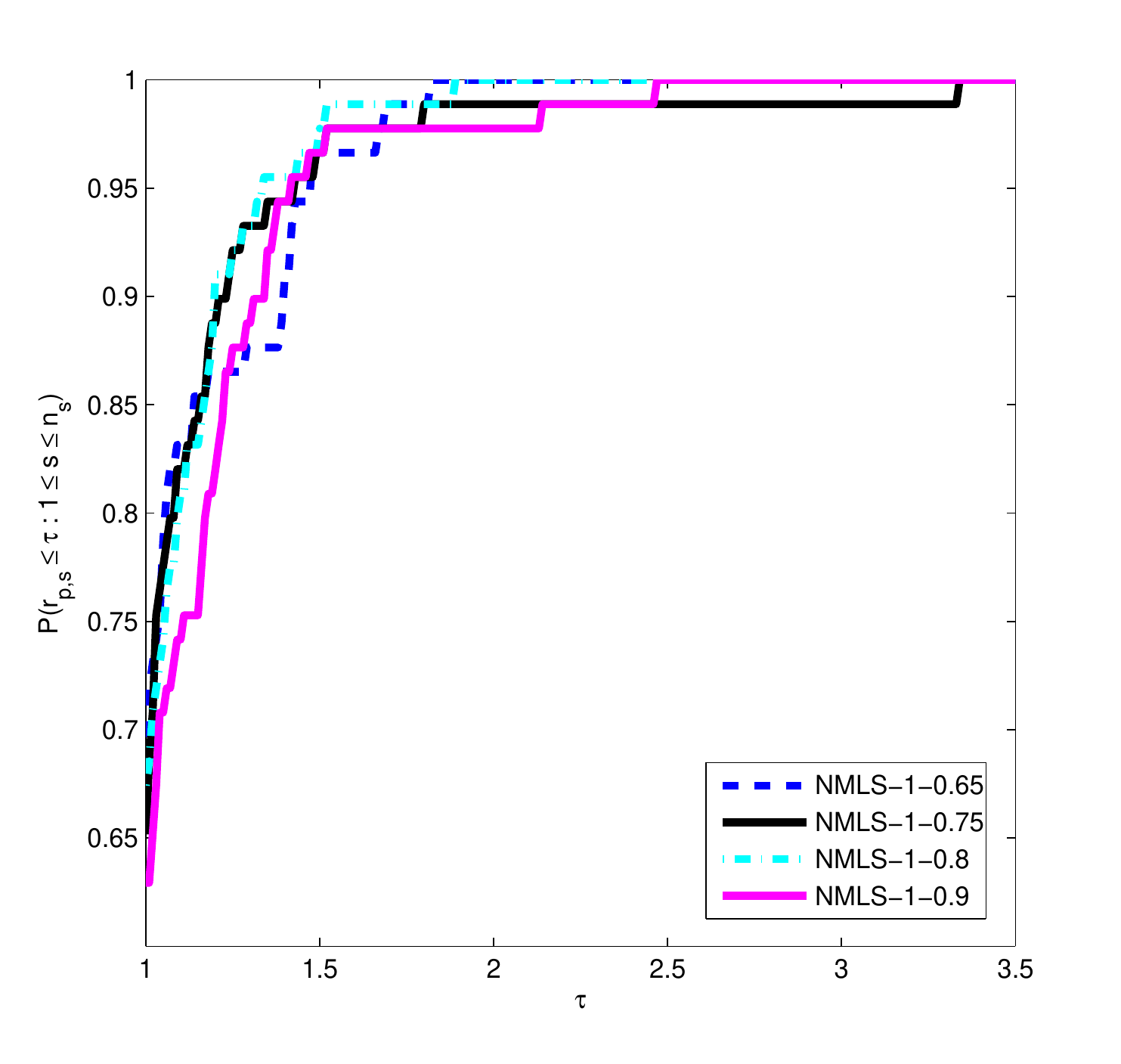}}% 
 \qquad 
 \subfloat[][$N_i$ and $N_g$ performance profile (NMLS-2)]{\includegraphics[width=7.5cm]{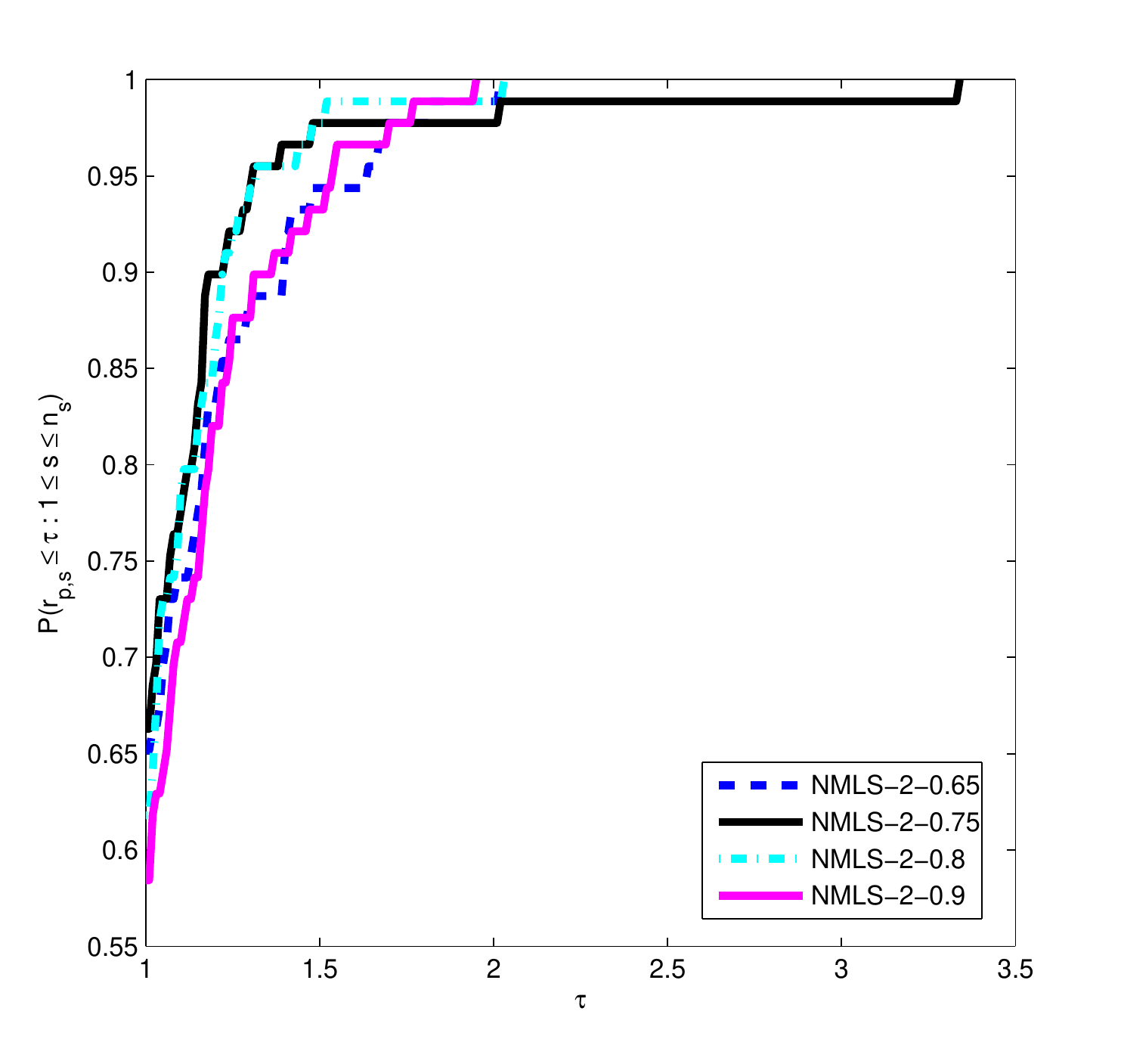}} 
 \qquad 
 \subfloat[][$N_f$ performance profile (NMLS-1)]{\includegraphics[width=7.5cm]{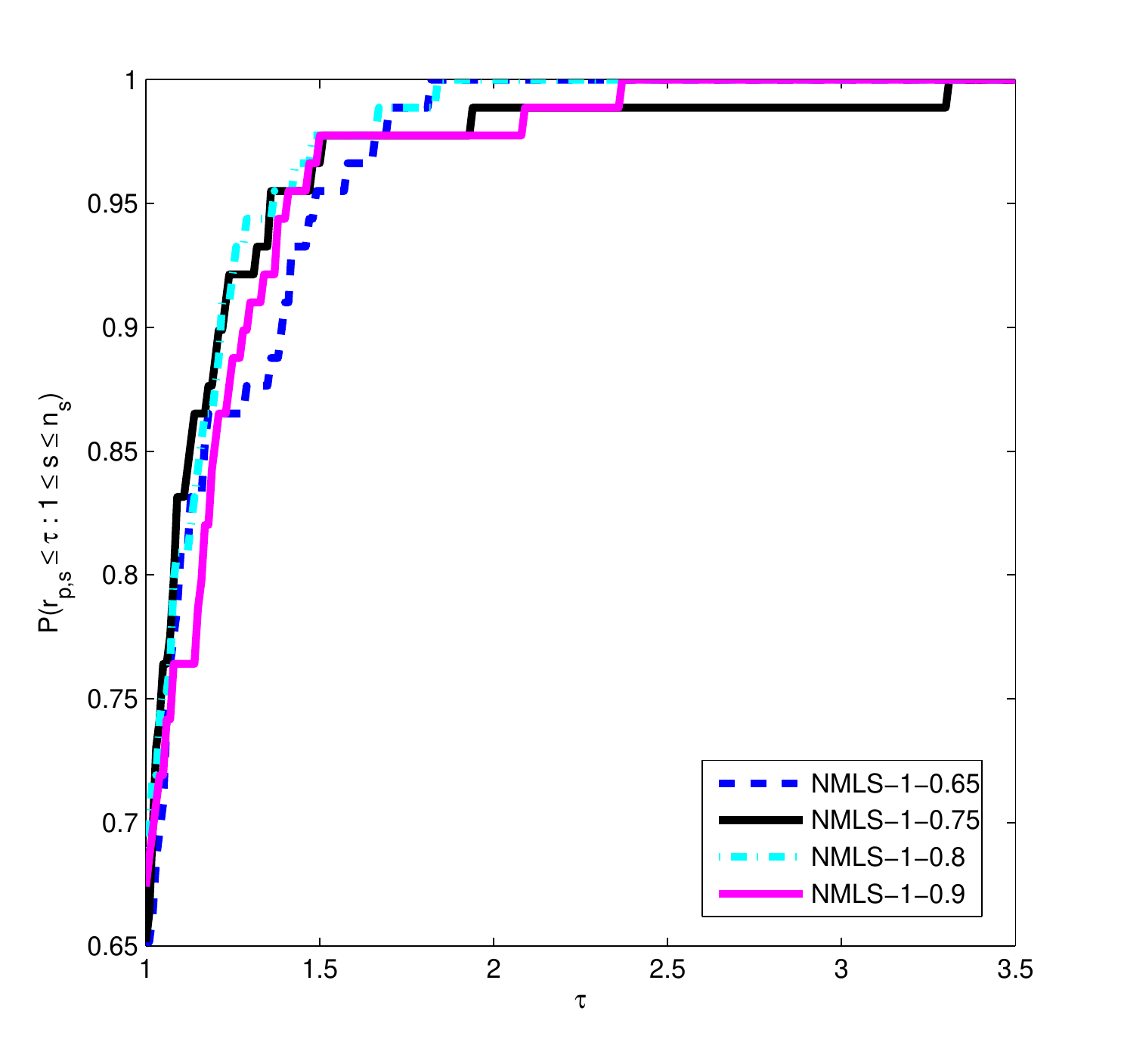}}% 
 \qquad 
 \subfloat[][$N_f$ performance profile (NMLS-2)]{\includegraphics[width=7.5cm]{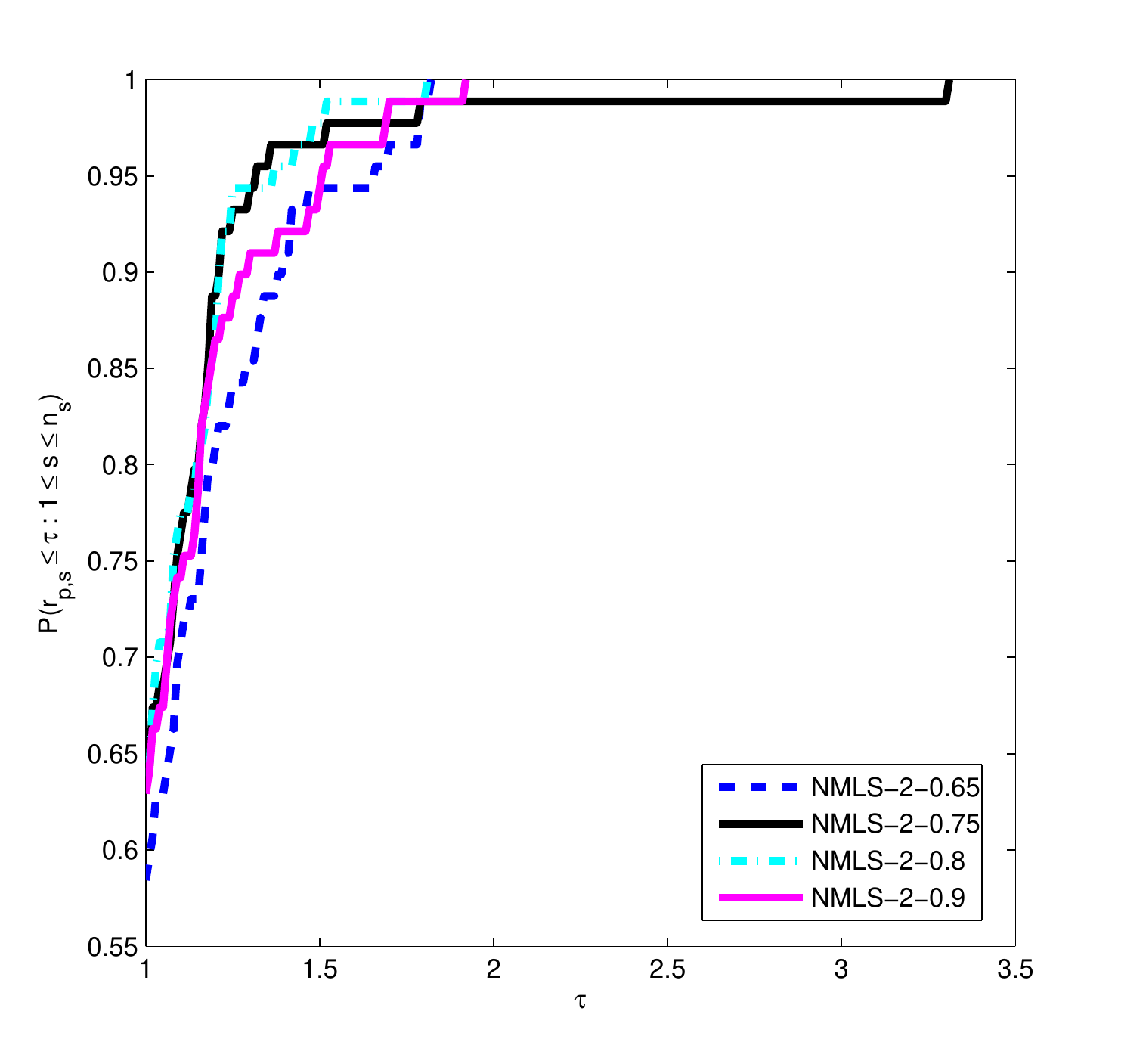}} 
 \qquad 
 \subfloat[][$N_f  +  3 N_g$ performance profile (NMLS-1)]{\includegraphics[width=7.5cm]{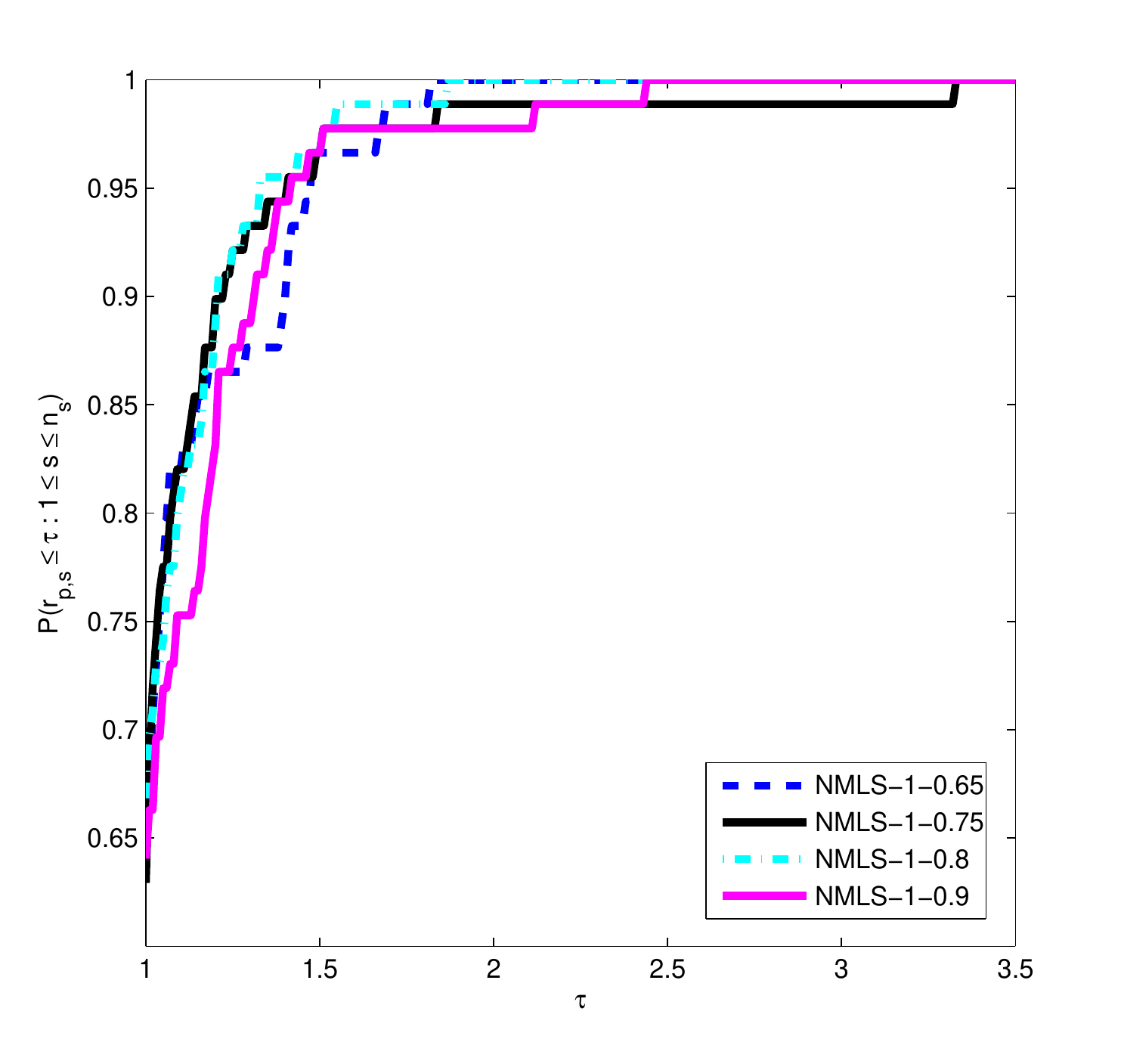}}% 
 \qquad 
 \subfloat[][$N_f  +  3 N_g$ performance profile (NMLS-2)]{\includegraphics[width=7.5cm]{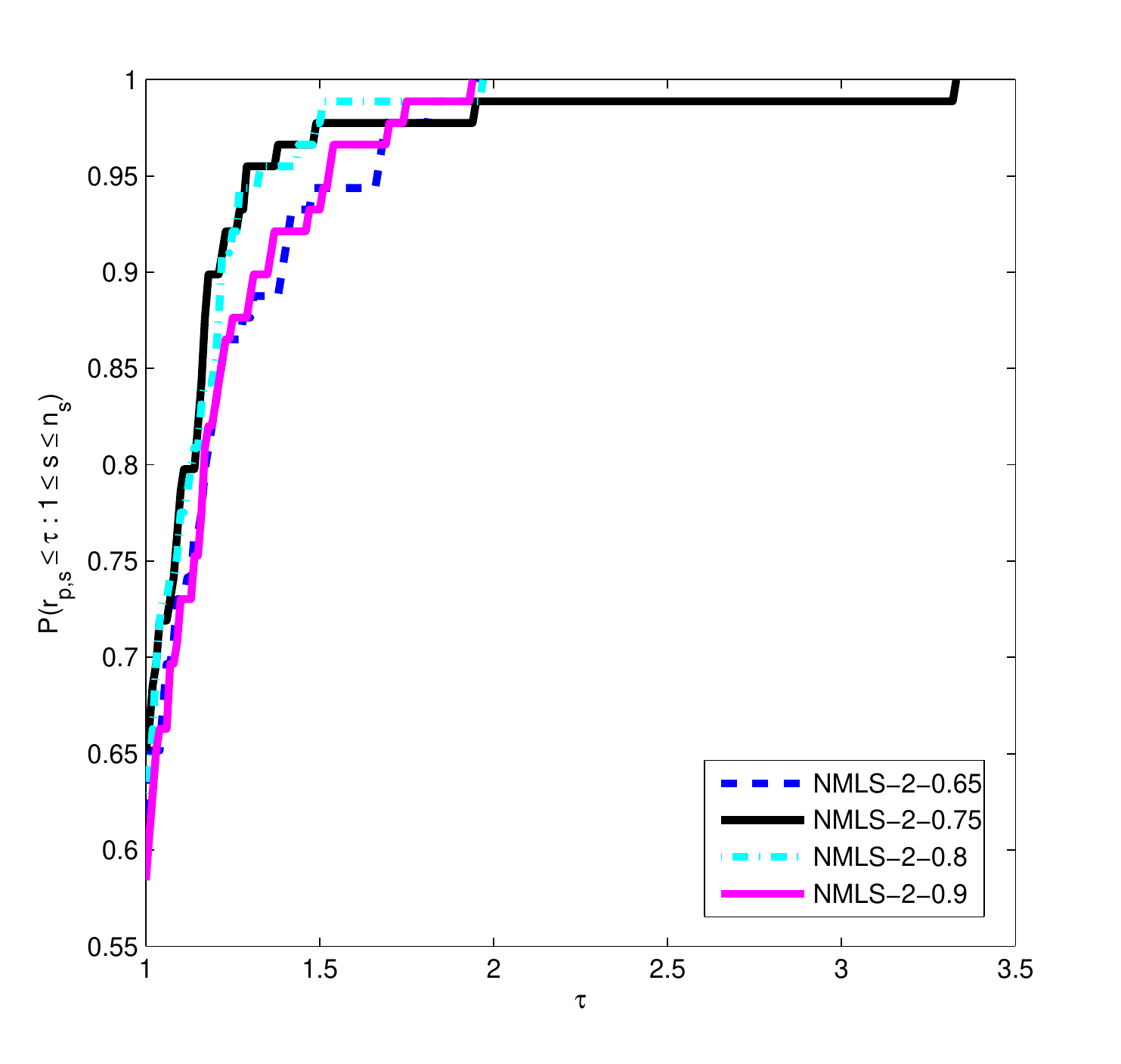}} 
 \caption{Performance profiles of NMLS-1 and NMLS-2 with the performance measures: (a) and (b) for the number of iterations ($N_i$) or gradient evaluations ($N_g$); (c) and (d) for the number of function evaluations ($N_f$); (e) and (f) for the hybrid measure $N_f  +  3 N_g$ .}
\end{figure} 

\begin{figure}\label{com7_par} 
 \centering 
 \subfloat[][$N_i$ and $N_g$ performance profile (NMLS-1)]{\includegraphics[width=7.5cm]{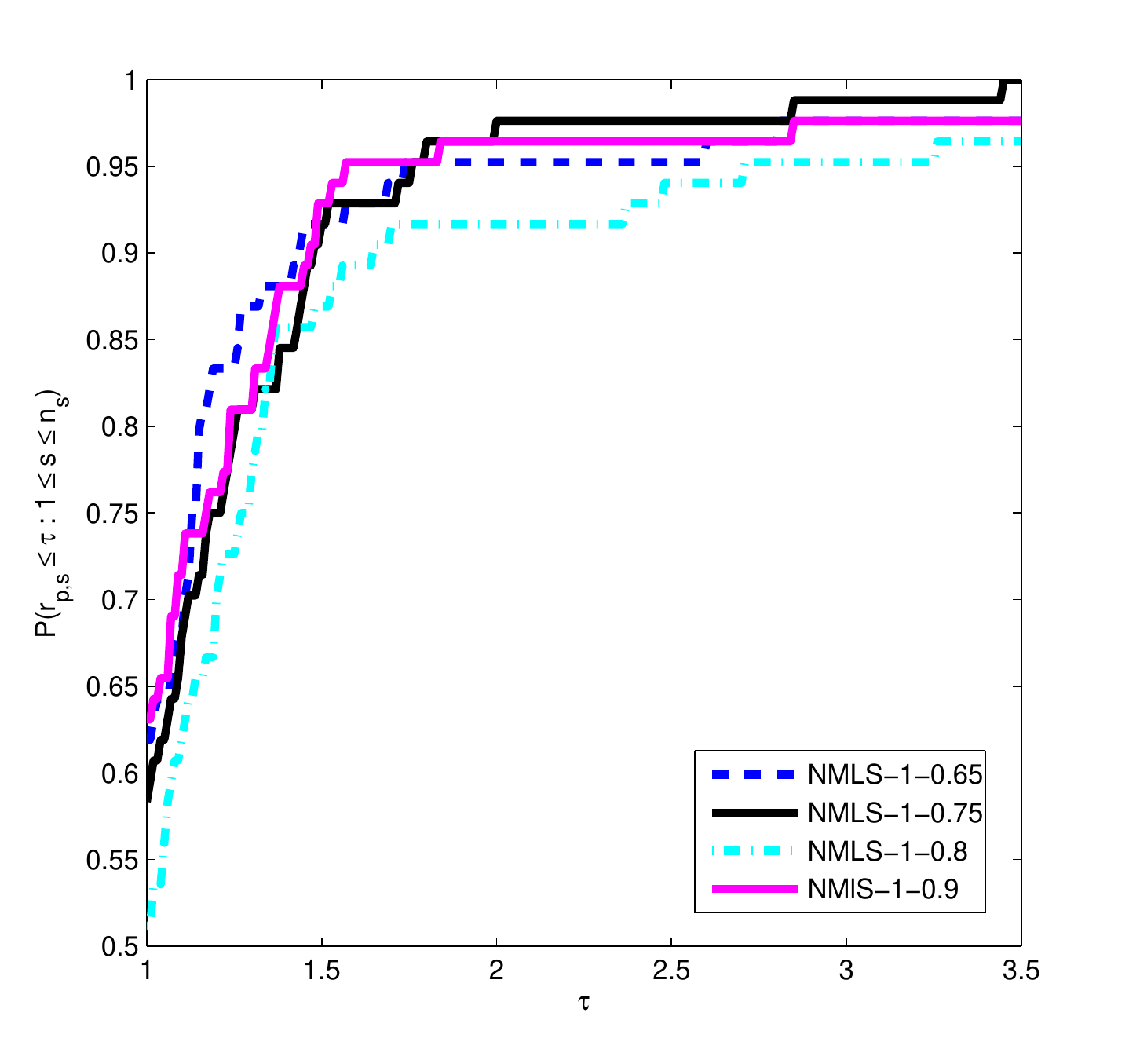}}% 
 \qquad 
 \subfloat[][$N_i$ and $N_g$ performance profile (NMLS-2)]{\includegraphics[width=7.5cm]{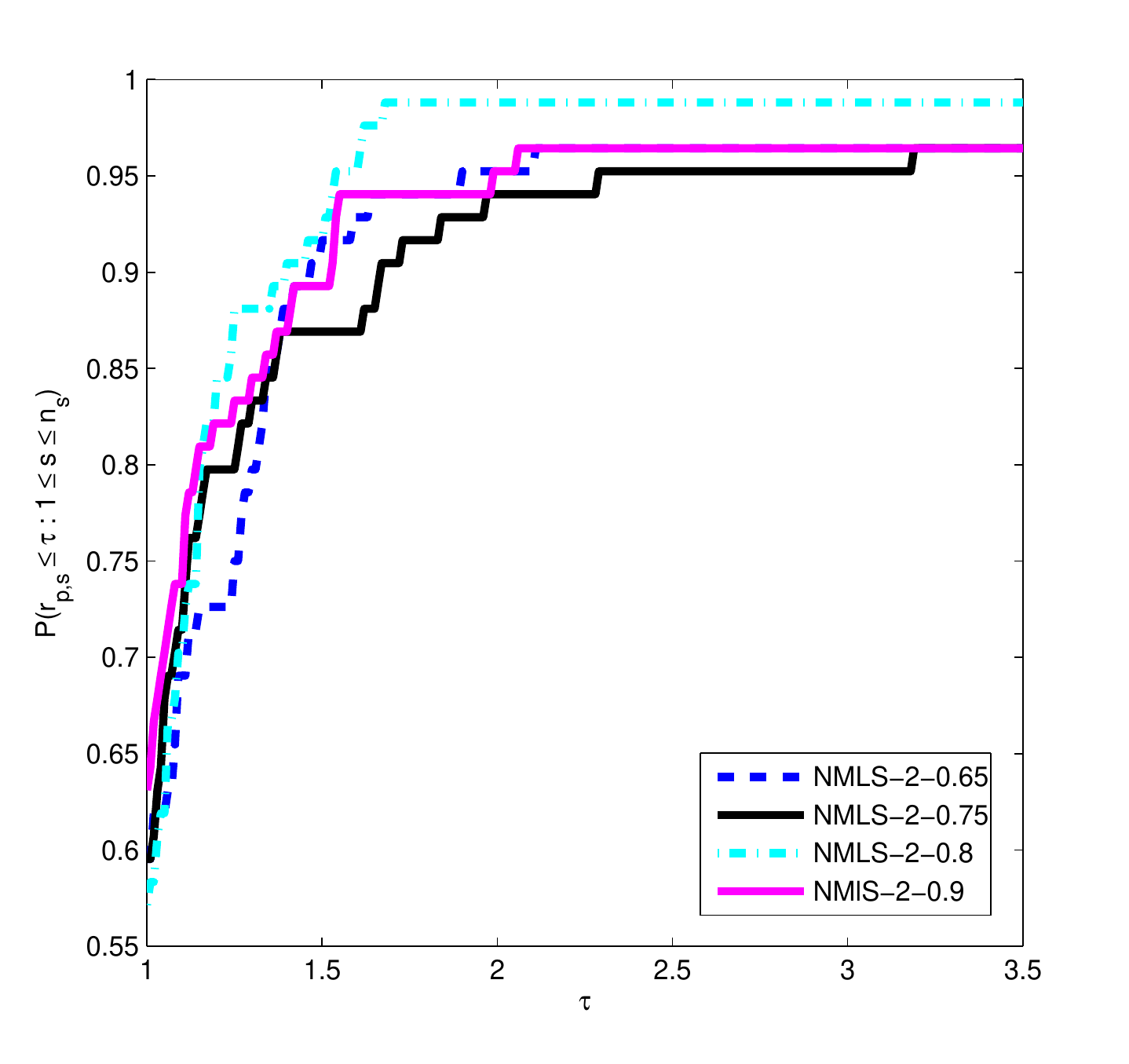}} 
 \qquad 
 \subfloat[][$N_f$ performance profile (NMLS-1)]{\includegraphics[width=7.5cm]{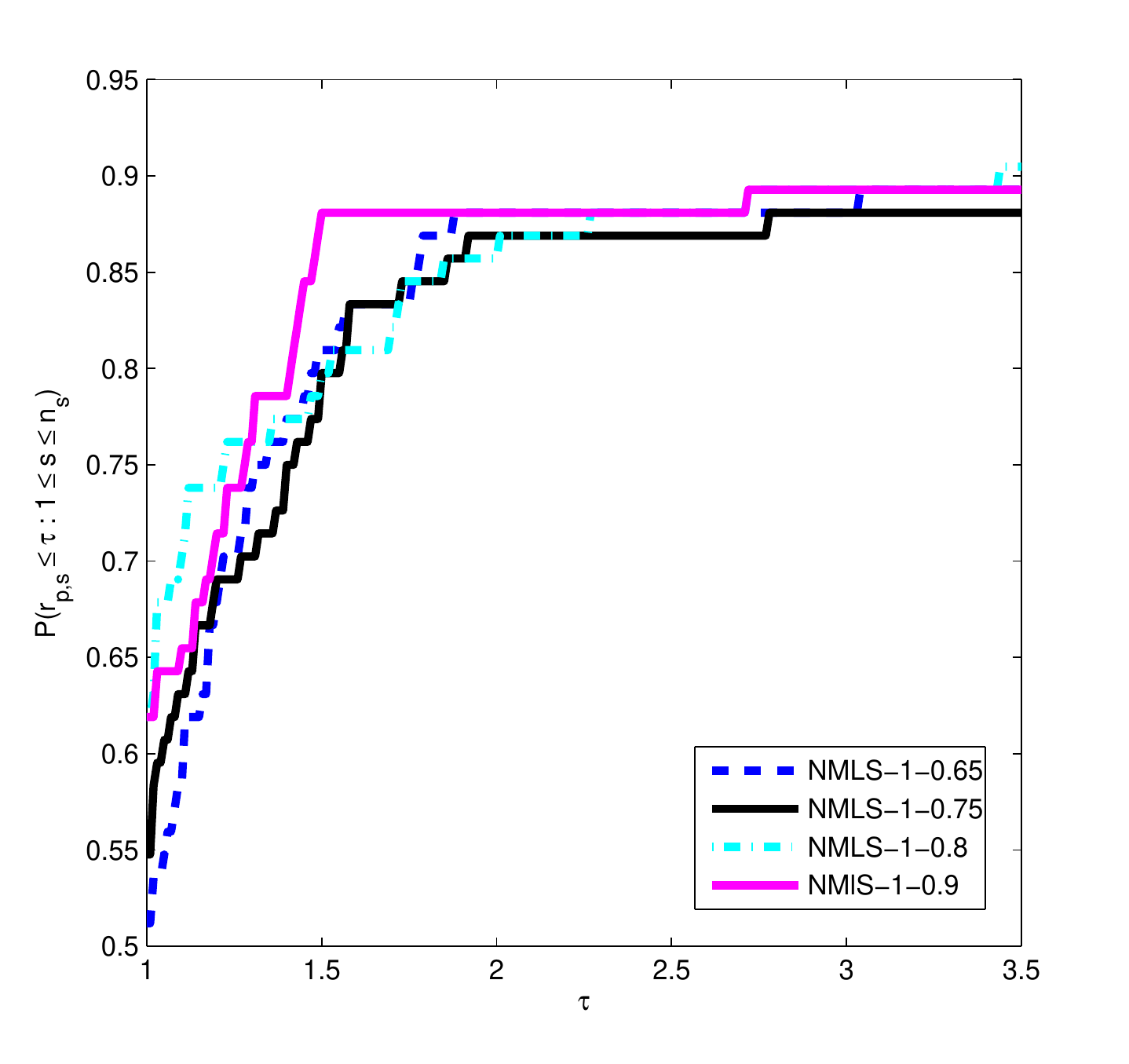}}% 
 \qquad 
 \subfloat[][$N_f$ performance profile (NMLS-2)]{\includegraphics[width=7.5cm]{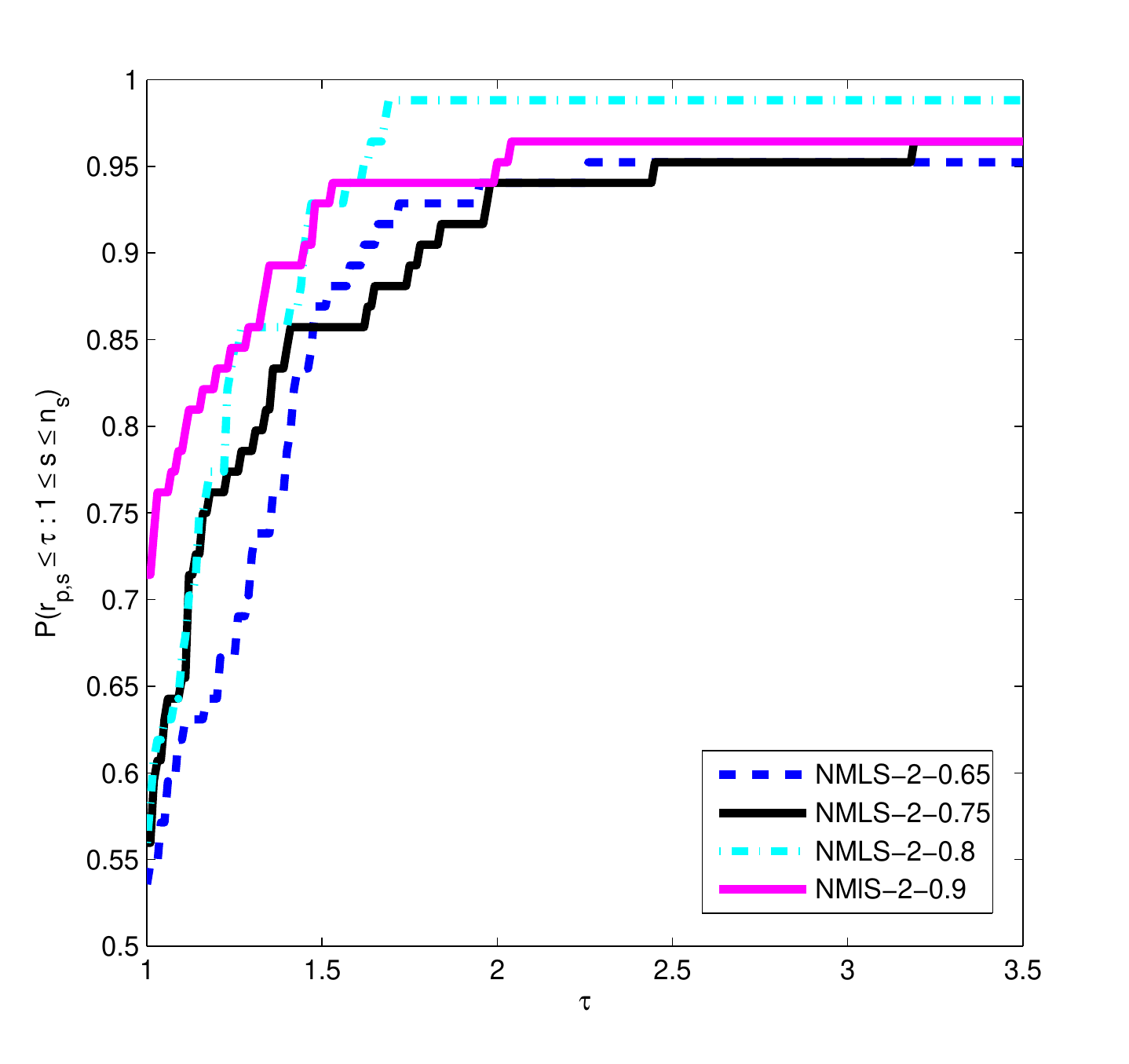}} 
 \qquad 
 \subfloat[][$N_f  +  3 N_g$ performance profile (NMLS-1)]{\includegraphics[width=7.5cm]{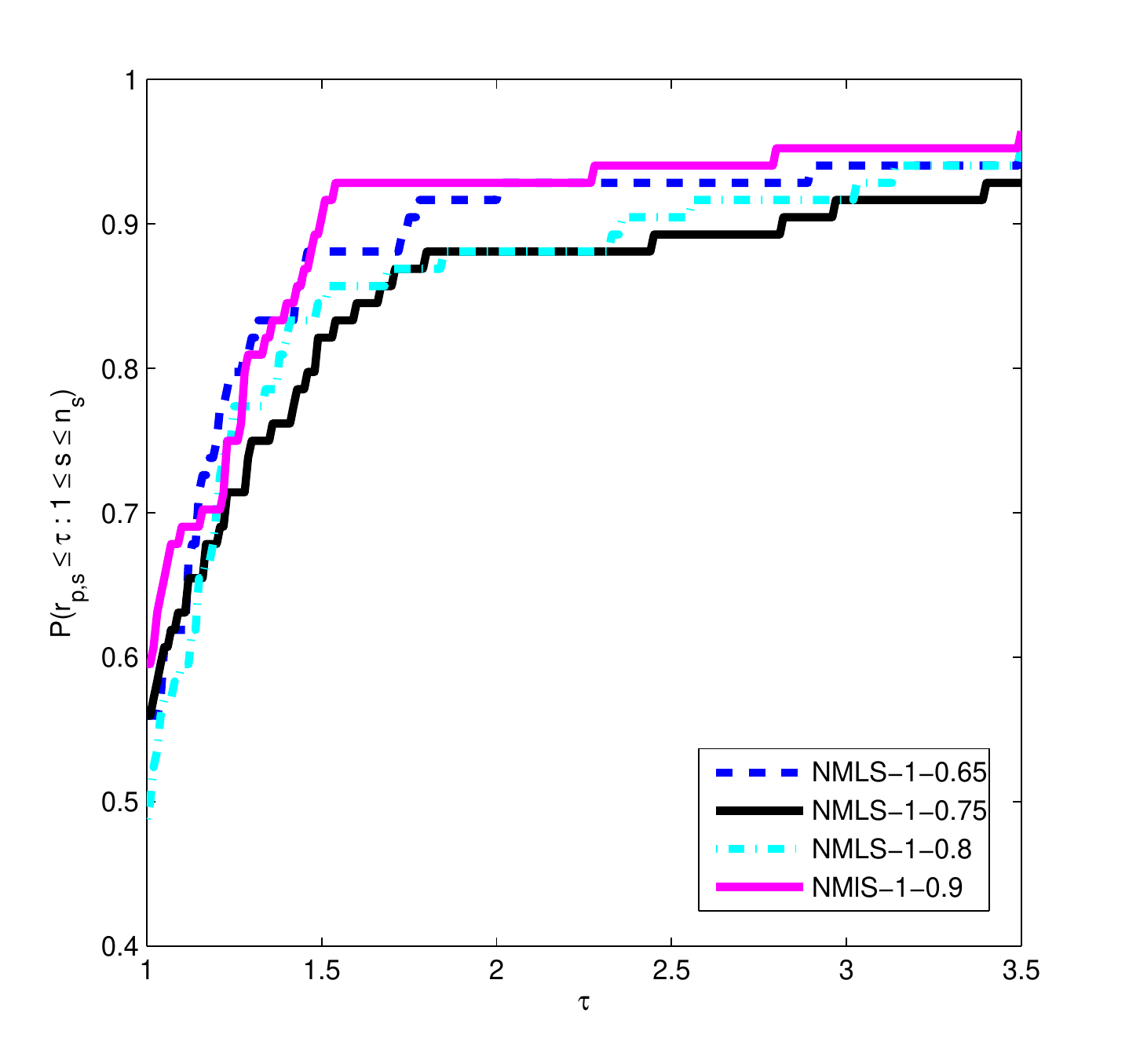}}% 
 \qquad 
 \subfloat[][$N_f  +  3 N_g$ performance profile (NMLS-2)]{\includegraphics[width=7.5cm]{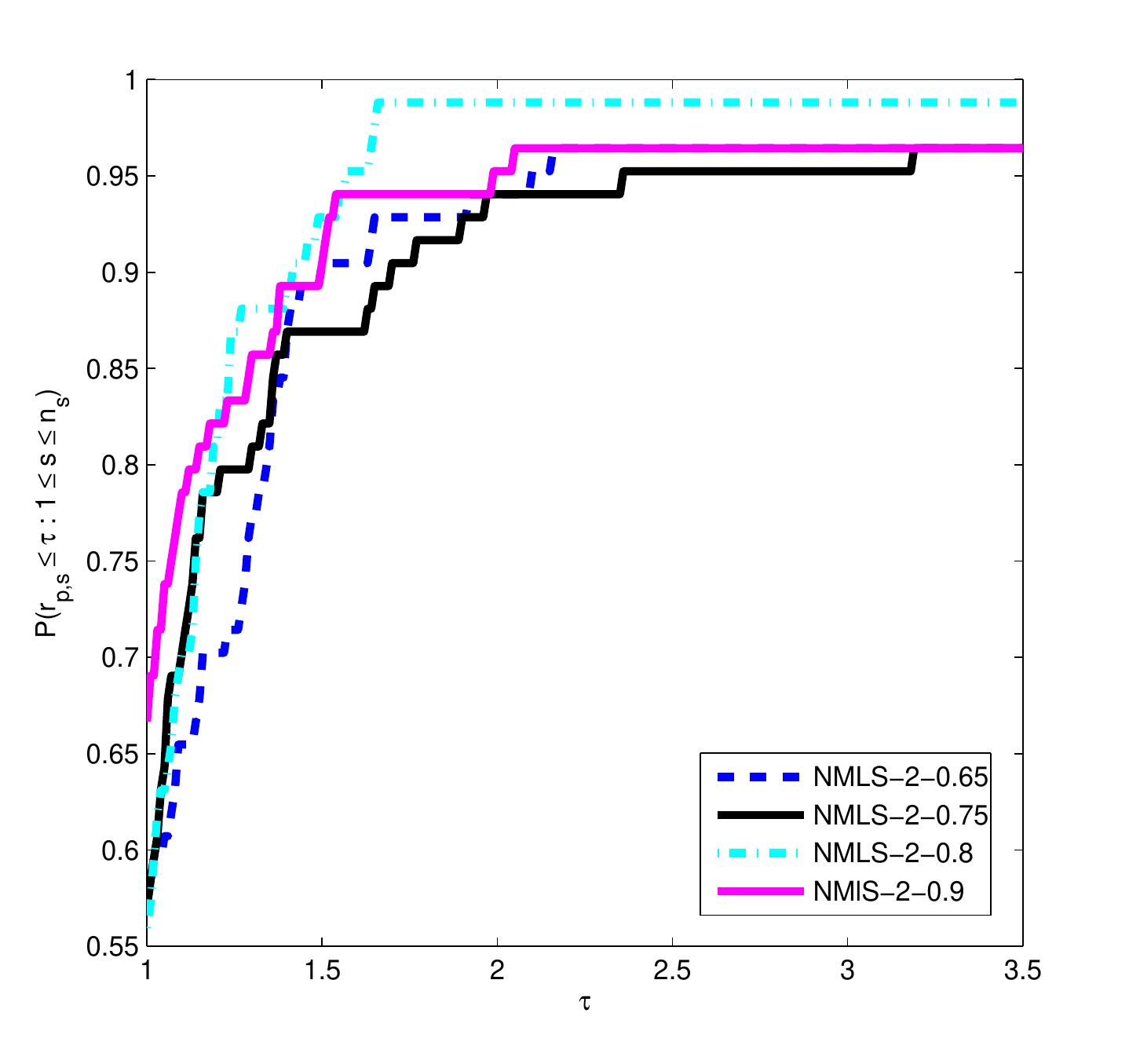}} 
 \caption{Performance profiles of NMLS-1 and NMLS-2 with the performance measures: (a) and (b) for the number of iterations ($N_i$) or gradient evaluations ($N_g$); (c) and (d) for the number of function evaluations ($N_f$); (e) and (f) for the hybrid measure $N_f  +  3 N_g$ .}
\end{figure} 

We now compare the performance of NMLS-G, NMLS-H, NMLS-N and NMLS-M, NMLS-1 and NMLS-2 using the 
directions (\ref{e.mdbb1}) and (\ref{e.mdbb2}). The test problems are those considered in the previous subsection. The related results are gathered in Tables 4 and 5. All considered algorithms are failed for some test problems in our implementation, $\mathtt{WASTON}$, $\mathtt{HARKAPER2}$, $\mathtt{POWER}$, $\mathtt{Extended~ HILBERT}$, $\mathtt{FLETCHER}$, $\mathtt{CUBE}$ and $\mathtt{Generalized~ White~ and~ Holst}$ for (\ref{e.mdbb1}) and $\mathtt{CUBE}$ and $\mathtt{Generalized~ White~ and~ Holst}$ for (\ref{e.mdbb2}), so we delete them from Tables 4 and 5. The performance profile of the algorithms are demonstrated in Figures 8 and 9 for the measures $N_g$, $N_f$ and $N_f + 3N_g$. Subfigures (a) and (b) of Figure 8 respectively illustrate the performance profile for $N_g$ and $N_f$ and indicate that the algorithms are comparable, however, NMLS-G and NMLS-H perform a little better regarding the number of function values. Subfigures (c) and (d) stand for the measure $N_f + 3N_g$ with $\tau = 1.5$ and $\tau = 3$, respectively. They show that MNLS-H attains the most wins by about 58\% and then NMLS-G by 57\% while NMLS-N, NMLS-1 and NMLS-2 attain about 53\% score of the wins and NMLS-M get the worst result by about 50\%. Similarly, subfigures (a) and (b) of Figure 9 demonstrate the performance profile for $N_g$ and $N_f$, where they are comparable regarding the number of gradient evaluations, and NMLS-H and NMLS-G perform better regarding the number of function evaluations. Subfigures (c) and (d) of Figure 9 show that NMLS-H, MNLS-G and NMLS-1 attain the most wins by about 49\%, 48\% and 47\% score, respectively.

\begin{figure}\label{com8} 
 \centering 
 \subfloat[][$N_i$ and $N_g$ performance profile]{\includegraphics[width=7.5cm]{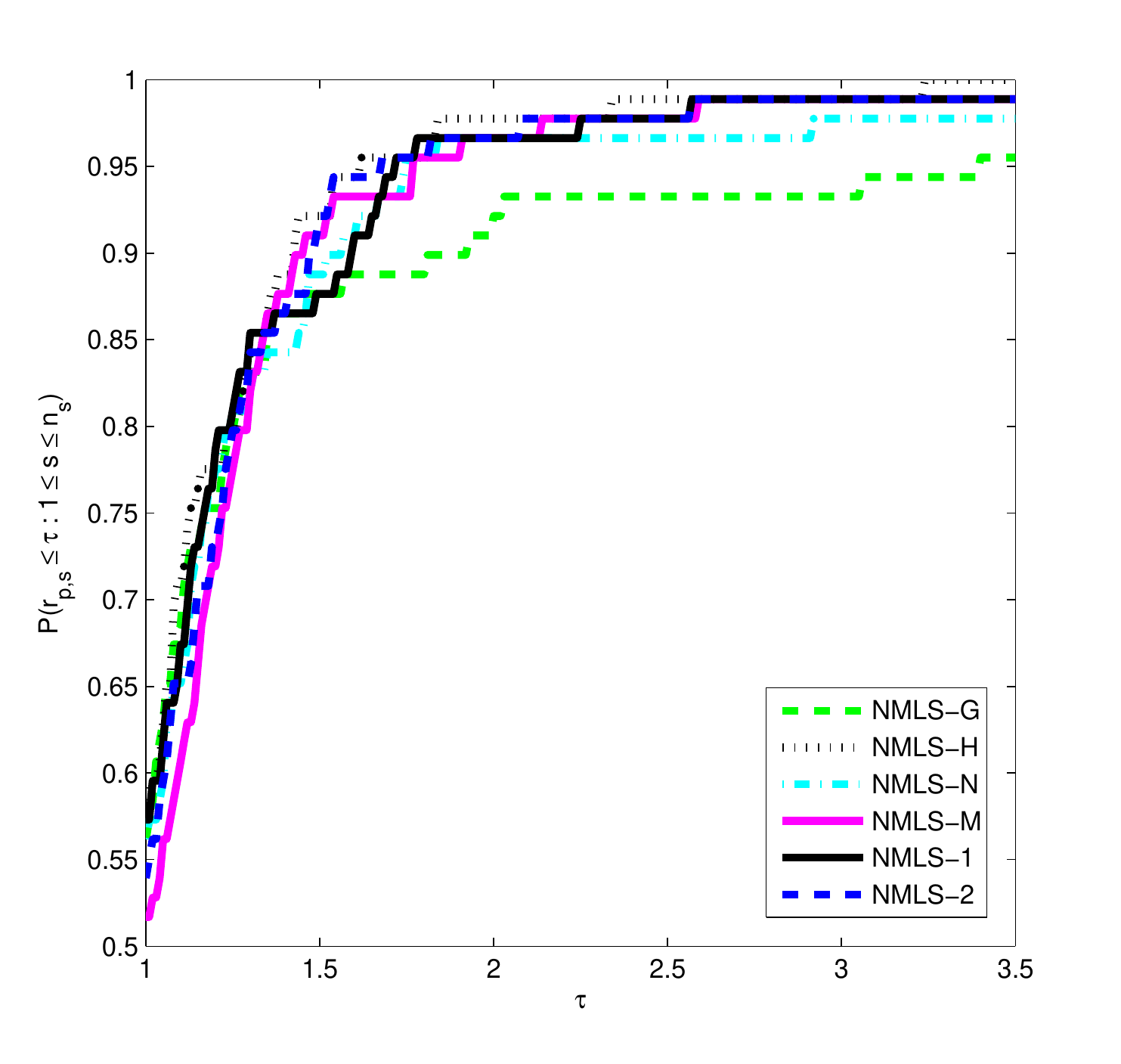}}% 
 \qquad 
 \subfloat[][$N_f$ performance profile]{\includegraphics[width=7.5cm]{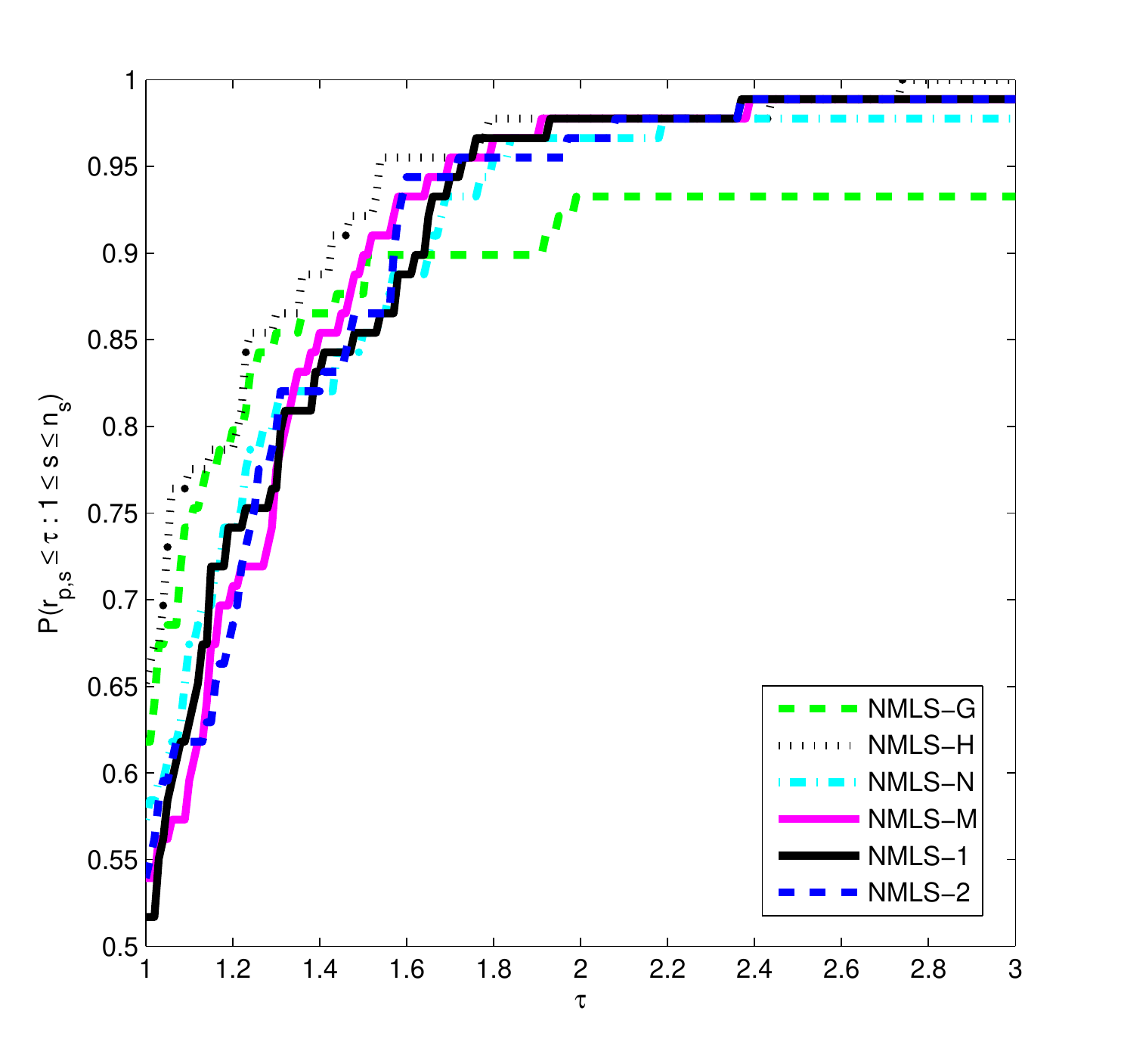}} 
 \qquad 
 \subfloat[][$N_f  +  3 N_g$ performance profile ($\tau = 1.5$)]{\includegraphics[width=7.5cm]{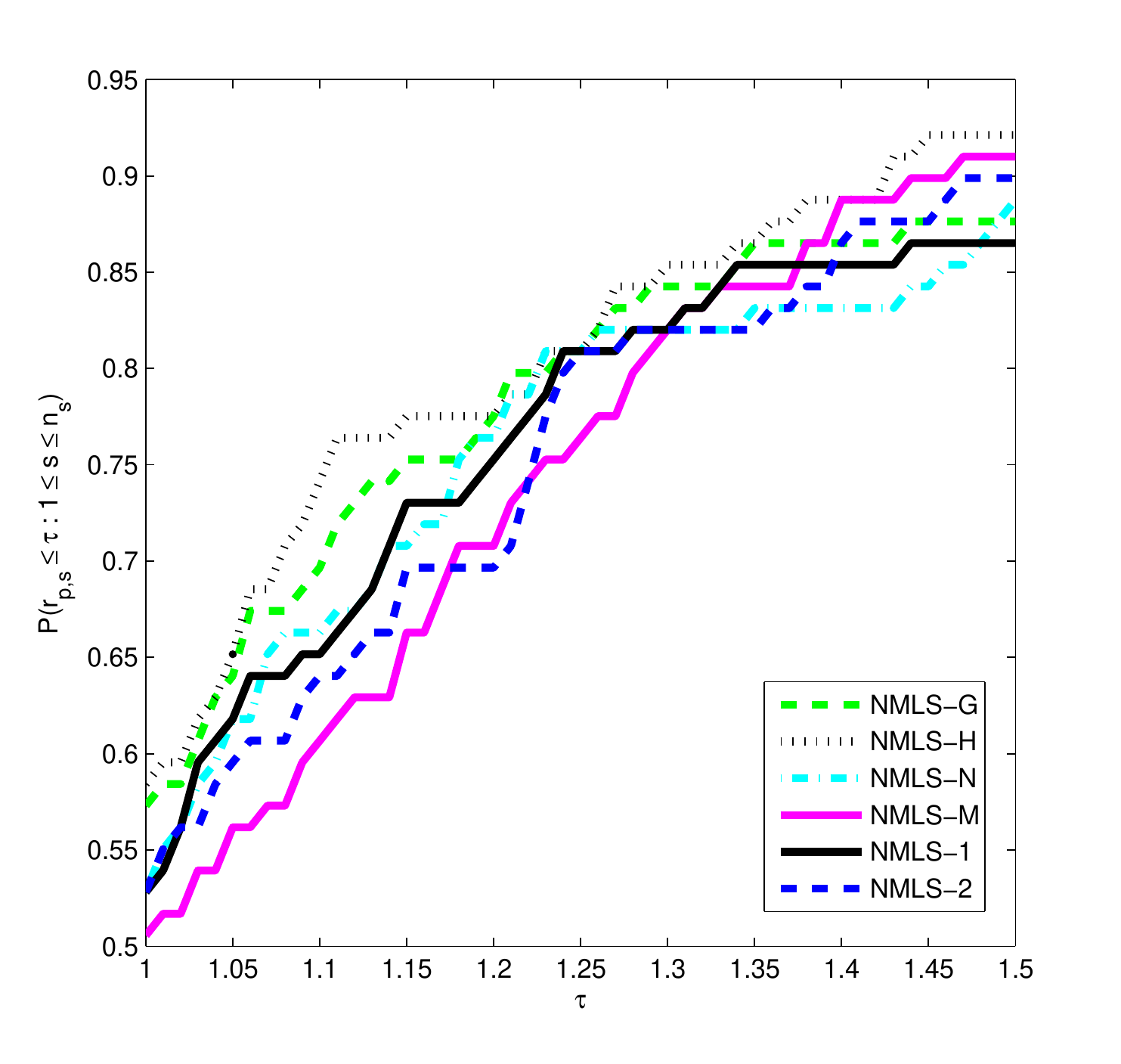}}% 
 \qquad 
 \subfloat[][$N_f  +  3 N_g$ performance profile ($\tau = 3$)]{\includegraphics[width=7.5cm]{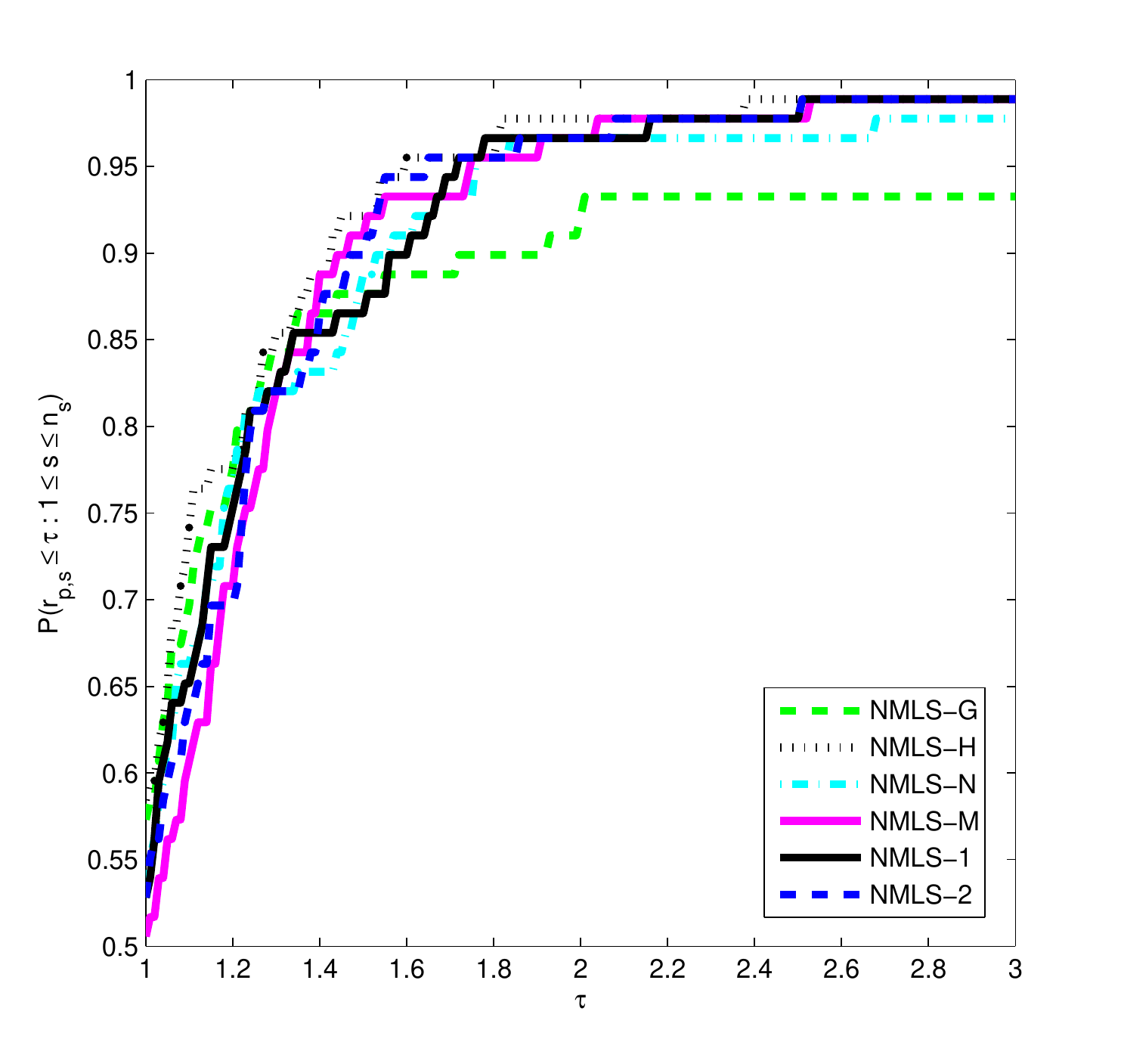}} 
 \caption{Performance profiles of all considered algorithms measured by: (a) the number of iterations ($N_i$) or gradient evaluations ($N_g$); (b) the number of function evaluations ($N_f$); (c) and (d) the hybrid measure $N_f  +  3 N_g$.} 
 \end{figure} 
 
\begin{figure}\label{com8} 
 \centering 
 \subfloat[][$N_i$ and $N_g$ performance profile]{\includegraphics[width=7.7cm]{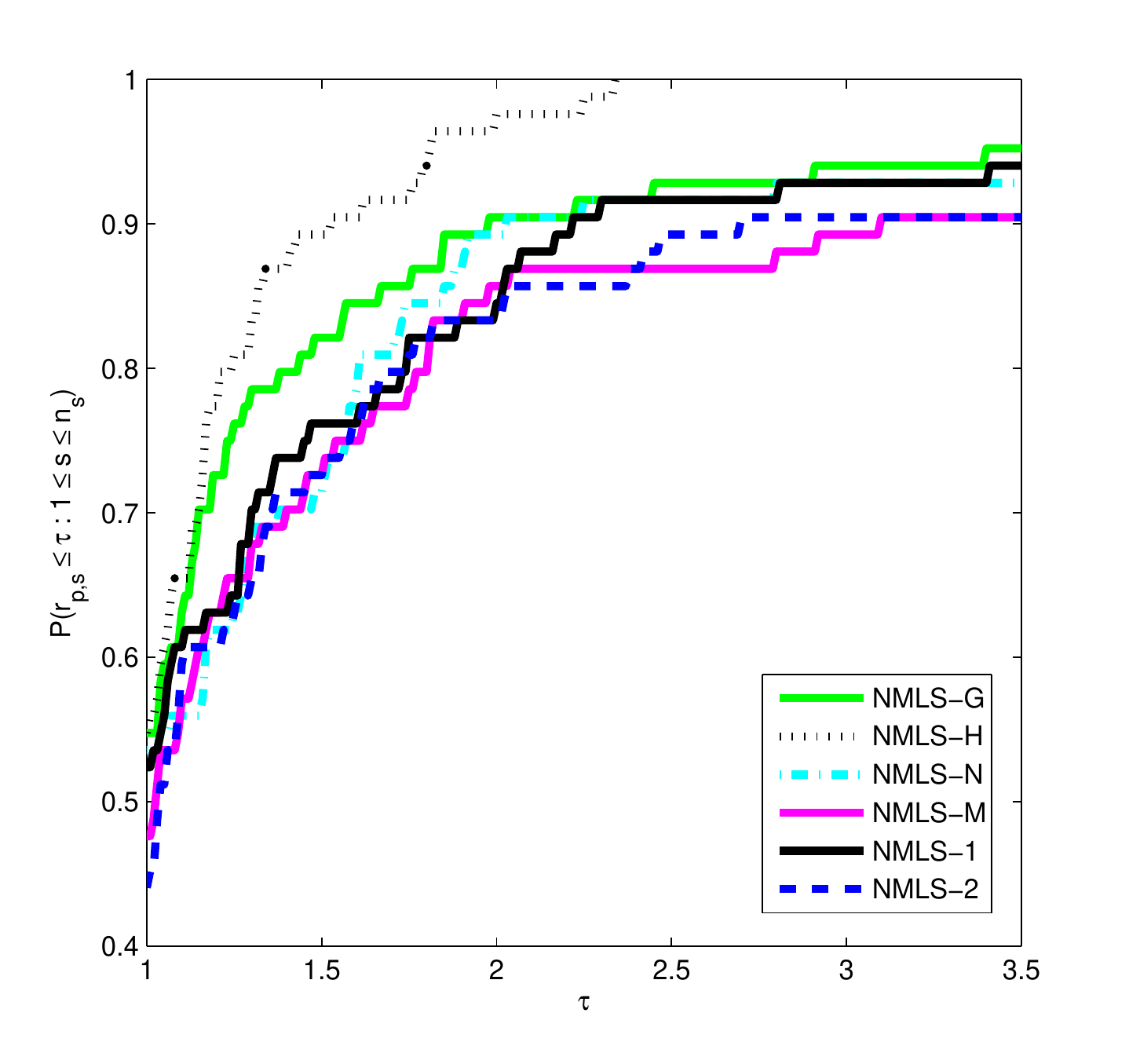}}% 
 \qquad 
 \subfloat[][$N_f$ performance profile]{\includegraphics[width=7.7cm]{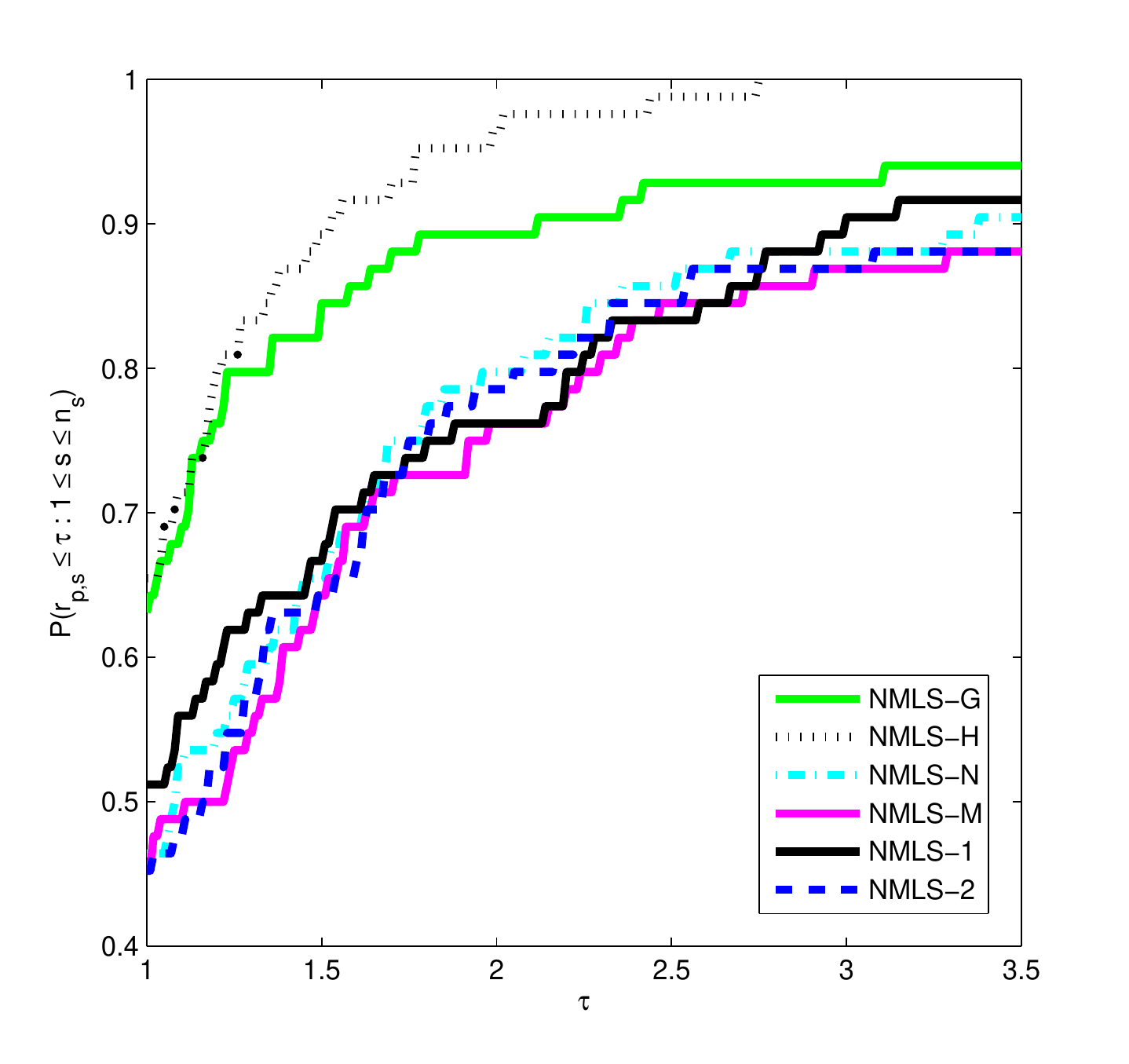}} 
 \qquad 
 \subfloat[][$N_f  +  3 N_g$ performance profile ($\tau = 1.5$)]{\includegraphics[width=7.7cm]{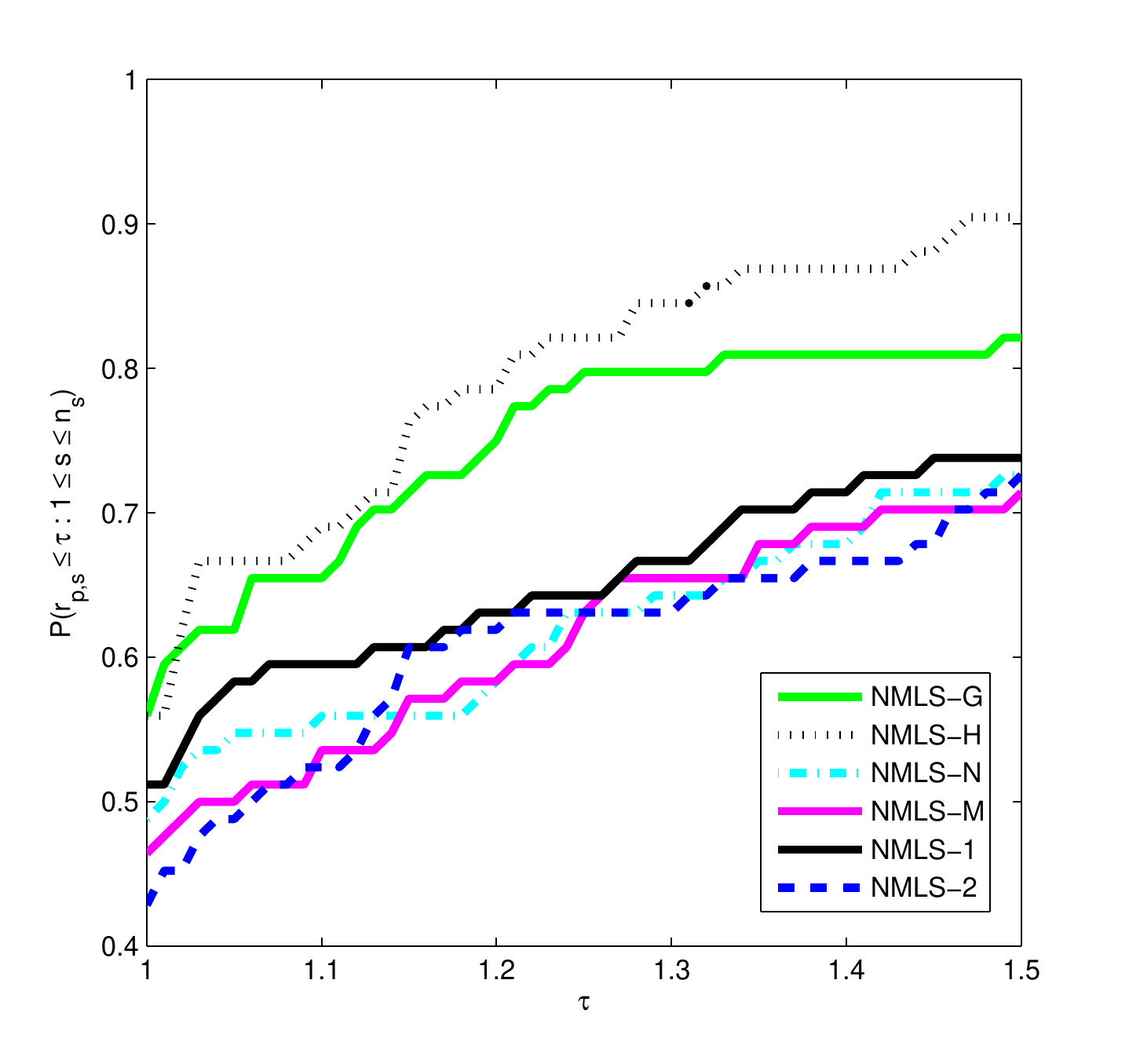}}% 
 \qquad 
 \subfloat[][$N_f  +  3 N_g$ performance profile ($\tau = 3.5$)]{\includegraphics[width=7.7cm]{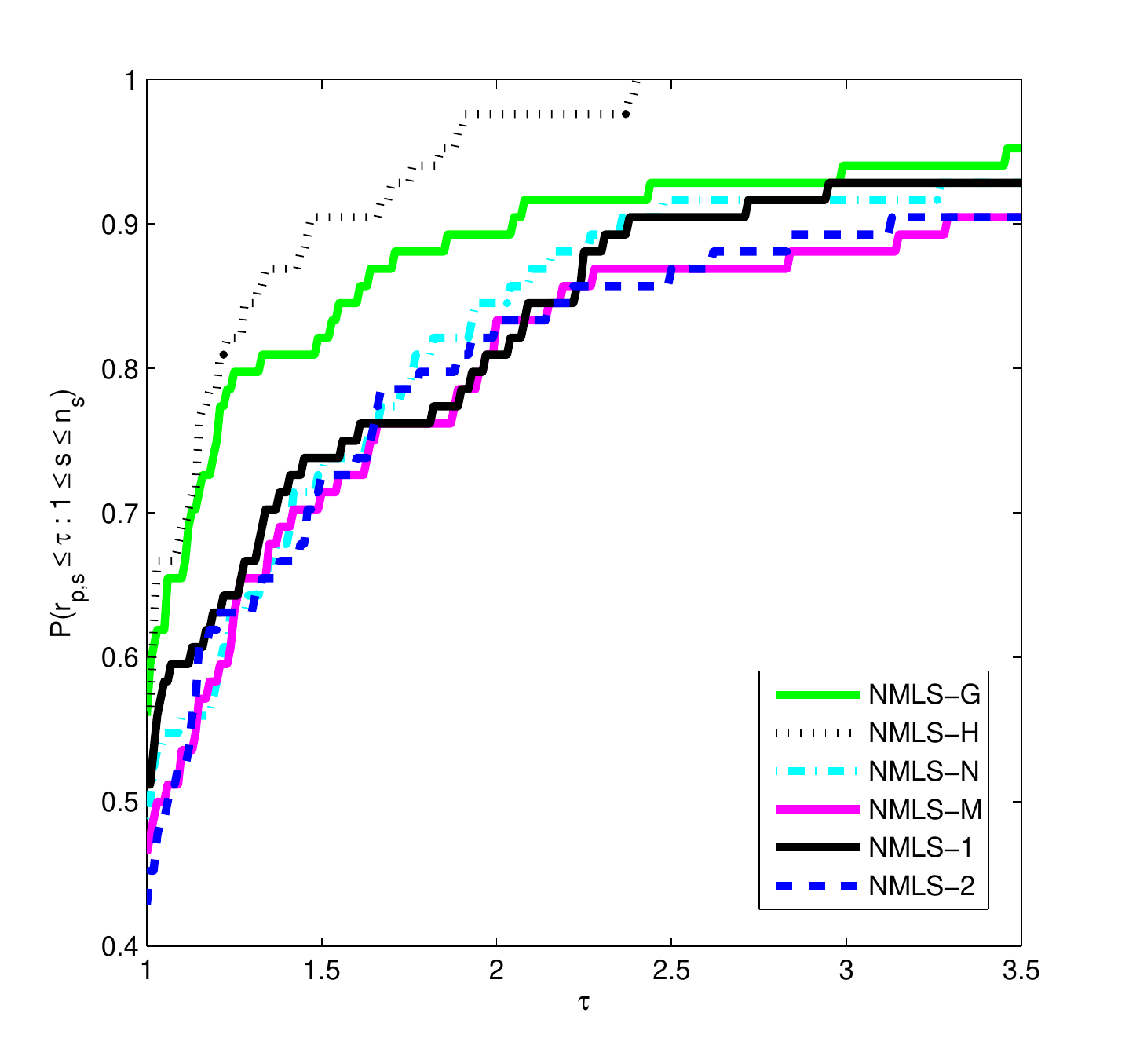}} 
 \caption{Performance profiles of all considered algorithms measured by: (a) the number of iterations ($N_i$) or gradient evaluations ($N_g$); (b) the number of function evaluations ($N_f$); (c) and (d) the hybrid measure $N_f  +  3 N_g$.} 
 \end{figure} 
 
Summarizing the results of this subsection, we see that the considered algorithms are comparable, however, NMLS-H and NMLS-G attain the best performance by using the direction (\ref{e.mdbb1}), while NMLS-H, NMLS-G and NMLS-1 outperform the others by employing the direction (\ref{e.mdbb2}). 

% ##########################################################################################################   
\subsection{Image deblurring/denoiding}
Image blur is a common problem that frequently happens in the photography and often can ruin the photograph. In digital photography, the motion blur is caused by camera shakes, which is unavoidable in many situations. Hence image deblurring/denoising is one of the fundamental tasks in the 
context of digital imaging processing, aiming at recovering an image from
a blurred/noisy observation. The problem is typically modelled as linear
inverse problem 
\begin{equation}\label{e.inv}
y = Ax + \omega,~~~ x \in X,
\end{equation}
where $X$ is a finite-dimensional vector space, $A$ is a blurring linear operator, $x$ is a clean image, $y$ is an 
observation, and $\omega$ is either Gaussian or impulsive noise. 

The system of equations (\ref{e.inv}) is mostly underdetermined and ill-conditioned, 
and $\omega$ is not commonly available, so one is not able to solve it directly, see \cite{BerB, Neu}. 
Hence, its solution is generally approximated by an optimization problem of the form  
\begin{equation}\label{e.least}
\begin{array}{ll}
\mathrm{minimize}      &~~ \frac{1}{2} \| Ax - y \|_2^2 + \lambda \varphi(x)\\
\mathrm{subject ~ to}  &~~ x \in X,
\end{array}
\end{equation}
where $\varphi$ is a smooth or nonsmooth regularizer such as $\varphi(x) = \frac{1}{2} \|x\|_2^2$,
$\varphi(x) = \|x\|_1$, $\varphi(x) = \|x\|_{ITV}$, or $\varphi(x) = \|x\|_{ATV}$ in which $\|.\|_{ITV}$ and $\|.\|_{ATV}$ denote isotropic and anisotropic total variation, for more information see \cite{Aho, ChaCCNP} and references therein. Among these regularizers, $\varphi(x) = \frac{1}{2} \|x\|_2^2$ is differentiable and the others are nondifferentiable. Therefore, by the aim of this paper to study differentiable objective function, we consider the next problem
\begin{equation}\label{e.l22l22}
\begin{array}{ll}
\mathrm{minimize}      &~~ \frac{1}{2} \| Ax - y \|_2^2 + \frac{\lambda}{2} \|Wx\|_2^2\\
\mathrm{subject ~ to}  &~~ x \in \mathbb{R}^n,
\end{array}
\end{equation}
where $A, W \in \mathbb{R}^{m \times n}$ and $y \in \mathbb{R}^m$. It is assumed that $A^T A + \lambda W^T W$ is positive definite, i.e., the problem (\ref{e.l22l22}) is a strictly convex problem and has the unique optimizer $x^* \in \mathbb{R}^n$ for an arbitrary vector $y$.

We now consider the recovery of the $256 \times 256$ blurred/noisy Lena image by minimizing the problem (\ref{e.l22l22}) using NMLS-G, NMLS-H, NMLS-N, NMLS-M, NMLS-1 and NMLS-2 with the search direction (\ref{e.mdbb2}). The algorithms stopped after 25 iterations. In particular, we choose the blurring matrix $A \in \mathbb{R}^{n \times n}$ to be the out-of-focus blur with radius 3 and the regularization matrix $W$ to be the gradient matrix for the problem (\ref{e.l22l22}). Thus, the matrix $W^T W$ is the two-dimensional discrete Laplacian matrix. For both matrices, we exploit the Neumann boundary conditions, which usually gives less artifacts at the boundary, see \cite{MorPC,NgCT}. The use of such boundary conditions means that $A^T A + \lambda W^T W$ is a block-Toeplitz-plus-Hankel matrix with Toeplitz-plus-Hankel blocks. The original and blurred/noisy version of Lena are demonstrated in Figure 10, and the recovered images by the considered algorithms of this image are depicted in Figure 12. 

\begin{figure}[h] \label{f.deb}
\centering
\subfloat[][Original image]{\includegraphics[width=7.5cm]{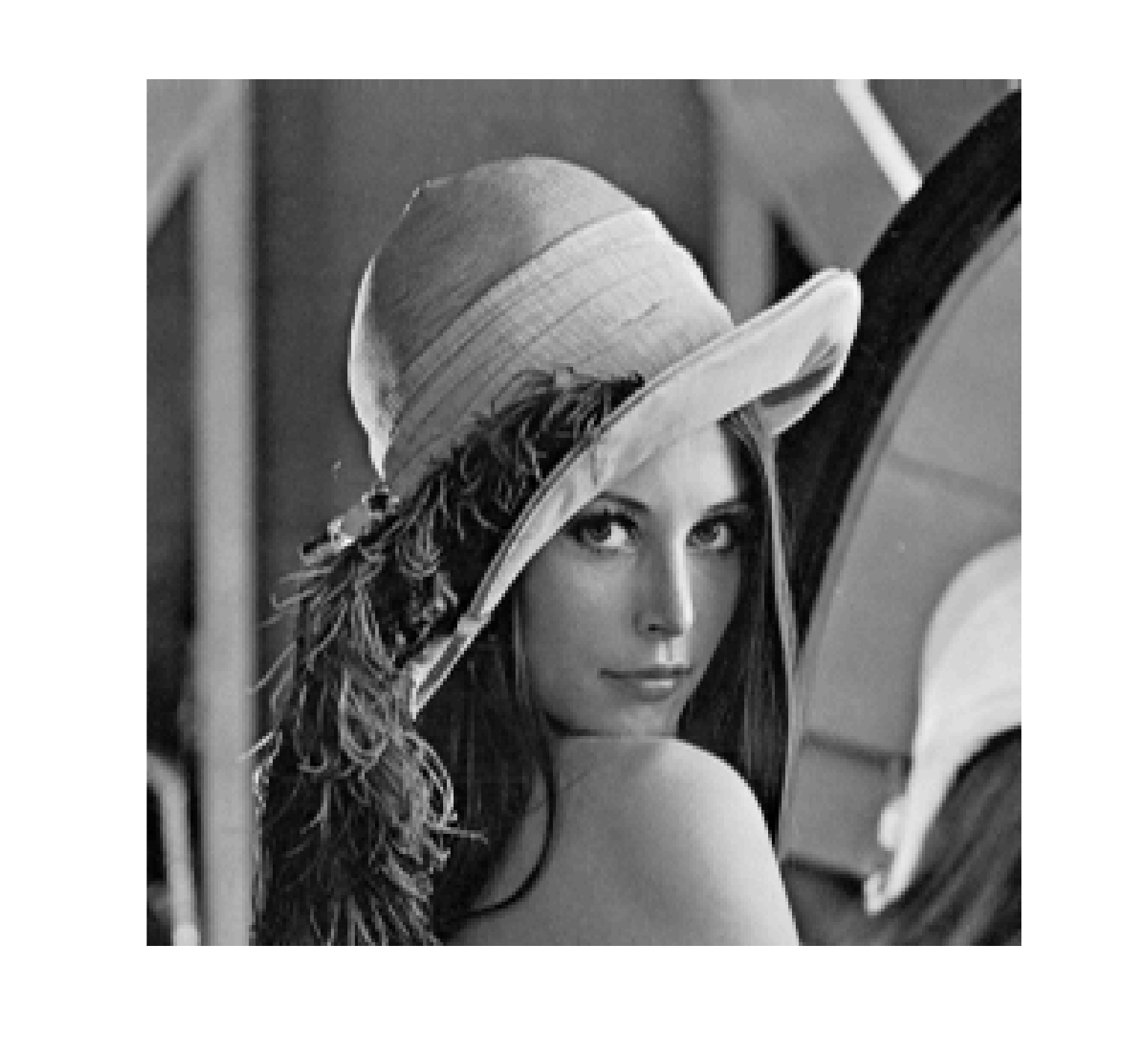}}%
\qquad
\subfloat[][Blurred/noisy image]{\includegraphics[width=7.5cm]{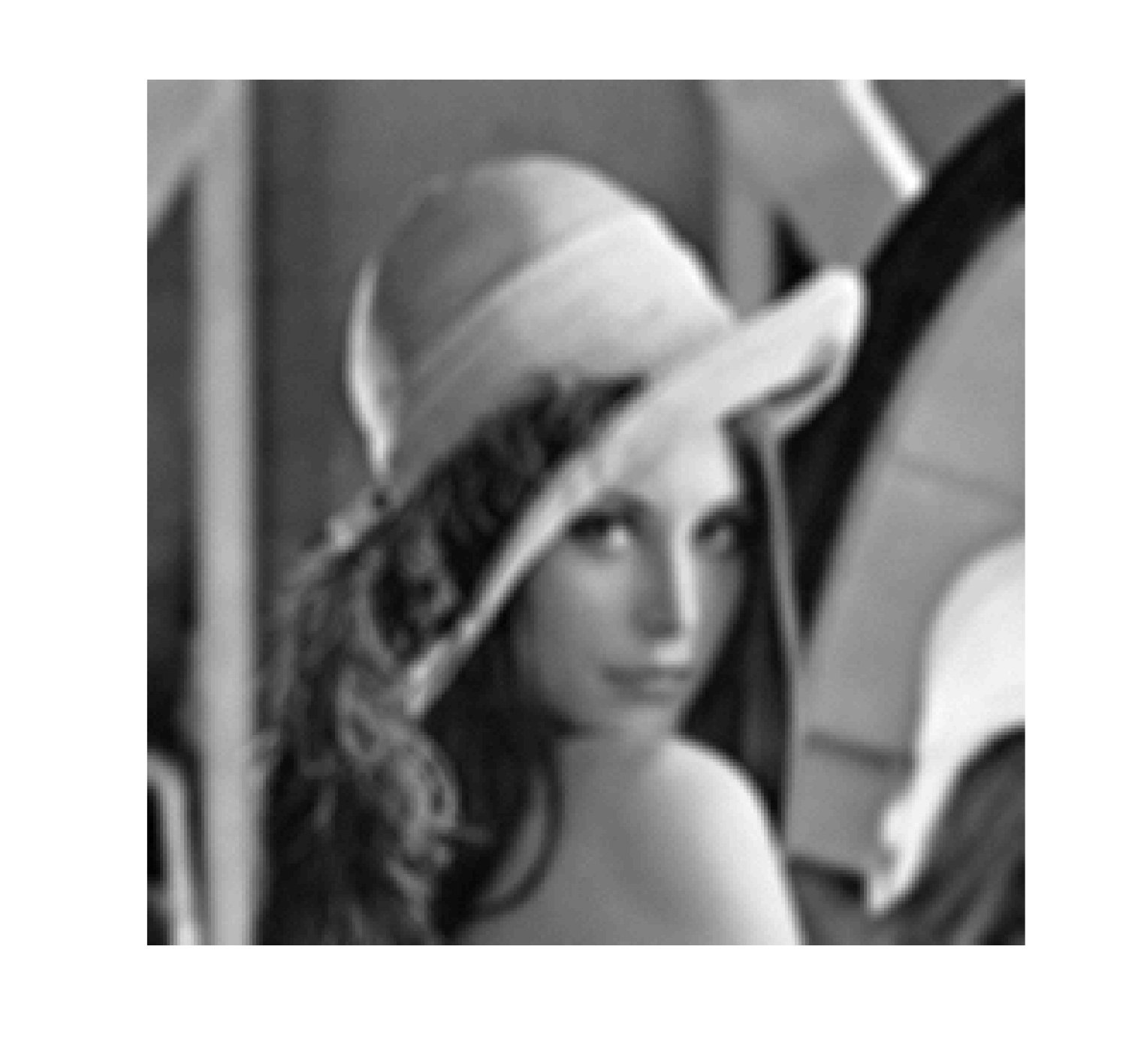}}
\caption{The $256 \times 256$ original and blurred/noisy Lena images}%
\end{figure}

To see details of this experiment, we compare the function values and signal-to-noise improvement (ISNR) for the algorithms in Figure 11, where ISNR is defined by
\[
\mathrm{ISNR} = 20 \log_{10} \left( \frac{\|y - x_0\|_2}{\|x_b - x_0\|_2} \right),
\]
where $y$ is the observed image. Generally, this ratio measures the quality of the restored image $x_b$ relative to the blurred/noisy observation $y$. The subfigure (a) of Figure 11 shows that the algorithms perform comparable, while the subfigure (b) of Figure 11 indicates that NMLS-1 and NMLS-2 outperform the other algorithms regarding ISNR.

\begin{figure}[h] \label{f.deb2}
\centering
\subfloat[][ rel. vs. iterations]{\includegraphics[width=7.5cm]{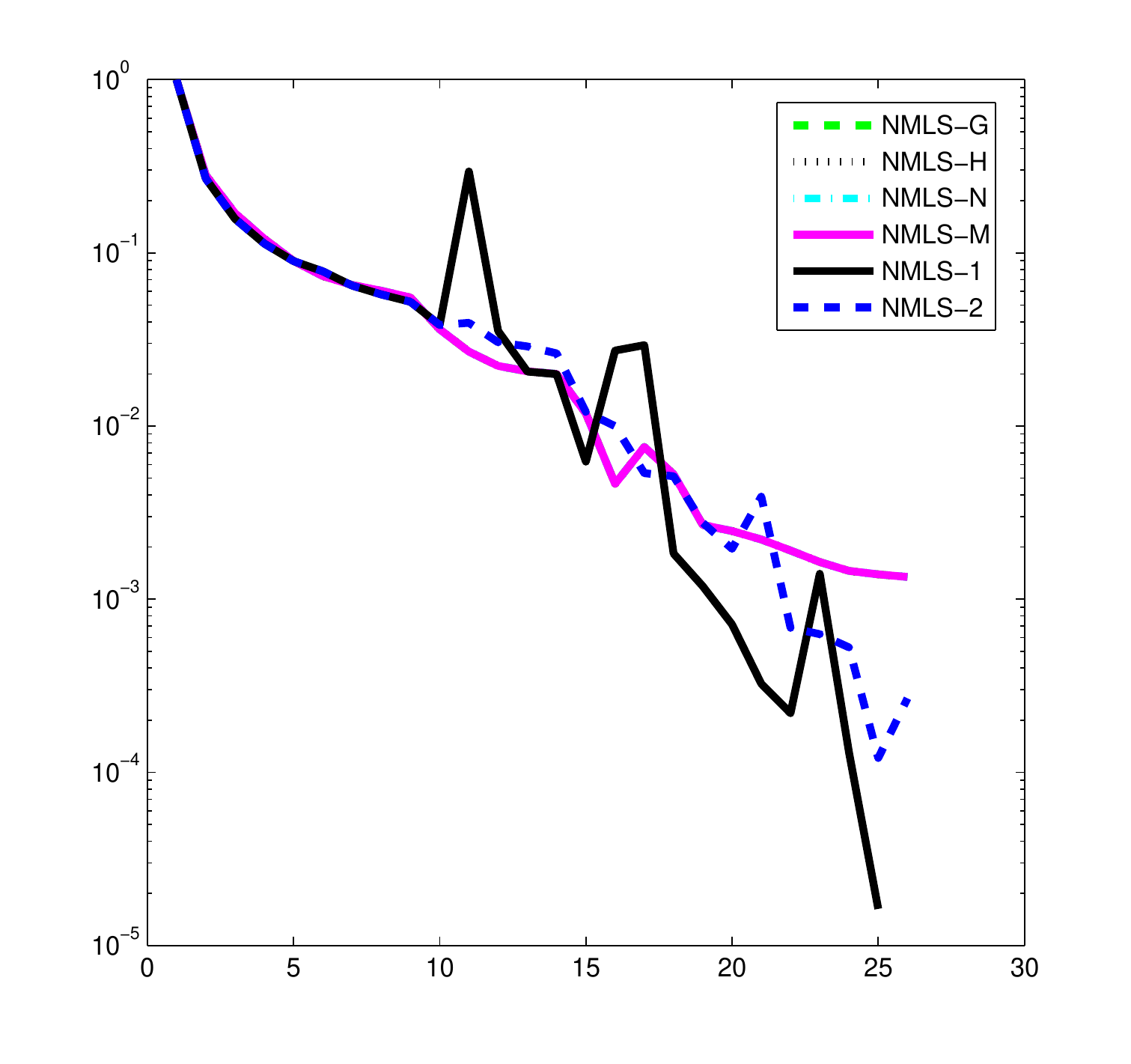}}%
\qquad
\subfloat[][ISNR vs. iterations]{\includegraphics[width=7.5cm]{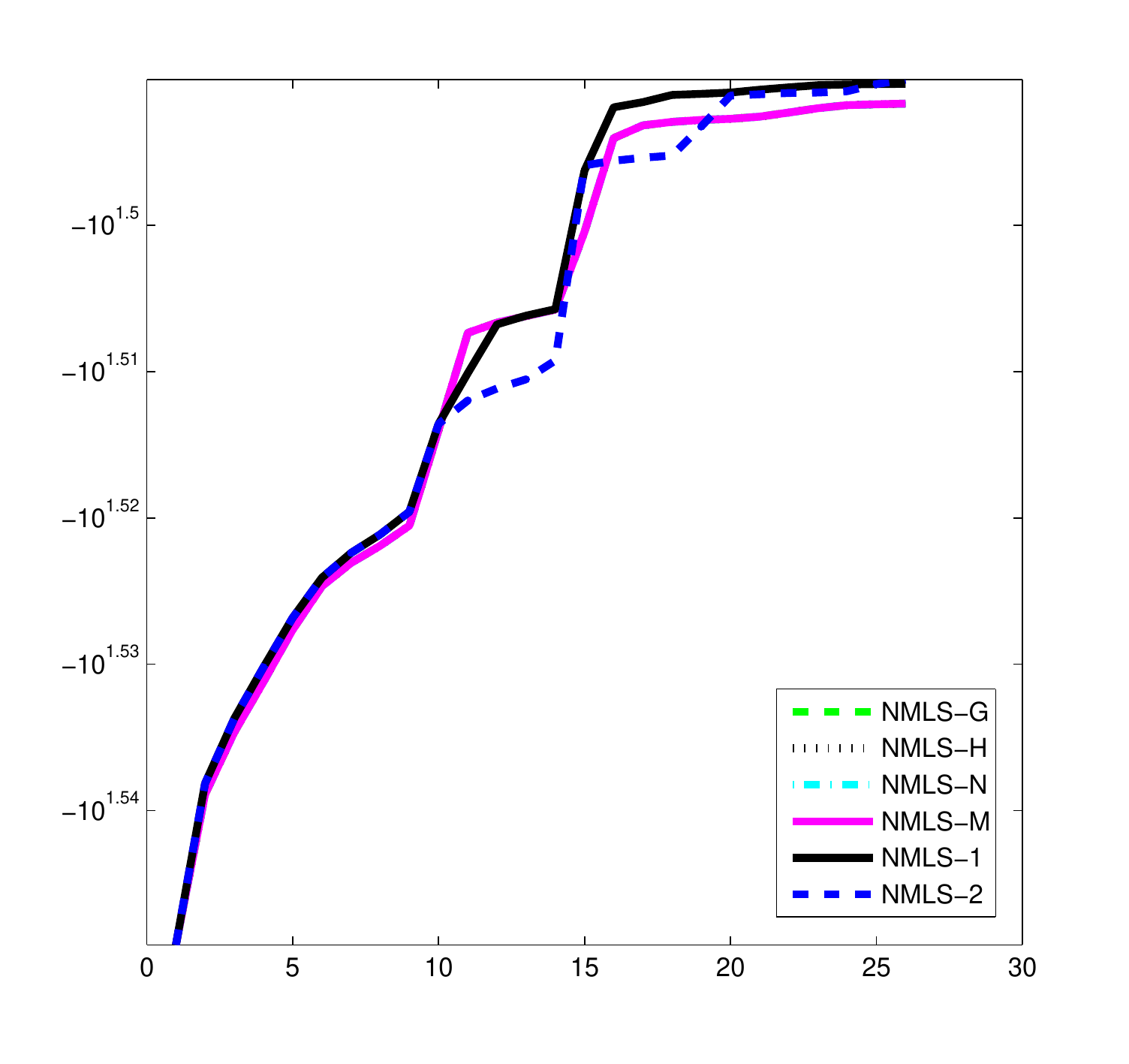}}
\caption{Deblurring of the $256 \times 256$ blurred/noisy image by NMLS-G, NMLS-H, NMLS-N, NMLS-M, NMLS-1 and NMLS-2, where $rel.~ := (f_k - f^*)/(f_0 - f^*)$ is the relative error of function values.}%
\end{figure}

From Figure 12, it is observed that the algorithms recover the image in acceptable quality, where the last function value and PSNR are also reported. Here peak signal-to-noise (PSNR) is defined by
\[
\mathrm{PSNR} = 20 \log_{10} \left( \frac{255 n}{\|x_b - x_0\|_2} \right),
\]
where $x_b$ is the approximated solution of (\ref{e.l22l22}) and $x_0$ is an initial point. This ratio is a common measure to assess the quality of the restored image $x_b$, i.e., it implies that NMLS-1 and NMLS-2 recover the blurred/noisy image better than the others.

\begin{figure}[h] \label{f.deb1}
\centering
\subfloat[][NMLS-G: $f = 50137.53,~ \mathrm{PSNR} = 29.79$]{\includegraphics[width=7.5cm]{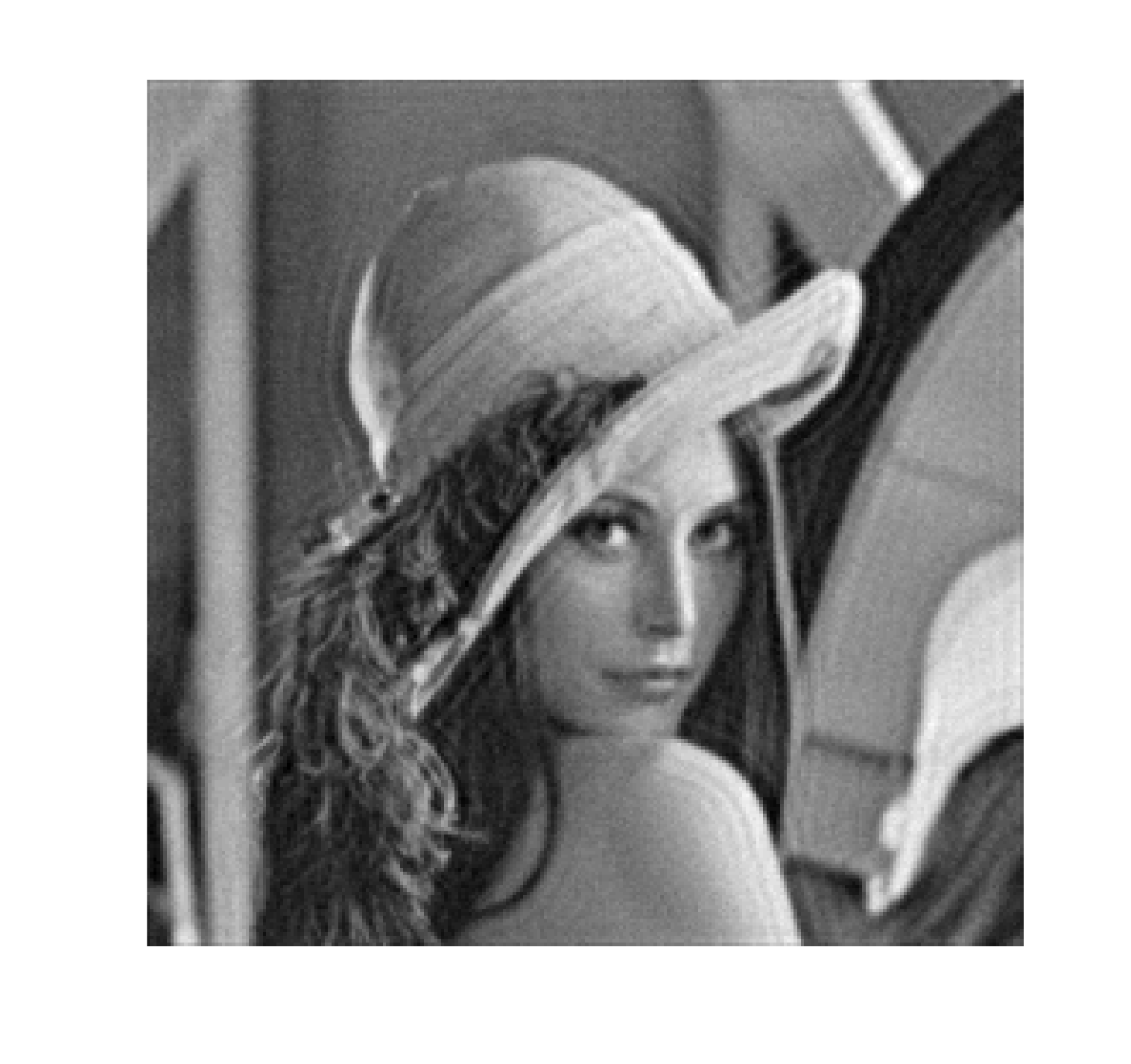}}%
\qquad
\subfloat[][NMLS-H: $f = 50137.53,~ \mathrm{PSNR} = 29.79$]{\includegraphics[width=7.5cm]{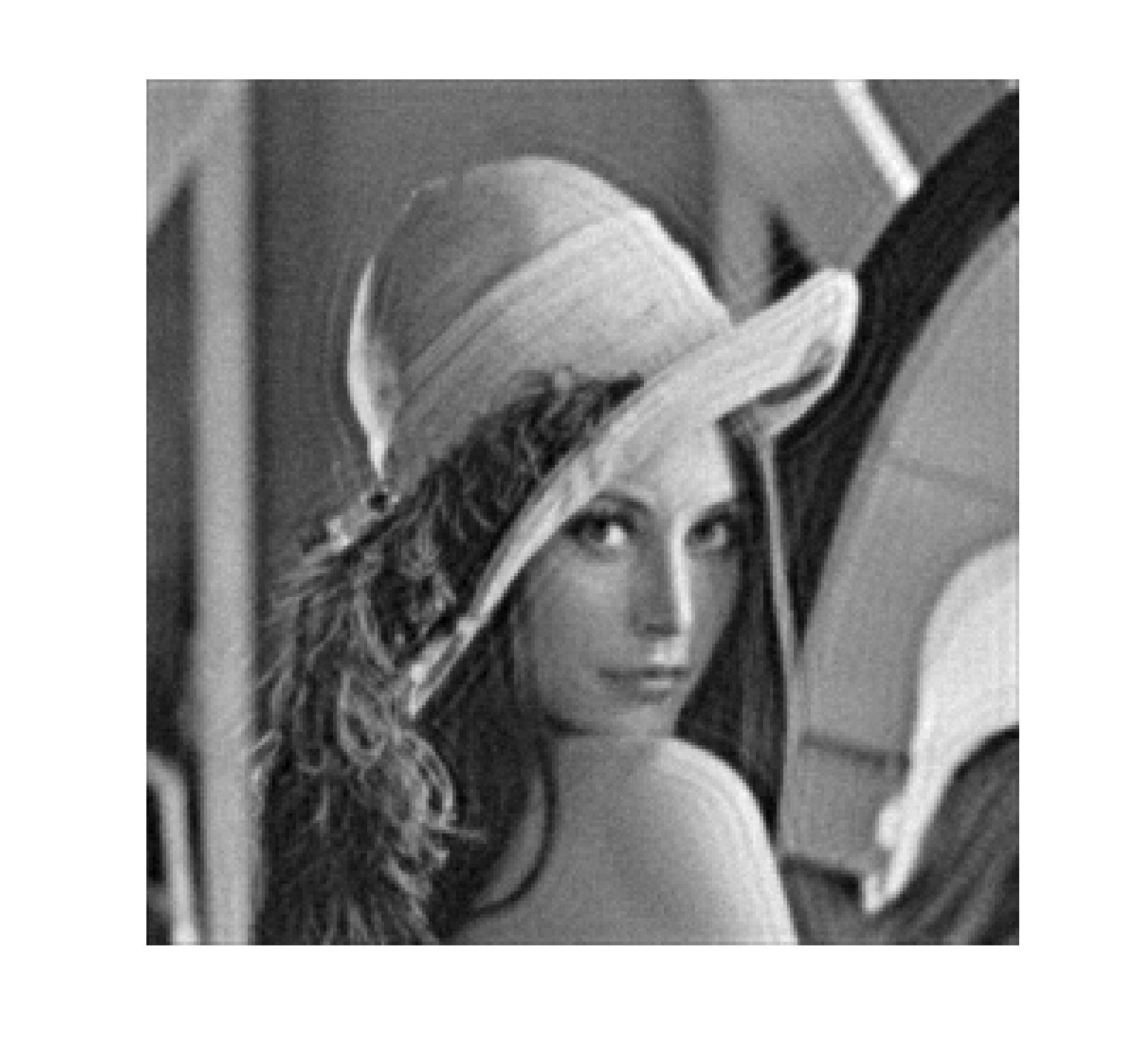}}
\qquad
\subfloat[][NMLS-N: $f = 50137.53,~ \mathrm{PSNR} = 29.79$]{\includegraphics[width=7.5cm]{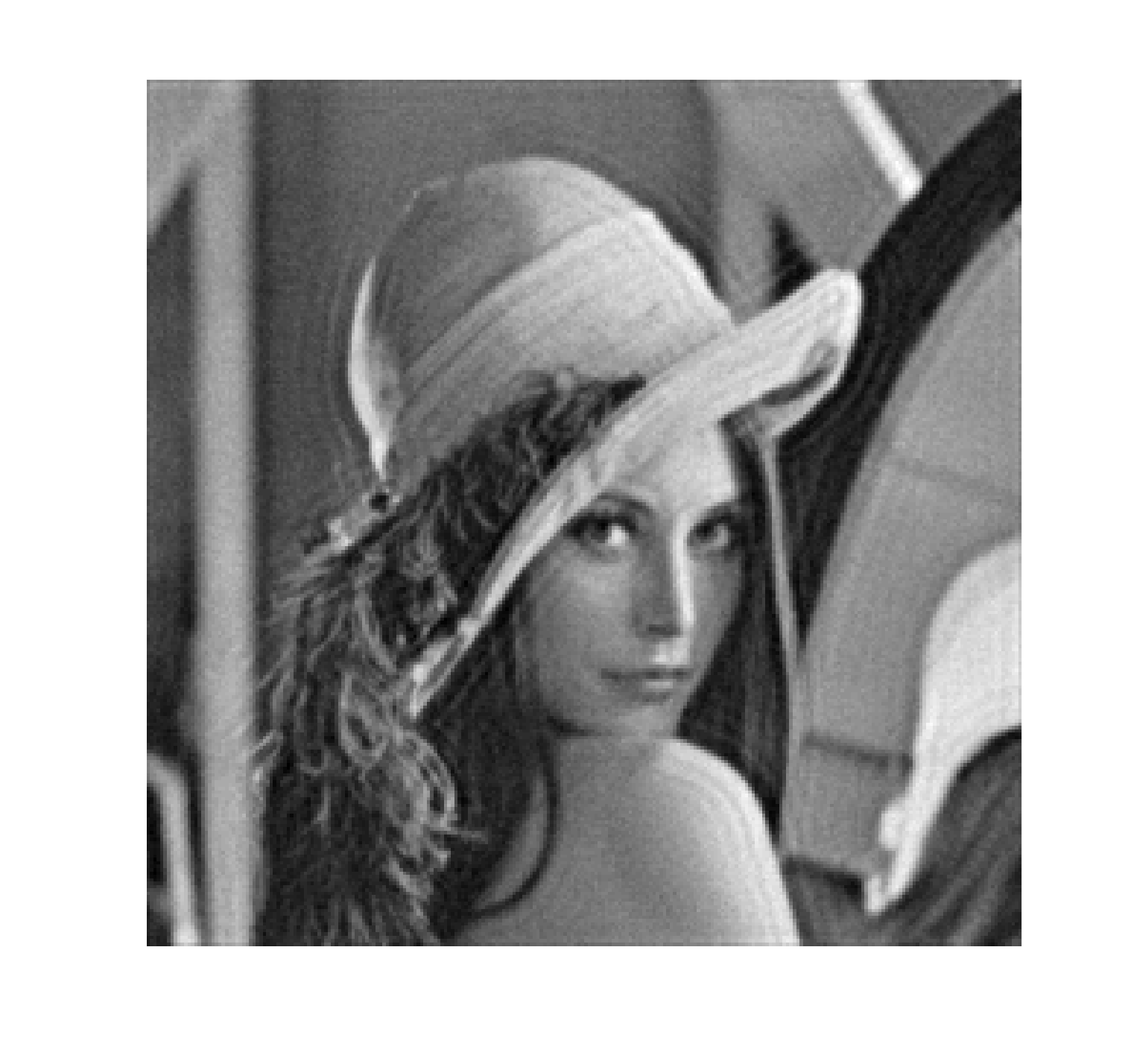}}%
\qquad
\subfloat[][NMLS-M: $f = 50137.53,~ \mathrm{PSNR} = 29.79$]{\includegraphics[width=7.5cm]{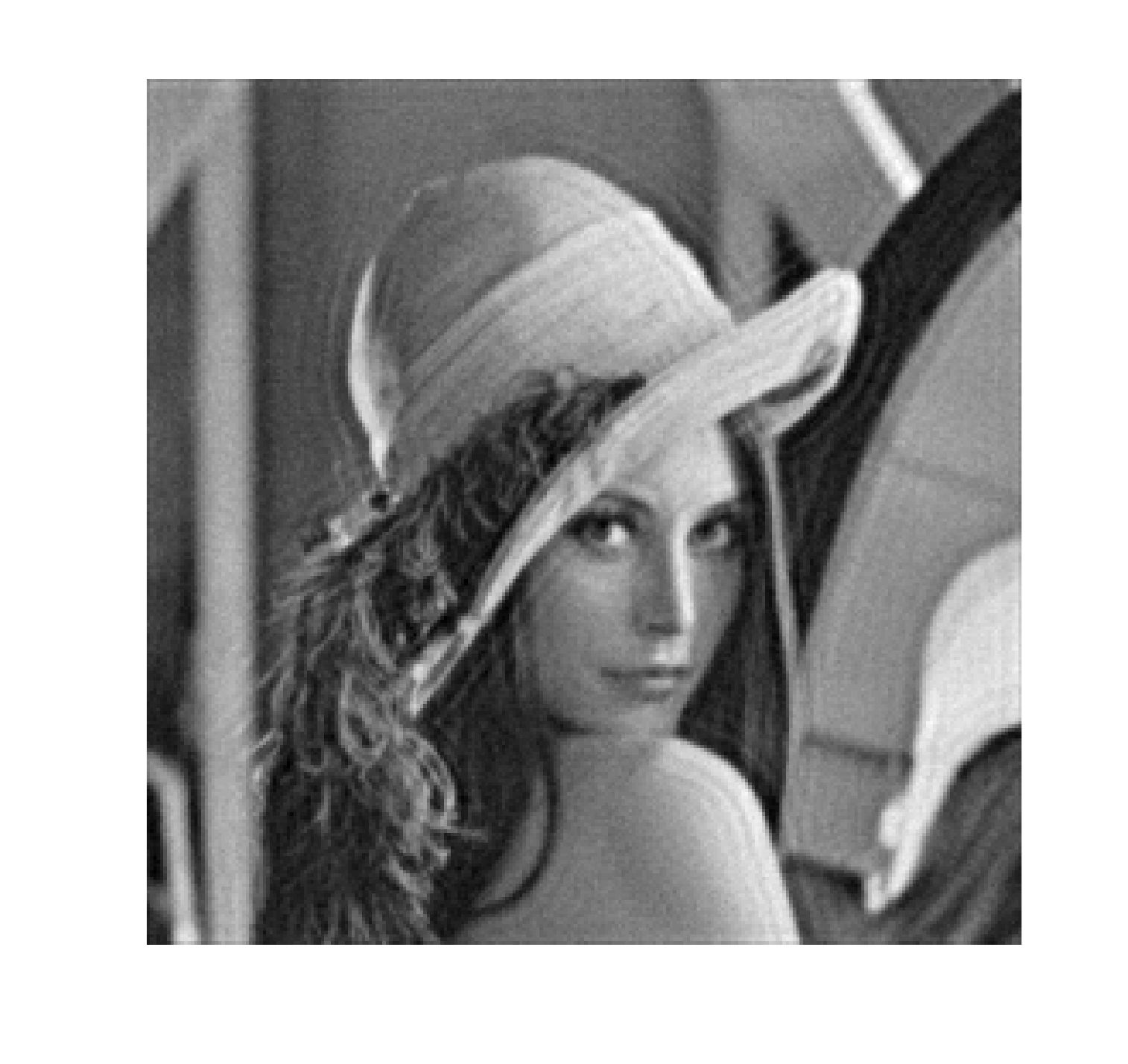}}
\qquad
\subfloat[][NMLS-1: $f = 49445.59,~ \mathrm{PSNR} = 29.89$]{\includegraphics[width=7.5cm]{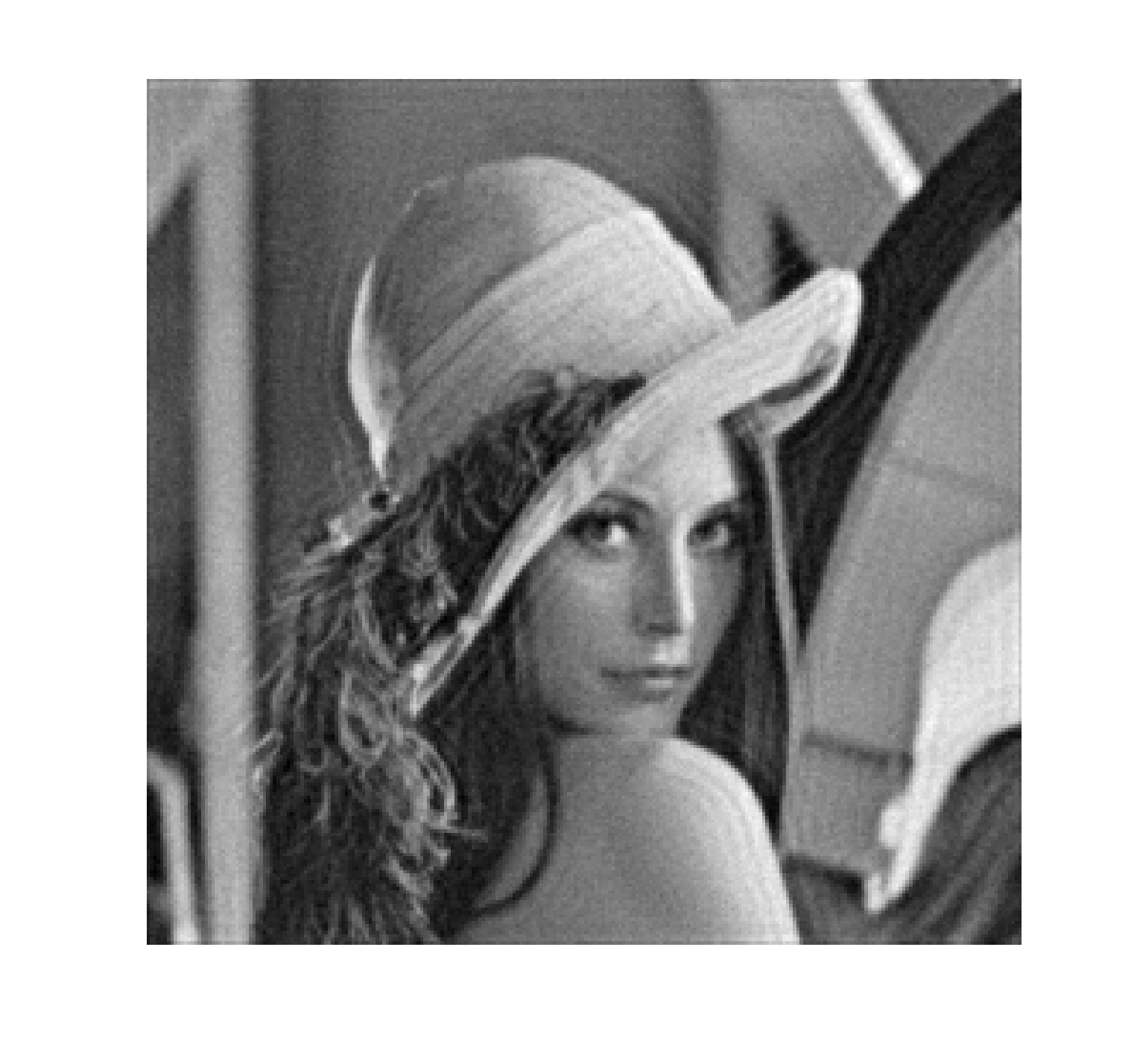}}%
\qquad
\subfloat[][NMLS-2: $f = 49600.29,~ \mathrm{PSNR} = 29.90$]{\includegraphics[width=7.5cm]{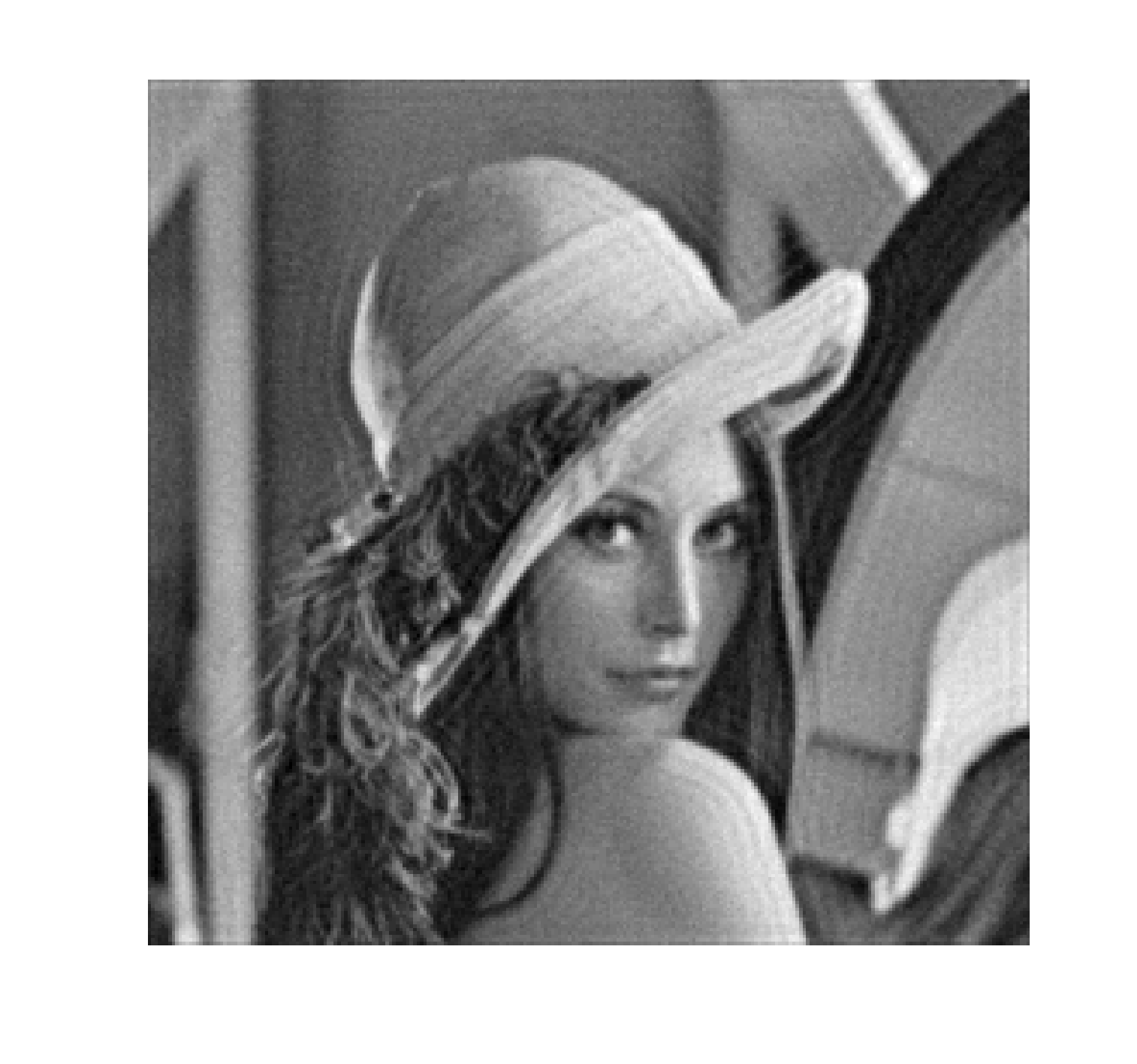}}
\caption{Deblurring of the $256 \times 256$ blurred/noisy image by NMLS-G, NMLS-H, NMLS-N, NMLS-M, NMLS-1 and NMLS-2. The algorithms were stopped after 25 iterations.}%
\end{figure}

%###########################################################################################################
%########################################################################################################### 
\section {Conclusions and perspectives}
This study first describes the motivation behind nonmonotone schemes, reviews the most popular nonmonotone terms and investigates the efficiency of them when they are incorporated into a backtracking
Armijo-type line search in presence of some search directions. In particular, we propose two novel nonmonotone terms, combine them with Armijo's rule and study their convergence. Afterwards, we report extensive numerical results and comparison among the two proposed nonmonotone schemes and some state-of-the-art nonmonotone line searches. The reported numerical results by using some measures of efficiency show that the performance of the considered nonmonotone line searches are varied depends on choosing search directions. Finally, employing the nonmonotone Armijo line searches for solving the dblurring problem produces acceptable results.

The experiments of this paper are limited to unconstrained optimization problems, however, the same experiments can be done for bound-constrained or general constrained optimization problems. We here consider only an Armijo-type line search, but one can investigate more numerical experiments with Wolfe-type or Goldestein-type line searches. For example studying of the behaviour of nonmonotone Wolfe-type line searches using conjugate-gradient directions is interesting. Furthermore, much more experiments on the parameters $ \rho $ and $ \delta $ for nonmonotone Armijo-type line searches can be done. It is also possible to extend our experiments to trust-region methods. One can consider many more applications with convex or nonconvex objective functions in the context of signal and image processing, machine learning, statistics and so on, which is out of the scope of this study. \\\\
\textbf{Acknowledgement.} We would like to thank {\sc Benedetta Morini} that generously makes the codes of the paper \cite{MorPC} available for us. \\\\
%###########################################################################################################
%###########################################################################################################
\textbf{Appendix: Tables 1--5}
\renewcommand{\arraystretch}{0.92} %Dampednewton
\begin{landscape} 
\small 
\begin{longtable}{llllllllllllllll} 
\multicolumn{5}{l} 
{Table 1. Numerical results with Newton's direction}\\ 
\cmidrule{1-14} 

\multicolumn{1}{l}{}        &\multicolumn{1}{l}{} 
&\multicolumn{1}{l}{NMLS-G}  &\multicolumn{1}{l}{}\hspace{10mm} 
&\multicolumn{1}{l}{NMLS-H} &\multicolumn{1}{l}{}\hspace{10mm} 
&\multicolumn{1}{l}{NMLS-N}  &\multicolumn{1}{l}{} \hspace{10mm} 
&\multicolumn{1}{l}{NMLS-M}  &\multicolumn{1}{l}{}\hspace{10mm} 
&\multicolumn{1}{l}{NMLS-1} &\multicolumn{1}{l}{}\hspace{10mm} 
&\multicolumn{1}{l}{NMLS-2}   &\multicolumn{1}{l}{} \\ 
\cmidrule(lr){3-4} \cmidrule(lr){5-6} \cmidrule(lr){7-8} \cmidrule(lr){9-10} \cmidrule(lr){11-12} \cmidrule(lr){13-14} 
 
\multicolumn{1}{l}{Problem name}      &\multicolumn{1}{l}{Dimension} 
 
&\multicolumn{1}{l}{$N_i$}   &\multicolumn{1}{l}{$N_f$}
&\multicolumn{1}{l}{$N_i$}   &\multicolumn{1}{l}{$N_f$}
&\multicolumn{1}{l}{$N_i$}   &\multicolumn{1}{l}{$N_f$}
&\multicolumn{1}{l}{$N_i$}   &\multicolumn{1}{l}{$N_f$}
&\multicolumn{1}{l}{$N_i$}   &\multicolumn{1}{l}{$N_f$}
&\multicolumn{1}{l}{$N_i$}   &\multicolumn{1}{l}{$N_f$} \\  
\cmidrule{1-14} 
\endfirsthead 
\multicolumn{5}{l}% 
{Table 1. Numerical results (\textit{continued})}\\[5pt] 
\hline 
\endhead 
\hline 
\endfoot 
\endlastfoot 
Beale  &2&17&25&14&27&17&25&14&27&13&22&16&28\\ 
Brown badly scaled  &2&7&113&7&113&7&113&7&113&7&113&7&113\\ 
Powell badly scaled &2&4&89&4&89&4&89&4&89&4&89&4&89\\ 
Variably Dim &2&7&8&7&8&7&8&7&8&7&8&7&8\\ 
Watson  &2&4&5&4&5&4&5&4&5&4&5&4&5\\
Box three-dim &3&8&9&8&9&8&9&8&9&8&9&8&9\\ 
Gaussian &3&1&2&1&2&1&2&1&2&1&2&1&2\\ 
Gulf r . and d. &3&39&46&43&47&39&46&37&47&41&55&39&46\\  
Helical valley &3&21&22&20&24&21&22&20&24&14&19&14&17\\ 
Bro . a . Dennis &4&11&66&11&66&11&66&11&66&11&66&11&66\\ 
E. Rosenbrock &4&11&16&15&23&11&16&15&23&18&27&15&22\\ 
E. Powell singular &4&15&16&15&16&15&16&15&16&15&16&15&16\\ 
Penalty I &4&16&17&16&17&16&17&16&17&16&17&16&17\\ 
Penalty  II &4&8&9&8&9&8&9&8&9&8&9&8&9\\ 
Trigonometric  &4&8&12&8&12&8&12&8&12&8&12&8&12\\ 
Wood  &4&29&33&27&31&29&33&31&43&33&48&29&38\\ 
Biggs EXP6 &6&13&18&17&24&23&33&15&34&15&34&21&31\\ 
Chebyquad &6&11&17&10&17&11&17&10&17&10&17&11&17\\ 
Penalty  II  &10&29&30&29&30&29&30&29&30&29&30&29&30\\ 
 
\hline
Average&&13.631&29.11&13.89&29.95&14.16&29.89&13.68&31.11&13.79&31.47&13.84&30.26\\
\hline
\end{longtable}

%=========================================================================================
\renewcommand{\arraystretch}{0.92} %BFGS
\small 
\begin{longtable}{llllllllllllllll} 
\multicolumn{5}{l} 
{Table 2. Numerical results with the BFGS direction}\\ 
\cmidrule{1-14} 
\multicolumn{1}{l}{}        &\multicolumn{1}{l}{} 
&\multicolumn{1}{l}{NMLS-G}  &\multicolumn{1}{l}{}\hspace{10mm} 
&\multicolumn{1}{l}{NMLS-H} &\multicolumn{1}{l}{}\hspace{10mm} 
&\multicolumn{1}{l}{NMLS-N}  &\multicolumn{1}{l}{} \hspace{10mm} 
&\multicolumn{1}{l}{NMLS-M}  &\multicolumn{1}{l}{}\hspace{10mm} 
&\multicolumn{1}{l}{NMLS-1} &\multicolumn{1}{l}{}\hspace{10mm} 
&\multicolumn{1}{l}{NMLS-2}   &\multicolumn{1}{l}{} \\ 
\cmidrule(lr){3-4} \cmidrule(lr){5-6} \cmidrule(lr){7-8} \cmidrule(lr){9-10} \cmidrule(lr){11-12} \cmidrule(lr){13-14} 
\multicolumn{1}{l}{Problem name}      &\multicolumn{1}{l}{Dimension} 
&\multicolumn{1}{l}{$N_i$}   &\multicolumn{1}{l}{$N_f$}
&\multicolumn{1}{l}{$N_i$}   &\multicolumn{1}{l}{$N_f$}
&\multicolumn{1}{l}{$N_i$}   &\multicolumn{1}{l}{$N_f$}
&\multicolumn{1}{l}{$N_i$}   &\multicolumn{1}{l}{$N_f$}
&\multicolumn{1}{l}{$N_i$}   &\multicolumn{1}{l}{$N_f$}
&\multicolumn{1}{l}{$N_i$}   &\multicolumn{1}{l}{$N_f$} \\ 
\cmidrule{1-14} 
\endfirsthead 
\multicolumn{5}{l}% 
{Table 2. Numerical results (\textit{continued})}\\[5pt] 
\hline 
\endhead 
\hline 
\endfoot 
\endlastfoot 
Beale  &2&17&24&17&24&24&24&17&24&16&24&17&24\\ 
Brown badly scaled  &2&18&60&14&57&43&68&48&72&48&72&48&72\\ 
Powell badly scaled &2&63&91&63&91&63&88&69&95&63&88&65&90\\ 
Variably Dim &2&5&13&5&13&5&13&5&13&5&13&5&13\\ 
Watson  &2&9&21&9&21&9&21&9&21&9&21&9&21\\
Box three-dim &3&30&39&30&39&30&39&30&39&30&39&30&39\\  
Gaussian &3&3&6&3&6&3&6&3&6&3&6&3&6\\ 
Gulf r . and d. &3&30&40&33&44&30&40&33&44&33&44&33&44\\ 
Helical valley &3&50&80&43&73&50&80&38&68&33&60&31&53\\ 
Bro . a . Dennis &4&26&80&24&79&26&80&24&79&24&79&24&79\\ 
E. Rosenbrock &4&73&89&72&94&73&89&62&85&56&85&70&96\\ 
E. Powell singular  &4&31&51&36&57&31&51&36&57&24&47&36&57\\ 
Penalty I &4&1258&1356&543&659&462&522&479&552&482&546&482&546\\ 
Penalty  II &4&14&21&14&21&14&21&14&21&12&20&14&21\\ 
Trigonometric  &4&14&15&13&15&14&15&13&15&13&15&13&15\\ 
Wood  &4&33&68&30&65&33&68&33&68&30&65&33&68\\ 
Biggs EXP6 &6&45&58&39&47&45&58&43&58&43&58&45&60\\ 
Chebyquad &6&19&32&19&32&19&32&18&32&19&32&18&33\\ 
Penalty  II  &10&1492&1682&1412&1727&509&786&524&781&514&767&648&982\\ 
\hline
Average&&248.46&294.31&186.08&243.38&114.08&161.62&115.23&163.85&112.08&160.08&124.92&178.38\\         
 \hline 
 \end{longtable} 
 \end{landscape}
 
 %===================================================================================================== 

\renewcommand{\arraystretch}{0.97} 
\begin{landscape} 
\small 
\begin{longtable}{llllllllllllllll} 
\multicolumn{5}{l} 
{Table 3. Numerical results with the LBFGS direction}\\ 
\cmidrule{1-14} 
\multicolumn{1}{l}{}        &\multicolumn{1}{l}{} 
&\multicolumn{1}{l}{NMLS-G}  &\multicolumn{1}{l}{}\hspace{10mm} 
&\multicolumn{1}{l}{NMLS-H} &\multicolumn{1}{l}{}\hspace{10mm} 
&\multicolumn{1}{l}{NMLS-N}  &\multicolumn{1}{l}{} \hspace{10mm} 
&\multicolumn{1}{l}{NMLS-M}  &\multicolumn{1}{l}{}\hspace{10mm} 
&\multicolumn{1}{l}{NMLS-1} &\multicolumn{1}{l}{}\hspace{10mm} 
&\multicolumn{1}{l}{NMLS-2}   &\multicolumn{1}{l}{} \\ 
\cmidrule(lr){3-4} \cmidrule(lr){5-6} \cmidrule(lr){7-8} \cmidrule(lr){9-10} \cmidrule(lr){11-12} \cmidrule(lr){13-14} 
\multicolumn{1}{l}{Problem name}      &\multicolumn{1}{l}{Dimension} 

&\multicolumn{1}{l}{$N_i$}   &\multicolumn{1}{l}{$N_f$} %&\multicolumn{1}{l}{$\mathrm{T}$}
&\multicolumn{1}{l}{$N_i$}   &\multicolumn{1}{l}{$N_f$}
&\multicolumn{1}{l}{$N_i$}   &\multicolumn{1}{l}{$N_f$} 
&\multicolumn{1}{l}{$N_i$}   &\multicolumn{1}{l}{$N_f$} %&\multicolumn{1}{l}{$\mathrm{T}$} 
&\multicolumn{1}{l}{$N_i$}   &\multicolumn{1}{l}{$N_f$}% &\multicolumn{1}{l}{$\mathrm{T}$} 
&\multicolumn{1}{l}{$N_i$}   &\multicolumn{1}{l}{$N_f$}\\ %&\multicolumn{1}{l}{$\mathrm{T}$} \\ 
\cmidrule{1-14} 
\endfirsthead 
\multicolumn{5}{l}% 
{Table 3. Numerical results (\textit{continued})}\\[5pt] 
\hline 
\endhead 
\hline 
\endfoot 
\endlastfoot 
Beale &2&21&26&21&26&21&26&21&26&21&26&21&26\\ 
Brown badly scaled  &2&11&15&11&15&11&15&11&15&11&15&11&15\\ 
Full Hess .  FH1 &2&31&37&31&36&31&39&31&36&31&36&31&36\\ 
Full Hess .  FH2 &2&7&10&7&10&7&10&7&10&7&10&7&10\\ 
Gaussian  &3&4&7&4&7&4&7&4&7&4&7&4&7\\ 
Box three-dim  &3&49&68&54&61&46&65&42&62&51&63&42&62\\ 
Gulf r .  and d.  &3&64&85&64&87&61&84&68&94&64&90&64&90\\ 
Staircase 1   &4&10&15&10&15&10&15&10&15&10&15&10&15\\ 
Staircase 2  &4&10&15&10&15&10&15&10&15&10&15&10&15\\ 
Bro .  a .  Dennis  &4&31&50&31&50&31&50&31&50&31&50&31&50\\ 
wood  &4&66&92&90&119&78&101&49&68&64&102&48&68\\ 
Biggs EXP66  &6&72&117&67&93&70&105&70&116&62&92&69&115\\ 
Penalty II	  &10&767&1561&525&1155&420&1068&541&1255&430&1012&538&1215\\ 
Variably Dim  &10&15&37&15&37&15&37&15&37&15&37&15&37\\ 
E .  Powell Sin &16&62&75&85&162&61&72&61&72&57&69&57&69\\ 
Watson  &31&1704&2037&1813&2377&754&1033&1844&2480&788&1144&764&1039\\ 
HARKERP2  &50&8917&9824&50001&1461198&9722&27035&9430&15756&8233&13643&9521&12691\\ 
ARGLINB  &100&11&49&11&49&3&25&3&25&3&25&3&25\\ 
Diagonal 3  &100&55&62&55&62&55&62&55&62&55&62&55&62\\ 
E .  Rosenbrock  &100&98&152&91&174&90&131&82&179&81&178&81&178\\ 
Diagonal 2 &500&97&98&97&98&97&98&97&98&97&98&85&87\\ 
DIXON3DQ &1000&2016&2019&2016&2019&1739&1772&1969&1979&2001&2016&2001&2016\\ 
E .  Beal &1000&13&16&13&16&13&16&13&16&13&16&13&16\\ 
Fletcher  &1000&21076&22370&4353&4626&4354&4710&4336&4843&3930&4447&3930&4447\\ 
G .  Rosenbrock  &1000&14358&21530&10241&15876&10438&16115&8516&14717&8376&14871&8375&14811\\ 
Par .  Pert .  Quadratic &1000&149&163&149&163&144&160&152&168&144&160&144&160\\ 
 Hager &1000&42&46&42&46&42&46&42&46&42&46&42&46\\ 
 BDQRTIC &1000&171&237&154&188&246&942&305&1889&196&727&196&727\\ 
BG2 &1000&23&34&23&34&23&34&23&34&23&34&23&34\\ 
POWER  &1000&9391&9412&9041&9063&8932&9029&7451&7518&8618&8706&8618&8706\\ 
Purt .  Trid. Quadratic &5000&579&592&579&592&577&595&558&576&581&603&581&603\\ 
Al .  Pert .  Quadratic  &5000&551&564&551&564&586&604&604&620&549&566&549&566\\ 
Pert .  Quadratic diag   &5000&684&847&836&1049&807&1146&728&931&739&980&756&1001\\ 
E .  Hiebert  &5000&2368&6491&2103&4723&2224&5184&2040&5119&2017&5229&1987&5067\\ 
Fletcher   &5000&12792&15151&50001&51822&25614&29246&18879&21076&17310&19750&17310&19750\\ 
E .  Trid .  1  &5000&22&25&22&25&22&25&22&25&22&25&22&25\\ 
NONSCOMP  &5000&928&1852&2235&4504&1778&4391&1653&3832&1566&3684&1878&4596\\ 
LIARWHD  &5000&63&114&71&91&60&112&41&65&41&65&42&67\\ 
CUBE  &5000&2722&5079&9060&18506&8910&22743&7743&18930&8433&20787&8433&20787\\ 
TRIDIA  &5000&1904&1917&1904&1917&1673&1695&1901&1927&1641&1666&1641&1666\\ 
DIXMAANA  &9000&10&14&10&14&10&14&10&14&10&14&10&14\\ 
DIXMAANB   &9000&8&12&8&12&8&12&8&12&8&12&8&12\\ 
DIXMAANC   &9000&9&14&9&14&9&14&9&14&9&14&9&14\\ 
DIXMAAND  &9000&10&16&10&16&10&16&10&16&10&16&10&16\\ 
DIXMAANE  &9000&376&380&376&380&390&403&382&387&382&387&382&387\\ 
DIXMAANF &9000&268&272&268&272&251&258&261&268&261&268&261&268\\ 
DIXMAANG &9000&229&234&229&234&229&237&229&235&223&230&223&230\\ 
DIXMAANH &9000&223&229&223&229&233&243&229&236&229&236&229&236\\ 
DIXMAANI  &9000&2798&2802&2798&2802&2653&2698&2801&2818&2751&2780&2751&2780\\ 
DIXMAANJ   &9000&420&424&420&424&431&437&433&439&433&440&433&440\\ 
DIXMAANK    &9000&361&366&361&366&377&385&361&366&361&366&361&366\\ 
DIXMAANL   &9000&321&327&321&327&340&350&321&327&321&327&321&327\\ 
ARWHEAD &10000&13&30&13&30&13&30&13&30&13&30&13&30\\ 
BDEXP  &10000&34&35&34&35&34&35&34&35&34&35&34&35\\ 
Broyden Tridiagonal  &10000&50&55&53&59&50&55&50&55&53&59&50&55\\ 
COSIN   &10000&26&28&15&19&26&28&15&19&15&19&15&19\\ 
Diagonal 4 &10000&6&13&6&13&6&13&6&13&6&13&6&13\\ 
Diagonal 5   &10000&5&6&5&6&5&6&5&6&5&6&5&6\\ 
Ext .  Tridiagonal 2  &10000&29&32&29&32&30&36&29&32&29&32&29&32\\ 
Diagonal 7  &10000&5&7&5&7&5&7&5&7&5&7&5&7\\ 
Diagonal 8  &10000&5&8&5&8&5&8&5&8&5&8&5&8\\ 
DQDRTIC  &10000&10&19&10&19&10&19&10&19&10&19&10&19\\ 
ENGVAL1 &10000&45&74&37&51&37&51&37&51&37&54&27&42\\ 
EDENSCH &10000&22&27&22&27&22&27&22&27&22&27&22&27\\ 
Ext .  BD1 &10000&14&17&13&15&14&17&12&16&13&15&12&16\\ 
Ext .  DENSCHNB &10000&6&9&6&9&6&9&6&9&6&9&6&9\\ 
Ext .  DENSCHNF  &10000&14&23&14&23&14&23&14&23&14&23&14&23\\ 
Ext .  Himmelblau &10000&9&16&9&16&9&16&9&16&9&17&9&16\\ 
Ext .  Maratos  &10000&219&558&221&452&228&473&192&436&188&424&205&479\\ 
Ext .  Powell   &10000&81&94&91&168&81&94&74&85&71&83&71&83\\ 
Ext .  Penalty  &10000&111&160&173&353&115&137&115&137&115&137&115&137\\ 
Ext .  PSC1   &10000&11&17&11&17&11&17&11&17&11&17&11&17\\ 
Ext .  QP1 &10000&25&41&25&41&25&41&25&41&25&41&25&41\\ 
Ext .   QP2  &10000&132&199&88&121&107&200&110&199&99&189&99&189\\ 
Ext .   TET  &10000&9&13&9&13&9&13&9&13&9&13&9&13\\ 
Ext . Wood  &10000&68&94&90&115&79&102&50&69&66&104&50&70\\ 
Ext .  White a .  Holst &10000&83&123&53&86&81&116&73&134&68&117&61&95\\ 
Ext .   Tridiagonal 1 &10000&22&25&22&25&22&25&22&25&22&25&22&25\\ 
  FH3  &10000&4&19&4&19&4&19&4&19&4&19&4&19\\ 
G .  PSC1  &10000&45&52&45&52&45&52&45&52&45&52&45&52\\ 
G .  Tridiagonal &10000&20&24&20&24&20&24&20&24&20&24&20&24\\ 
HIMMELBG   &10000&2&3&2&3&2&3&2&3&2&3&2&3\\ 
NONDQUAR   &10000&2081&2233&2055&2231&2167&2584&1973&2271&2043&2340&2043&2340\\ 
Perturbed  &10000&835&849&835&849&837&857&789&810&835&854&835&854\\ 
QUARTC   &10000&1&4&1&4&1&4&1&4&1&4&1&4\\ 
QF1 &10000&839&852&839&852&841&861&834&853&880&900&880&900\\ 
QF2  &10000&946&960&946&960&949&973&926&946&888&914&888&914\\ 
Raydan  &10000&7&8&7&8&7&8&7&8&7&8&7&8\\ 
SINCOS  &10000&11&17&11&17&11&17&11&17&11&17&11&17\\ 
TRIDIA &10000&2884&2898&2884&2898&3038&3079&3023&3054&2446&2479&2328&2355\\ 
\hline
 Average&&1063.51&1306.92&1780.31&17738.19&1040.96&1597.82&923.33&1327.82&0883.30&1282.16&900&1278.77\\
\hline
 
 \end{longtable} 
 \end{landscape}

 %===================bb1======================================
 
 \renewcommand{\arraystretch}{1.04} 
 \begin{landscape} 
 \small 
 \begin{longtable}{llllllllllllllll} 
 \multicolumn{5}{l} 
 {Table 4. Numerical results with the Barzilai-Borwein direction (BB1)}\\ 
 \cmidrule{1-14} 
 \multicolumn{1}{l}{}        &\multicolumn{1}{l}{} 
 &\multicolumn{1}{l}{NMLS-G}  &\multicolumn{1}{l}{}\hspace{10mm} 
 &\multicolumn{1}{l}{NMLS-H} &\multicolumn{1}{l}{}\hspace{10mm} 
 &\multicolumn{1}{l}{NMLS-N}  &\multicolumn{1}{l}{} \hspace{10mm} 
 &\multicolumn{1}{l}{NMLS-M}  &\multicolumn{1}{l}{}\hspace{10mm} 
 &\multicolumn{1}{l}{NMLS-1} &\multicolumn{1}{l}{}\hspace{10mm} 
 &\multicolumn{1}{l}{NMLS-2}   &\multicolumn{1}{l}{} \\ 
 \cmidrule(lr){3-4} \cmidrule(lr){5-6} \cmidrule(lr){7-8} \cmidrule(lr){9-10} \cmidrule(lr){11-12} \cmidrule(lr){13-14} 
 \multicolumn{1}{l}{Problem name}      &\multicolumn{1}{l}{Dimension} 

 &\multicolumn{1}{l}{$N_i$}   &\multicolumn{1}{l}{$N_f$} %&\multicolumn{1}{l}{$\mathrm{T}$}
 &\multicolumn{1}{l}{$N_i$}   &\multicolumn{1}{l}{$N_f$}
 &\multicolumn{1}{l}{$N_i$}   &\multicolumn{1}{l}{$N_f$} 
 &\multicolumn{1}{l}{$N_i$}   &\multicolumn{1}{l}{$N_f$} %&\multicolumn{1}{l}{$\mathrm{T}$} 
 &\multicolumn{1}{l}{$N_i$}   &\multicolumn{1}{l}{$N_f$}% &\multicolumn{1}{l}{$\mathrm{T}$} 
 &\multicolumn{1}{l}{$N_i$}   &\multicolumn{1}{l}{$N_f$}\\ %&\multicolumn{1}{l}{$\mathrm{T}$} \\ 
 \cmidrule{1-14} 
 \endfirsthead 
 \multicolumn{5}{l}% 
 {Table 4. Numerical results (\textit{continued})}\\[5pt] 
 \hline 
 \endhead 
 \hline 
 \endfoot 
 \endlastfoot 

Beale &2&46&52&42&50&46&53&36&52&38&54&37&54\\ 
Brown badly scaled  &2&Failed&Failed&8&109&8&109&8&109&8&109&8&109\\ 
Full Hess .  FH1 &2&49&56&43&46&50&75&40&57&42&52&41&53\\ 
 Full Hess .  FH2 &2&9&10&9&10&9&10&9&10&9&10&9&10\\ 
Gaussian  &3&4&14&4&14&4&14&4&14&4&14&4&14\\ 
Box three-dim  &3&36&62&36&64&36&67&105&219&882&1854&882&1854\\ 
Gulf r .  and d.  &3&1679&4232&2027&5724&4697&16171&43749&118335&36528&99453&2276&6860\\ 
Staircase 1   &4&13&14&13&14&13&14&13&14&13&14&13&14\\ 
Staircase 2  &4&13&14&13&14&13&14&13&14&13&14&13&14\\ 
Bro .  a .  Dennis  &4&52&56&57&68&52&56&49&62&52&67&49&62\\ 
wood  &4&242&326&138&154&255&520&1055&2291&222&450&805&1705\\ 
Biggs EXP66  &6&567&904&1025&1826&4003&9631&3559&7571&3517&7309&4058&8536\\ 
Penalty II	  &10&579&1271&674&1518&2756&7686&1792&4497&1623&3997&1165&2954\\ 
Variably Dim  &10&1&2&1&2&1&2&1&2&1&2&1&2\\ 
E .  Powell Sin &16&214&348&145&213&5097&11888&1820&3812&2711&5493&2575&5284\\ 
ARGLINB  &100&1075&41863&31&592&31&592&31&592&31&592&31&592\\ 
Diagonal 3  &100&102&119&107&109&106&147&105&150&113&166&103&145\\ 
E .  Rosenbrock  &100&96&208&44&79&63&131&36&82&33&67&36&82\\ 
Diagonal 2 &500&108&132&111&131&99&156&112&181&130&216&109&177\\ 
DIXON3DQ &1000&31289&52569&36487&66657&27809&65253&40299&86697&37983&80702&37983&80702\\ 
E .  Beal &1000&48&55&41&49&48&55&36&56&35&45&36&55\\ 
Fletcher  &1000&26718&44020&28940&51865&32507&74158&40440&85575&14475&29456&35721&74592\\ 
G .  Rosenbrock  &1000&24999&38314&24573&41178&23971&40964&25895&50162&25741&49242&26089&50267\\ 
Par .  Pert .  Quadratic &1000&203&270&179&180&305&590&325&592&411&711&433&731\\ 
 Hager &1000&61&66&57&59&54&65&55&60&57&59&57&59\\ 
 BDQRTIC &1000&82&91&86&87&80&94&81&106&87&98&69&74\\ 
BG2 &1000&44&58&99&159&65&88&196&433&57&104&209&464\\ 
Purt .  Trid. Quadratic &5000&2092&3344&854&1382&1633&3121&1685&3404&1765&3558&1727&3536\\ 
Al .  Pert .  Quadratic  &5000&1602&2522&1831&3170&1026&1866&1956&3999&2052&4097&1117&2192\\ 
Pert .  Quadratic diag   &5000&362&595&470&797&1439&4045&1285&3160&1231&2900&977&2330\\ 
E .  Trid .  1  &5000&32&37&32&33&36&47&35&50&31&33&34&42\\ 
NONSCOMP  &5000&62&63&62&63&54&56&54&57&63&65&54&57\\ 
LIARWHD  &5000&50&80&45&101&63&122&66&191&40&87&53&139\\ 
TRIDIA  &5000&8165&13522&11603&21006&10306&21041&14720&31687&17699&37371&14839&31400\\ 
DIXMAANA  &9000&9&10&9&10&9&10&9&10&9&10&9&10\\ 
DIXMAANB   &9000&8&9&8&9&8&9&8&9&8&9&8&9\\ 
DIXMAANC   &9000&10&11&10&11&10&11&10&11&10&11&10&11\\ 
DIXMAAND  &9000&13&14&13&14&13&14&13&14&13&14&13&14\\ 
DIXMAANE  &9000&1177&1807&987&1695&821&1831&1061&2176&1039&2081&1023&1986\\ 
DIXMAANF &9000&442&675&715&1191&666&1449&641&1335&895&1800&1051&2076\\ 
DIXMAANG &9000&1024&1602&835&1377&734&1612&461&903&1022&2031&732&1457\\ 
DIXMAANH &9000&658&1030&924&1528&776&1706&754&1522&966&1929&878&1726\\ 
DIXMAANI  &9000&6918&11583&6719&12096&6289&14861&10180&22178&7951&16803&9165&19375\\ 
DIXMAANJ   &9000&580&857&475&725&819&1706&661&1239&896&1685&575&1080\\ 
DIXMAANK    &9000&483&741&636&1020&618&1250&853&1703&596&1087&841&1604\\ 
DIXMAANL   &9000&379&552&437&644&607&1244&660&1209&629&1176&477&890\\ 
ARWHEAD &10000&3&4&3&4&3&4&3&4&3&4&3&4\\ 
BDEXP  &10000&18&19&18&19&18&19&18&19&18&19&18&19\\ 
Broyden Tridiagonal  &10000&78&83&71&72&83&98&78&100&72&78&75&96\\ 
COSIN   &10000&Failed&Failed&33&40&Failed&Failed&33&40&33&40&Failed&Failed\\ 
Diagonal 4 &10000&3&4&3&4&3&4&3&4&3&4&3&4\\ 
Diagonal 5   &10000&4&5&4&5&4&5&4&5&4&5&4&5\\ 
Ext .  Tridiagonal 2  &10000&31&36&31&36&30&36&31&36&31&36&31&36\\ 
Diagonal 7  &10000&7&8&7&8&7&8&7&8&7&8&7&8\\ 
Diagonal 8  &10000&6&8&6&8&6&8&6&8&6&8&6&8\\ 
DQDRTIC  &10000&27&28&27&28&27&28&27&28&27&28&27&28\\ 
ENGVAL1 &10000&28&29&28&29&28&29&28&29&28&29&28&29\\ 
EDENSCH &10000&27&28&27&28&27&28&27&28&27&28&27&28\\ 
Ext .  BD1 &10000&16&19&16&19&16&19&16&19&16&20&16&19\\ 
Ext .  DENSCHNB &10000&10&11&10&11&10&11&10&11&10&11&10&11\\ 
Ext .  DENSCHNF  &10000&12&13&12&13&12&13&12&13&12&13&12&13\\ 
Ext .  Himmelblau &10000&15&18&15&18&15&18&15&18&15&18&15&18\\ 
Ext .  Maratos  &10000&76&127&64&124&101&209&78&202&129&340&85&217\\ 
Ext .  Powell   &10000&360&674&231&397&14866&35239&6512&13509&6920&14447&4413&9095\\ 
Ext .  Penalty  &10000&123&2858&6&67&6&67&6&67&6&67&6&67\\ 
Ext .  PSC1   &10000&17&18&17&18&17&18&17&18&17&18&17&18\\ 
Ext .  QP1 &10000&13&14&13&14&13&14&13&14&10&14&13&14\\ 
Ext .   QP2  &10000&51&106&35&73&26&54&23&47&15&30&25&48\\ 
Ext .   TET  &10000&10&13&10&13&10&13&10&13&10&13&10&13\\ 
Ext . Wood  &10000&274&365&139&155&281&575&1157&2464&242&464&1163&2474\\ 
Ext .  White a .  Holst &10000&122&157&76&116&103&168&68&154&66&141&68&154\\ 
Ext .   Tridiagonal 1 &10000&32&37&35&37&36&47&36&51&31&33&34&42\\ 
  FH3  &10000&3&4&3&4&3&4&3&4&3&4&3&4\\ 
G .  PSC1  &10000&26&27&26&27&26&27&26&27&26&27&26&27\\ 
G .  Tridiagonal &10000&26&27&26&27&26&27&26&27&26&27&26&27\\ 
HIMMELBG   &10000&20&21&20&21&20&21&20&21&20&21&20&21\\ 
NONDQUAR   &10000&5734&9701&6121&11539&10738&25859&8646&18582&9890&21268&9310&19820\\ 
Perturbed  &10000&3332&5371&1996&3418&3028&6319&3610&7646&2894&5941&3096&6337\\ 
QUARTC   &10000&1&2&1&2&1&2&1&2&1&2&1&2\\ 
QF1 &10000&1914&3035&3367&6017&3070&6296&3901&8223&3339&6896&2887&5878\\ 
QF2  &10000&1853&3020&1857&3161&2424&4866&2269&4724&2398&4889&3277&6715\\ 
Raydan  &10000&1&2&1&2&1&2&1&2&1&2&1&2\\ 
SINCOS  &10000&17&18&17&18&17&18&17&18&17&18&17&18\\ 
TRIDIA &10000&21252&35581&28315&51638&18599&38133&24617&52700&25154&53422&30126&64210\\ 
\hline
Average&&2951.65&27514.52&1956.45&3512.77&2759.61&9558.57&2932.75&6493.89&2538.79&5542.02&2991.69&6378.95\\
\hline 
%BB1      
\end{longtable} 
\end{landscape} 
%WASTON,HARKAPER2,POWER, E.HILBERT,FLETCHER,cube and G.holest were removed here.

  %===================bb2======================================
 
 \renewcommand{\arraystretch}{0.98} 
 \begin{landscape} 
 \small 
 \begin{longtable}{llllllllllllllll} 
 \multicolumn{5}{l} 
 {Table 5. Numerical results with the Barzilai-Borwein direction (BB2)}\\ 
 \cmidrule{1-14} 
 \multicolumn{1}{l}{}        &\multicolumn{1}{l}{} 
 &\multicolumn{1}{l}{NMLS-G}  &\multicolumn{1}{l}{}\hspace{10mm} 
 &\multicolumn{1}{l}{NMLS-H} &\multicolumn{1}{l}{}\hspace{10mm} 
 &\multicolumn{1}{l}{NMLS-N}  &\multicolumn{1}{l}{} \hspace{10mm} 
 &\multicolumn{1}{l}{NMLS-M}  &\multicolumn{1}{l}{}\hspace{10mm} 
 &\multicolumn{1}{l}{NMLS-1} &\multicolumn{1}{l}{}\hspace{10mm} 
 &\multicolumn{1}{l}{NMLS-2}   &\multicolumn{1}{l}{} \\ 
 \cmidrule(lr){3-4} \cmidrule(lr){5-6} \cmidrule(lr){7-8} \cmidrule(lr){9-10} \cmidrule(lr){11-12} \cmidrule(lr){13-14} 
 \multicolumn{1}{l}{Problem name}      &\multicolumn{1}{l}{Dimension} 

 &\multicolumn{1}{l}{$N_i$}   &\multicolumn{1}{l}{$N_f$} %&\multicolumn{1}{l}{$\mathrm{T}$}
 &\multicolumn{1}{l}{$N_i$}   &\multicolumn{1}{l}{$N_f$}
 &\multicolumn{1}{l}{$N_i$}   &\multicolumn{1}{l}{$N_f$} 
 &\multicolumn{1}{l}{$N_i$}   &\multicolumn{1}{l}{$N_f$} %&\multicolumn{1}{l}{$\mathrm{T}$} 
 &\multicolumn{1}{l}{$N_i$}   &\multicolumn{1}{l}{$N_f$}% &\multicolumn{1}{l}{$\mathrm{T}$} 
 &\multicolumn{1}{l}{$N_i$}   &\multicolumn{1}{l}{$N_f$}\\ %&\multicolumn{1}{l}{$\mathrm{T}$} \\ 
 \cmidrule{1-14} 
 \endfirsthead 
 \multicolumn{5}{l}% 
 {Table 4. Numerical results (\textit{continued})}\\[5pt] 
 \hline 
 \endhead 
 \hline 
 \endfoot 
 \endlastfoot 
Beale &2&30&31&30&31&30&31&31&36&30&35&32&37\\ 
Brown badly scaled  &2&Failed&Failed&8&109&8&109&8&109&8&109&8&109\\ 
Full Hess .  FH1 &2&38&40&38&40&33&40&38&41&38&44&33&40\\ 
 Full Hess .  FH2 &2&9&10&9&10&9&10&9&10&9&10&9&10\\ 
Gaussian  &3&4&14&4&14&4&14&4&14&4&14&4&14\\ 
Box three-dim  &3&72&91&75&89&71&93&72&91&71&91&72&91\\ 
Gulf r .  and d.  &3&532&578&762&848&531&633&645&737&690&800&631&742\\ 
Staircase 1   &4&7&8&7&8&7&8&7&8&7&8&7&8\\ 
Staircase 2  &4&7&8&7&8&7&8&7&8&7&8&7&8\\ 
Bro .  a .  Dennis  &4&49&50&49&50&49&50&49&50&49&50&47&51\\ 
wood  &4&581&704&705&808&552&653&638&793&602&750&758&953\\ 
Biggs EXP66  &6&453&495&413&457&605&702&590&661&536&597&473&529\\ 
Penalty II	  &10&223&301&73&76&329&446&433&556&333&436&421&536\\ 
Variably Dim  &10&1&2&1&2&1&2&1&2&1&2&1&2\\ 
E .  Powell Sin &16&133&152&120&131&162&220&134&175&140&184&156&205\\ 
Watson  &31&4961&5148&7062&7312&5216&5606&6798&7100&7677&8102&6079&6472\\ 
HARKERP2  &50&27851&28183&Failed&Failed&15438&18678&32888&33460&34703&36001&23470&24294\\
ARGLINB  &100&1069&40515&30&591&30&591&30&591&30&591&30&591\\ 
Diagonal 3  &100&92&93&92&93&108&114&116&121&116&121&116&121\\ 
E .  Rosenbrock  &100&58&65&58&65&64&73&78&95&73&100&74&102\\ 
Diagonal 2 &500&112&115&103&105&134&145&92&101&115&120&104&120\\ 
DIXON3DQ &1000&8286&8363&9367&9488&10934&11188&6567&6654&6856&6958&6567&6654\\ 
E .  Beal &1000&28&29&28&29&28&29&28&32&28&32&30&36\\ 
Fletcher  &1000&4569&5626&8089&9784&4445&5550&11476&13223&11407&13141&11385&13112\\ 
G .  Rosenbrock  &1000&24950&28069&23700&26822&23814&26980&23673&26804&25645&28796&23797&26941\\ 
Par .  Pert .  Quadratic &1000&218&233&212&217&244&269&308&328&314&342&314&342\\ 
 Hager &1000&48&50&48&50&48&50&48&50&48&50&48&50\\ 
 BDQRTIC &1000&88&89&88&89&93&98&96&100&82&88&82&88\\ 
BG2 &1000&55&60&55&60&79&99&68&78&75&99&100&118\\ 
POWER  &1000&29416&29678&14753&14914&27012&27355&28046&28376&25356&25678&21596&21839\\ 
Purt .  Trid. Quadratic &5000&981&1009&825&852&932&1002&1021&1059&1318&1381&780&817\\ 
Al .  Pert .  Quadratic  &5000&840&866&1136&1173&1019&1095&954&1009&935&991&954&1009\\ 
Pert .  Quadratic diag   &5000&273&299&308&312&357&430&299&332&308&353&332&368\\ 
E .  Hiebert  &5000&3181&3185&3181&3185&3181&3185&3181&3185&3181&3185&3181&3185\\ 
Fletcher   &5000&23321&28385&22936&31043&Failed&Failed&17180&25290&18139&26278&17864&25970\\ 
E .  Trid .  1  &5000&31&32&31&32&31&32&31&32&31&32&31&32\\ 
NONSCOMP  &5000&47&48&47&48&47&48&47&48&47&48&47&48\\ 
LIARWHD  &5000&54&55&49&53&49&51&55&71&44&57&47&64\\ 
TRIDIA  &5000&4804&4913&4076&4162&3574&3685&4319&4409&5477&5634&3334&3430\\ 
DIXMAANA  &9000&9&10&9&10&9&10&9&10&9&10&9&10\\ 
DIXMAANB   &9000&8&9&8&9&8&9&8&9&8&9&8&9\\ 
DIXMAANC   &9000&10&11&10&11&10&11&10&11&10&11&10&11\\ 
DIXMAAND  &9000&13&14&13&14&13&14&13&14&13&14&13&14\\ 
DIXMAANE  &9000&537&551&622&647&483&537&853&883&764&792&809&848\\ 
DIXMAANF &9000&631&637&478&502&473&509&473&489&521&543&312&328\\ 
DIXMAANG &9000&527&543&374&382&589&638&595&612&337&361&337&361\\ 
DIXMAANH &9000&564&585&678&693&419&453&594&609&500&520&522&550\\ 
DIXMAANI  &9000&2608&2650&3197&3239&3815&3965&3450&3522&2898&2979&3152&3226\\ 
DIXMAANJ   &9000&531&546&605&622&634&667&394&404&394&415&406&417\\ 
DIXMAANK    &9000&364&374&416&428&410&438&395&410&369&389&510&527\\ 
DIXMAANL   &9000&365&377&435&452&396&424&379&399&324&348&400&420\\ 
ARWHEAD &10000&3&4&3&4&3&4&3&4&3&4&3&4\\ 
BDEXP  &10000&18&19&18&19&18&19&18&19&18&19&18&19\\ 
Broyden Tridiagonal  &10000&92&95&99&100&97&103&98&109&110&116&97&110\\ 
COSIN   &10000&Failed&Failed&35&44&Failed&Failed&40&50&35&44&40&65\\ 
Diagonal 4 &10000&3&4&3&4&3&4&3&4&3&4&3&4\\ 
Diagonal 5   &10000&4&5&4&5&4&5&4&5&4&5&4&5\\ 
Ext .  Tridiagonal 2  &10000&33&38&33&38&37&44&33&38&33&38&33&38\\ 
Diagonal 7  &10000&7&8&7&8&7&8&7&8&7&8&7&8\\ 
Diagonal 8  &10000&6&8&6&8&6&8&6&8&6&8&6&8\\ 
DQDRTIC  &10000&21&22&21&22&21&22&21&22&21&22&21&22\\ 
ENGVAL1 &10000&31&32&31&32&31&32&31&32&31&32&31&32\\ 
EDENSCH &10000&30&31&30&31&30&31&30&31&30&31&30&31\\ 
Ext .  BD1 &10000&15&18&15&18&15&18&15&18&14&18&15&18\\ 
Ext .  DENSCHNB &10000&10&11&10&11&10&11&10&11&10&11&10&11\\ 
Ext .  DENSCHNF  &10000&13&14&13&14&13&14&13&14&13&14&13&14\\ 
Ext .  Himmelblau &10000&15&17&15&17&15&17&15&17&15&17&15&17\\ 
Ext .  Maratos  &10000&119&144&115&141&119&144&136&182&135&185&136&182\\ 
Ext .  Powell   &10000&194&226&166&185&187&241&202&273&186&256&243&318\\ 
Ext .  Penalty  &10000&122&2857&5&66&5&66&5&66&5&66&5&66\\ 
Ext .  PSC1   &10000&14&15&14&15&14&15&14&15&14&15&14&15\\ 
Ext .  QP1 &10000&13&14&13&14&13&14&13&14&10&14&13&14\\ 
Ext .   QP2  &10000&51&106&35&73&26&54&23&47&15&30&23&47\\ 
Ext .   TET  &10000&9&12&9&12&9&12&9&12&9&12&9&12\\ 
Ext . Wood  &10000&607&726&695&804&551&658&691&866&662&845&825&1047\\ 
Ext .  White a .  Holst &10000&66&68&71&72&66&68&88&107&72&89&88&108\\ 
Ext .   Tridiagonal 1 &10000&33&35&34&35&33&37&33&35&33&37&33&37\\ 
  FH3  &10000&3&4&3&4&3&4&3&4&3&4&3&4\\ 
G .  PSC1  &10000&26&27&26&27&26&27&26&27&26&27&26&27\\ 
G .  Tridiagonal &10000&24&25&24&25&24&25&24&25&24&25&24&25\\ 
HIMMELBG   &10000&21&22&21&22&21&22&21&22&21&22&21&22\\ 
NONDQUAR   &10000&2856&2916&2976&3035&3025&3193&2986&3068&2888&2999&2996&3092\\ 
Perturbed  &10000&1599&1650&1298&1337&1612&1732&1480&1525&2074&2156&1462&1524\\ 
QUARTC   &10000&1&2&1&2&1&2&1&2&1&2&1&2\\ 
QF1 &10000&1043&1097&1909&1961&1636&1714&1259&1313&1854&1925&1282&1339\\ 
QF2  &10000&1525&1576&1738&1786&1482&1557&1388&1449&1326&1402&1260&1321\\ 
Raydan  &10000&1&2&1&2&1&2&1&2&1&2&1&2\\ 
SINCOS  &10000&14&15&14&15&14&15&14&15&14&15&14&15\\ 
TRIDIA &10000&9767&9898&5519&5625&5839&5976&5073&5157&5742&5893&10527&10712\\ 
\hline
Average&&2944.38& 24306.72 &  1915.18&  2100.42&  2489.36 & 2685.27 & 1809.83& 1998.06 & 1866.80& 2068.75& 1672.60 &1867.37\\%BB2
\hline
\end{longtable} 
\end{landscape}
%###########################################################################################################
%###########################################################################################################


\begin{thebibliography}{12}
\bibitem {Aho} Ahookhosh, M.:
Optimal subgradient algorithms with application to large-scale linear inverse problems,
submitted, \url{http://www.optimization-online.org/DB_HTML/2014/02/4234.html},(2014)

\bibitem {AhoA} Ahookhosh, M., Amini, K.: 
An efficient nonmonotone trust-region method for unconstrained optimization. 
Numerical Algorithms, \textbf{59}(4), 523--540 (2012)

\bibitem {AhoAB1} Ahookhosh, M., Amini, K., Bahrami, S.: 
Two derivative-free projection approaches for systems of large-scale nonlinear monotone equations
Numerical Algorithms, \textbf{64}, 21--42 (2013)

\bibitem {AhoAB} Ahookhosh, M., Amini, K., Bahrami, S.: 
A class of nonmonotone Armijo-type line search method for unconstrained optimization, 
Optimization, \textbf{61}(4), 387--404 (2012)

\bibitem {AhoAP} Ahookhosh, M., Amini, K., Peyghami, M.R.: 
A nonmonotone trust-region line search method for large-scale unconstrained optimization, 
Applied Mathematical Modelling, \textbf{36}(1) , 478--487 (2012)

\bibitem {AmiAN} Amini, K., Ahookhosh, M., Nosratipour, H.:
An inexact line search approach using modified nonmonotone strategy for unconstrained optimization,
Numerical Algorithms, \textbf{66}, 49--78 (2014) 

\bibitem {AndBM} Andretta, M., Birgin, E.G., Marti\'nez, J.M.:
Partial spectral projected gradient method with
active-set strategy for linearly constrained optimization,
Numerical Algorithms, \textbf{53}, 23--52 (2010)

\bibitem {And} Andrei, N.: 
An unconstrained optimization test functions collection, 
Advanced Modelling and Optimization, \textbf{10}(1), 147--161 (2008)

\bibitem {Arm} Armijo, L.: 
Minimization of functions having Lipschitz continuous first partial derivatives, 
Pacific Journal of Mathematics, \textbf{16}, 1--3 (1966)

\bibitem {BarB} Barzilai, J., Borwein, J.M.:
Two point step size gradient method,
IMA Journal of Numerical Analysis, \textbf{8}, 141--148 (1988)

\bibitem{BerB} Bertero, M., Boccacci, P.: 
Introduction to Inverse Problems in Imaging, Bristol, U.K.: IOP, (1998)

\bibitem {BirMR1} Birgin, E.G., Marti\'nez, J.M., Raydan, M.: 
Nonmonotone spectral projected gradient methods on convex sets, 
SIAM Journal on Optimization, \textbf{10}, 1196--1211 (2000)

\bibitem {BirMR2} Birgin, E.G., Marti\'nez, J.M., Raydan, M.: 
Inexact spectral projected gradient methodson convex sets, 
IMA Journal of Numerical Analysis, \textbf{23}, 539--559 (2003)

\bibitem {BonPTZ} Bonnans J.F., Panier E., Tits A., Zhou J.L.:
Avoiding the Maratos effect by means of a nonmonotone line search, II: Inequality constrained problems -- feasible iterates,
SIAM Journal on Numerical Analysis, \textbf{29}, 1187--1202 (1992)

\bibitem {BorH}  Borsdorf, R., Higham, N.J.: 
A preconditioned Newton algorithm for the nearest correlation matrix,
IMA Journal of Numerical Analysis, \textbf{94}, 94--107 (2010)

\bibitem {CarST} Cartis, C., Sampaio, Ph.R., Toint, Ph.L.: 
Worst-case evaluation complexity of non-monotone gradient-related algorithms for unconstrained optimization,
Optimization, (2014) 

\bibitem {ChaPLP} Chamberlain, R.M., Powell, M.J.D., Lemarechal, C., Pedersen, H.C.:
The watchdog technique for forcing convergence in algorithm for constrained optimization, 
Mathematical Programming Studies, \textbf{16}, 1--17 (1982)

\bibitem {ChaCCNP} Chambolle, A., Caselles, V., Cremers, D., Novaga, M., Pock, T.: 
An introduction to total variation for image analysis. In: Theoretical Foundations and Numerical Methods for
Sparse Recovery, vol. 9, pp. 263340. De Gruyter, Radon Series Comp. Appl. Math. (2010)

\bibitem {Dai} Dai, Y.H.: 
On the nonmonotone line search, 
Journal of Optimization Theory and Applications, \textbf{112}(2), 315--330 (2002)

\bibitem {DaiF} Dai, Y.H., Fletcher, R.: 
Projected barzilai-borwein methods for large-scale box-constrained quadratic programming,
Numerische Mathematik, \textbf{100}, 21--47 (2005)

\bibitem {DaiHSZ} Dai, Y.H., Hager, W.W., Schittkowski, K., Zhang, H.: 
The cyclic Barzilai-Borwein method for unconstrained optimization,
IMA Journal of Numerical Analysis, \textbf{26}, 604--627 (2006)

\bibitem {DaiYY} Dai, Y.H., Yuan, J.Y., Yuan, Y.Y.: 
Modified two-point stepsize gradient methods for unconstrained optimization, 
Computational Optimization and Applications, \textbf{22}, 103--109 (2002)

\bibitem {R7} Dolan, E., Mor\'{e}, J.J.: 
Benchmarking optimization software with performance profiles, 
Mathematical Programming, \textbf{91}, 201--213 (2002)

\bibitem {Fle} R. Fletcher, Practical methods of optimization, 2nd Edition, Wiley, 2000, New York

\bibitem {GriLL1} Grippo, L., Lampariello, F., Lucidi, S.: 
A nonmonotone line search technique for Newton's method, 
SIAM Journal on Numerical Analysis, \textbf{23}, 707--716 (1986)

\bibitem {GriLL2} Grippo, L., Lampariello, F., Lucidi, S.: 
A truncated Newton method with nonmonotone line search for unconstrained optimization, 
Journal of Optimization Theory and Applications, \textbf{60}(3), 401--419 (1989)

\bibitem {GriLL3} Grippo, L., Lampariello, F., Lucidi, S.: 
A class of nonmonotone stabilization methods in unconstrained optimization, 
Numerische Mathematik, \textbf{59}, 779--805 (1991)

\bibitem {HagMZ} Hager, W.W., Mair, B.A., Zhang, H.: 
An affine-scaling interior-point CBB method for box-constrained optimization,
Mathematical Programming, \textbf{119}, 1--32 (2009)

\bibitem {JudRRS} Judice, J., Raydan, M., Rosa, S., Santos, S.:
On the solution of the symmetric eigenvalue complementarity problem by the spectral projected gradient algorithm, 
Numerical Algorithms, \textbf{47}, 391--407 (2008)

\bibitem {LiuN} Liu, D.C., Nocedal, J.: 
On the limited memory BFGS method for large scale optimization, 
Mathematical Programming, \textbf{45}, 503--528 (1989)

\bibitem {LueRGH} Luengo, F., Raydan, M., Glunt, W., Hayden T.L.:
Preconditioned spectral gradient method,
Numerical Algorithms, \textbf{30}, 241--258 (2002)

\bibitem {MoLY} J. Mo, C. Liu, S. Yan, 
A nonmonotone trust region method based on nonincreasing technique of weighted average of the successive function value, Journal of Computational and Applied Mathematics 209 (2007) 97--108.

\bibitem {MorGH} Mor\'{e}, J.J., Garbow, B.S., Hillstrom, K.E.: 
Testing unconstrained optimization software,
ACM Transformations on Mathematical Software, \textbf{7}(1), 17--41 (1981)

\bibitem {MorPC} Morini, B., Porcelli, M., Chan, R.H.: 
A reduced Newton method for constrained linear least-squares problems, 
Journal of Computational and Applied Mathematics, \textbf{233}, 2200--2212 (2010)

\bibitem{NesBo} Nesterov, Y.:
Introductory lectures on convex optimization: A basic course,
Kluwer, Dordrecht (2004)

\bibitem{Neu} Neumaier, A.:
Solving ill-conditioned and singular linear systems: A tutorial on regularization, 
SIAM Review, \textbf{40}(3), 636--666 (1998)

\bibitem{NgCT} Ng, M., Chan, R., Tang, W.:
A fast algorithm for deblurring models with Neumann boundary conditions,
SIAM Journal on Scientific Computing, \textbf{21}, 851--866 (2000)

\bibitem {NocW} Nocedal, J., Wright, J.S.: 
Numerical Optimization, 
Second ed., Springer, New York, (2006)

\bibitem {Noc} Nocedal, J.: 
Updating quasi-Newton matrices with limited storage,
Mathematics of Computation, \textbf{35}, 773--782 (1980)

\bibitem{OrtR} Ortega, J.M., Rheinboldt, W.C.: Iterative solution of
nonlinear equations in several variables, Academic Press, New York
(1970)

\bibitem {PanT} Panier, E.R., Tits, A.L.: 
Avoiding the Maratos effect by means of a nonmonotone linesearch I, 
SIAM Journal on Numerical Analysis, \textbf{28}, 1183--1195 (1991)

\bibitem {Ray} Raydan, M.: 
The Barzilai and Borwein gradient method for the large scale unconstrained minimization problem,
SIAM Journal on Optimization, \textbf{7}, 26--33 (1997)

\bibitem {Toi} Toint, Ph.L.: An assessment of nonmonotone linesearch technique for unconstrained optimization,
SIAM Journal on Scientific Computing, \textbf{17}, 725--739 (1996)

\bibitem {Url1} \url{http://www.ece.northwestern.edu/~nocedal/software.html}

\bibitem {ZhaH} Zhang, H.C., Hager, W.W.: 
A nonmonotone line search technique and its application to unconstrained optimization, 
SIAM Journal on Optimization, \textbf{14}(4), 1043--1056 (2004)

\end{thebibliography}
\end{document}